\documentclass[10pt, a4paper]{article}
\title{Global stability of vacuum for the relativistic Vlasov-Maxwell-Boltzmann system}
\usepackage{stmaryrd}
\author{
Chuqi Cao \thanks{Department of Mathematics, The Chinese University of Hong Kong, P. R. China. \href{mailto:chuqicao@gmail.com}{\texttt{chuqicao@gmail.com}}}
    \and
   Xingyu Li \thanks{Università degli Studi di Trieste, Dipartimento di Matematica, Informatica e Geoscienze, Italy. \href{mailto:xingyu.li@units.it}{\texttt{xingyuli92@gmail.com}}}
}

\usepackage{mathrsfs}
\usepackage{geometry}
\usepackage{empheq}
\usepackage{times}
\usepackage{stmaryrd}
\usepackage{fourier}
\usepackage[T1]{fontenc}
\usepackage{amssymb, amsmath,  bm, mathrsfs}
\usepackage{amsfonts}

\usepackage[english]{babel}
\usepackage{amssymb,amsthm}
\usepackage{mathtools}
\usepackage{cases}
\DeclareMathAlphabet{\mathcal}{OMS}{cmsy}{m}{n}
\DeclareFontFamily{U}{mathc}{}
\DeclareFontShape{U}{mathc}{m}{it}
{<->x*[1.03] mathc10}{}
\DeclareMathAlphabet{\mathscr}{U}{mathc}{m}{it}

\DeclareMathAlphabet{\mathpzc}{OT1}{pzc}{m}{it}

\DeclareMathAlphabet{\mathscr}{U}{rsfs}{m}{n}
\DeclareFontFamily{U}{mathx}{}
\DeclareFontShape{U}{mathx}{m}{n}{<-> mathx10}{}
\DeclareSymbolFont{mathx}{U}{mathx}{m}{n}
\DeclareMathAccent{\widehat}{0}{mathx}{"70}
\DeclareMathAccent{\widecheck}{0}{mathx}{"71}

\usepackage{hyperref}
\usepackage[capitalise]{cleveref}
\usepackage{cite}
\allowdisplaybreaks[1]
\usepackage{xfrac}
\usepackage[utf8]{inputenc}
\usepackage[T1]{fontenc}
\usepackage{enumerate}
\usepackage{accents}
\usepackage{bbm}
\usepackage{thmtools}
\usepackage{extarrows}

\usepackage{todonotes}

\usepackage{tikz}
\usetikzlibrary{arrows, automata, backgrounds, shadows, patterns, calc, hobby, plotmarks, shapes}
\usetikzlibrary{shapes.misc}

\tikzset{cross/.style={cross out, draw=black, minimum size=2*(#1-\pgflinewidth), inner sep=0pt, outer sep=0pt},
cross/.default={1pt}}
\allowdisplaybreaks

\makeatletter
\@ifpackageloaded{stix}{%
}{%
  \DeclareFontEncoding{LS2}{}{\noaccents@}
  \DeclareFontSubstitution{LS2}{stix}{m}{n}
  \DeclareSymbolFont{stix@largesymbols}{LS2}{stixex}{m}{n}
  \SetSymbolFont{stix@largesymbols}{bold}{LS2}{stixex}{b}{n}
  \DeclareMathDelimiter{\lBrace}{\mathopen} {stix@largesymbols}{"E8}%
                                            {stix@largesymbols}{"0E}
  \DeclareMathDelimiter{\rBrace}{\mathclose}{stix@largesymbols}{"E9}%
                                            {stix@largesymbols}{"0F}
}
\makeatother

\DeclareSymbolFontAlphabet{\amsmathbb}{AMSb}%

\makeatletter

\newcommand{\Rmnum}[1]{\expandafter\@slowromancap\romannumeral #1@}
\makeatother

\usepackage{listings}
\usepackage{color}

\definecolor{dkgreen}{rgb}{0,0.6,0}
\definecolor{gray}{rgb}{0.5,0.5,0.5}
\definecolor{mauve}{rgb}{0.58,0,0.82}

\makeatletter
\def\maketag@@@#1{\hbox{\m@th\normalfont\normalsize#1}}
\makeatother

\newcommand{\dd}{\textnormal{d}}

\newcommand{\bkk}[1]{\langle #1 \rangle}

\newcommand{\pare}[1]{\left( #1 \right)}

\newcommand{\av}[1]{\left| #1 \right|}

\newcommand{\defeq}{\vcentcolon=}

\newcommand{\vect}[2]{  \begin{bmatrix} #1 \\ #2 \end{bmatrix}  }

\newcommand{\RN}[1]{%
  \textup{\uppercase\expandafter{\romannumeral#1}}%
}

\newcommand{\R}{\mathbb{R}}

\numberwithin{equation}{section}

\newcommand{\p}{\pare}

\theoremstyle{theorem}
\newtheorem{theorem}{Theorem}[section]

\newtheorem*{theorem*}{Theorem}

\newtheorem{lemma}[theorem]{Lemma}
\newtheorem{cor}[theorem]{Corollary}

\theoremstyle{definition}
\newtheorem{definition}[theorem]{Definition}

\newtheorem{remark}[theorem]{Remark}

\makeatletter
\@addtoreset{step}{subsection}
\makeatother

\makeatletter
\@addtoreset{substep}{step}
\makeatother

\makeatletter
\@addtoreset{proofpart}{subsection}
\makeatother

\makeatletter
\@addtoreset{proofsubpart}{part}
\makeatother

\numberwithin{equation}{section}

\usepackage{pdfcomment}

\begin{document}

\maketitle

\begin{abstract}
We consider the three-dimensional relativistic Vlasov–Maxwell–Boltzmann system, where the speed of light $c$ is an arbitrary constant no less than 1, and we establish global existence and nonlinear stability of the vacuum for small initial data, with bounds that are uniform in $c$. The analysis is based on the vector field method combined with the Glassey–Strauss decomposition of the electromagnetic field, and does not require any compact support assumption on the initial data. A key ingredient of the proof is the derivation of a chain rule for the relativistic Boltzmann collision operator that is compatible with the commutation properties of the vector fields. These tools allow us to control the coupled kinetic and electromagnetic equations and to obtain global stability near vacuum.
\end{abstract}

	{\small\tableofcontents}

	\allowdisplaybreaks

	\section{Introduction}
	The collective dynamics of collision-dominated plasmas—central to astrophysics, fusion research, and space physics—are described at the microscopic level by kinetic theory. In relativistic regimes, the appropriate model is the relativistic Vlasov–Maxwell–Boltzmann (RVMB) system, which couples three fundamental mechanisms: relativistic transport of charged particles, self-consistent electromagnetic fields governed by Maxwell’s equations with finite speed of light $c$, and nonlinear particle interactions described by the Boltzmann collision operator.

The RVMB system describes the time evolution of dilute ionized charged particles (e.g., electrons, ions) undergoing binary collisions and interacting via internally generated Lorentz forces, and is defined as follows:
\begin{equation}
\label{RVMB1}
\partial_tf+\hat v\cdot \nabla_xf+\pare{E+\frac 1c\hat v\times B}\cdot \nabla_vf=Q_c(f,f),\\
\end{equation}
\begin{equation}
\label{RVMB2}
\partial_t E = c\nabla \times B -J,\quad \nabla \cdot E =  J_0(f),  \quad \partial_t B = -c \nabla \times E,\quad \nabla \cdot B =  0, 
\end{equation}
with initial data
\[
f(0, x, v) = f_0(x, v) ,\quad E(0, x) =E_0(x),\quad B(0, x) =B_0(x),
\]
for $t \ge 0, x , v \in \R^3$ and speed of light $c \ge 1$. Here $f$ denotes the density distribution function of the particles, and $E,B$ are electric and the magnetic field respectively. Moreover,
\begin{equation}
\label{definition J 0 J i}
J_0(f ) (t, x) :=\int_{\R^3} f(t, x, v) \dd v ,\quad J_i(f) (t, x) :=\int_{\R^3} \hat{v}_i f(t, x, v) \dd v, \quad i=1,2,3 ,
\end{equation}
are the total charge density and the total current density, and $\hat{v} = \frac {c v } {v_0}$ denotes the relativistic speed of particles, and $v_0=\sqrt{c^2+v^2}$. The relativistic Boltzmann collision operator $Q_c$ writes
\begin{equation}
\label{Bol1}
Q_c(h, f)  = \int_{\R^3} \int_{\mathbb{S}^2}   B_c(v, u) \left(   h(u') f(v' ) - h(u) f(v) \right)\dd\omega \dd u :=Q_c^+(h, f)  -  Q_c^-(h, f) ,
\end{equation}
here 
\begin{itemize}
\item $g = \sqrt{2 (v_0u_0 - v \cdot u - c^2)}$ denotes the \textit{relative momentum}, and $s =2 (v_0u_0 - v \cdot u + c^2)$ denotes the square of energy in the \textit{center-of-momentum}. From direct calculations, we have $s =g^2+4c^2$.

\item $B_c(v, u) = v_\phi \sigma(g,\vartheta)$, where $v_\phi= \frac {c g \sqrt{s} }  {4 v_0u_0}$ denotes the M\"oller velocity, and $\sigma(g,\vartheta)\defeq |g|^{\gamma-1}\sigma_0(\vartheta)$ is the scattering kernel, where $0\le \sigma_0(\vartheta)\lesssim 1$, and the scattering angle $\vartheta$ is defined as
\begin{equation}
\label{eq:vartheta}
\cos\vartheta=\frac{-(v_0-u_0)(v_0'-u_0')+(v-u)(v'-u')}{g^2}.
\end{equation}
\item $v'$ and $u'$
 stand for the post-collisional momenta of particles which had the momenta $v$ and $u$ before the collision, respectively. Then momentum and the energy are conserved after each collision, which is
 \[
 v+u=v'+u',\quad v_0+u_0=v'_0+u'_0.
 \]
 In the center-of-momentum frame, denote $\zeta = \frac {u_0+v_0} {\sqrt{s} }$. Under the scattering angle $\omega$, we have
 \begin{align*}
\nonumber
 v' = \frac {v+u} 2 +\frac g 2 \left(\omega + (\zeta-1) (v+u ) \frac {(v+u) \cdot \omega} {|v+u|^2} \right),\quad
u' = \frac {v+u} 2 -\frac g 2 \left(\omega + (\zeta-1) (v+u ) \frac {(v+u) \cdot \omega} {|v+u|^2} \right).
\end{align*}
\end{itemize}

The RVMB system presents a significant challenge due to the strong coupling between hyperbolic Maxwell equations and the relativistic Boltzmann operator.
We address this by combining three methods:\begin{enumerate}\item The commutator vector-field method to capture relativistic symmetries and yield decay for energy estimates;\item The Glassey–Strauss decomposition of the Maxwell field and expose the precise field–particle coupling;\item \textit{The first chain rule for the relativistic Boltzmann operator}, a sharp identity that allows momentum derivatives (and Lorentz-boost commutators) to pass through the collision integral with controlled lower-order errors, removing the barrier to applying vector-field techniques at the kinetic level. \end{enumerate}
Together with refined weighted energy estimates, these methods yield bounds that are uniform in the speed of light, close the Maxwell–collision loop near vacuum, and establish a robust framework for analyzing relativistic kinetic problems, including classical limits and scattering.

From a physical perspective, the proof of vacuum stability near small initial data ensures that the system remains well-behaved over time, preventing unbounded growth or collapse. This result is crucial in plasma physics, as it provides the foundation for understanding the long-term evolution of plasmas and the interaction between particles and electromagnetic fields.




\subsection{Main result}
Our main result is stated as follows.
\begin{theorem}[Global stability of the relativistic Vlasov-Maxwell-Boltzmann system]
\label{thm:global}
 Let $\gamma\in (-2,0]$.   For any $c \ge 1$, $M>0$, there exists a constant $\epsilon_0 >0$ independent of $c$, such that for any $\epsilon_1\in (0,\epsilon_0)$, if the initial data $(f_0, E_0, B_0)$ satisfy
\begin{equation}
\label{initial data requirement}
\Vert (1+|x|)^{1000}(1+|v|)^{1000}\nabla^{20}_{x, v} f_0(x, v )  \Vert_{L^\infty_{x,v}} \le\epsilon_1,\quad \Vert  (1+|x|)^{1000}\nabla^{20}_xE_0(x) \Vert_{L^{\infty}_x}+\Vert  (1+|x|)^{1000}\nabla^{20}_xB_0(x) \Vert_{L^{\infty}_x}\le M,
\end{equation}
then the relativistic Vlasov-Maxwell-Boltzmann system \eqref{RVMB1}-\eqref{RVMB2} admits a global solution. Moreover, we have the following estimates in $t$ for distribution function
\begin{equation}
\label{eq:theorem1.1-1}
\av{\int_{\mathbb R^3}\bkk v^5\hat Z^\beta f(t,x,v) \dd v}\lesssim  \epsilon_1(1+t+|x|)^{-3}\quad\textnormal{for all}\quad |\beta|\le 18,
\end{equation}
 where $\hat Z$ is defined in \Cref{def:vectorfields}, and the following optimal in $t$ estimates  for electromagnetic fields:
\begin{equation}
\label{eq:theorem1.1-2}
\av {E(t,x)}\lesssim\frac M{\pare{1+t+\av x}\pare{1+\av{t-\frac{\av x}c}}},  \quad \av {B(t,x)}\lesssim\frac M{c\pare{1+t+\av x}\pare{1+\av{t-\frac{\av x}c}}}.
\end{equation}
\end{theorem}

As a significant byproduct of our work, we establish the \textit{first chain rule for the relativistic Boltzmann operator}. We believe this identity is of independent mathematical interest and will serve as a fundamental tool for relativistic kinetic theory. 


\begin{theorem}[Chain rule for the relativistic Boltzmann operator]
\label{thm:chainrule}
For any smooth functions $h, f$, for the relativistic Boltzmann operator  $Q_c$ we have
\begin{equation}
\label{eq:commutatorQ1}
\partial_{v_j } Q_c (h, f)  =\frac 1 {v_0} Q_c(h, (v_0\partial_{v_j} f))  + \frac 1 {v_0}    Q_c((v_0\partial_{v_j} h) , f)  - \frac {v_j} {v_0^2}   Q_c(h, f)  ,\quad v_0 = \sqrt{c^2+v^2},
\end{equation}
where $c$ is the speed of light and $j=1, 2, 3$. 
\end{theorem}


We give several remarks in order.

\begin{remark}
All statements and results in Theorem \ref{thm:global} remain valid for the relativistic Vlasov-Maxwell system in the absence of the relativistic Boltzmann operator (i.e., the collisionless case).
\end{remark}

\begin{remark}
In Theorem \ref{thm:global} we only require the initial distribution function $f_0$ is small, and does not require any smallness assumption on the Maxwell field.
\end{remark}
\begin{remark}
For the case $c=1$, the constraint on $\gamma$ in Theorem \ref{thm:global} can be relaxed to $\gamma \in (-2,1]$, which contains the famous hard sphere (also called hard ball) case, we refer to  \Cref{sub:gamma} below for more discussion on this part.
\end{remark}
\begin{remark}
The estimates established in Theorem \ref{thm:global} are uniform for all speeds of light $c \ge 1$. This suggests a possible path towards a non-relativistic (Newtonian) limit  to the classical Vlasov-Poisson-Boltzmann (VPB) system:
\begin{equation}
\partial_t f + v \cdot \nabla_x f + E \cdot \nabla_v f = Q(f, f),\quad \nabla \cdot E = \int_{\mathbb{R}^3} f(t, x, v) \, \dd v, \quad \nabla \times E = 0,
\end{equation}
as $c \to \infty$, where $Q(f, f)$ denotes the classical Boltzmann operator. However, the rigorous derivation of the Newtonian limit involves a delicate analysis of the difference between  $Q_c$ and  $Q$. Furthermore, it also requires a re-proof of the stability of vacuum for the VPB system using the vector field method. We postpone the treatment of this limit to a future work.

\end{remark}

\begin{remark}
 The optimal decay rates obtained here coincide with the sharp rates for the relativistic Vlasov-Maxwell system ($c=1$)  in \cite{bigorgne2020sharp, wang2022propagation, wei20213d}. Moreover, by formally taking the limit $c \to \infty$,  our decay estimates yield for any $x\in\mathbb R^3$:
\[
\av {E(t,x)}\lesssim\frac 1{\pare{1+t+\av x}\pare{1+\av{t-\frac{\av x}c}}}\to \frac 1 {(1+t)(1+t+|x|)}\le \frac 1 {(1+t)^2},  \quad \av {B(t,x)}\lesssim\frac 1{c\pare{1+t+\av x}\pare{1+\av{t-\frac{\av x}c}}}\to  0.
\]
This matches the optimal decay rate $E \lesssim (1+t)^{-2}$ for the Vlasov-Poisson system established in \cite{bardos1985global, smulevici2016small, wang2023decay}.
\end{remark}

\begin{remark}
By utilizing \eqref{eq:betterestimate1} in Lemma \ref{estimate for f dx}, one can show that$$\int_{\mathbb{R}^3} f(t, x, v) \, \dd x \to F(v) \quad \text{as } t \to \infty,$$for some limiting distribution $F(v)$. This indicates that the solution scatters to a linear trajectory. The investigation of modified scattering, which requires more refined estimates, will be addressed in a subsequent paper.
\end{remark}

\begin{remark}
For the sake of clarity, this paper focuses on the single-species RVMB system. However, our methodology can be extended to the two-species system describing ions ($f_+$) and electrons ($f_-$):
\begin{align*}\label{main1-00}
&\partial_t f_+ +    \hat{v}\cdot \nabla_x f_+  + \big(E+\frac 1 c \hat{v} \times  B   \big)\cdot\nabla_v f_+ =  Q_c (f_-,  f_+)+Q_c(f_+,f_+),
\\
&\partial_t f_- +    \hat{v}\cdot \nabla_x f_- -\big(E+\frac 1 c \hat{v} \times  B   \big)\cdot\nabla_v f_- =  Q_c (f_-,  f_-)+Q_c(f_+,f_-),
\\
&\partial_t E-  c\nabla \times B =- \int_{\mathbb R^3}   \Big(\hat{v} f_+  - \hat{v}  f_-   \Big)    \dd v, \quad \partial_t B+ c\nabla \times E=0,
\quad \nabla \cdot E=\int_{\mathbb R^3}  \left(  f_+ - f_-    \right) \dd v, \quad \nabla\cdot B=0.
\end{align*}
\end{remark}

\begin{remark}
Our results on the chain rule Theorem \ref{thm:chainrule} are obtained without assuming the product form $\sigma(g,\vartheta)=|g|^{\gamma-1}\sigma_0(\vartheta)$, allowing for more general collision kernels.
\end{remark}

\begin{remark}
\label{rmk:chainrule}
Since $v_0-c\to 0$ as $c\to\infty$, we obtain that for $c\to\infty$,
\begin{equation}
\label{eq:chainrulelimit}
\frac 1 {v_0} Q_c(h, (v_0\partial_{v_j} f))\to Q_c(h, \partial_{v_j} f),  \quad\frac 1{v_0}Q_c((v_0\partial_{v_j} h) , f) \to Q_c(\partial_{v_j} h) , f),\quad - \frac {v_j} {v_0^2}   Q_c(h, f)\to 0.
\end{equation}
Therefore, \Cref{thm:chainrule} is an relativistic extension of the chain rule for non relativistic Boltzmann operator
\begin{equation}
\label{eq:chainrulenormal}
\partial_{v_i} Q(h, f) = Q (\partial_{v_i} h, f) + Q(h, \partial_{v_i}  f),
\end{equation}
and \eqref{eq:chainrulelimit} indicates that  \Cref{thm:chainrule} is not only a fundamental tool for relativistic collision operator, but also plays a key role for proving convergence of classical limit towards non-relativistic collision operator.
\end{remark}

\subsection{Related Works}
	 In this subsection we will briefly review some former works related to us.
 \subsubsection{The relativistic Vlasov-Maxwell system}
 Kinetic equation analysis has advanced notably in non-relativistic and relativistic regimes, yet vacuum stability for the 3D non-relativistic Vlasov-Maxwell system remains a key challenge. By contrast, Glassey and Strauss \cite{glassey1986singularity, glassey1988global} pioneered global vacuum stability results for the relativistic Vlasov-Maxwell (RVM) system (see also \cite{klainerman2002new, schaeffer2004small}), using the method of characteristics and compactly supported initial data: $f_0(x, v) = 0$ for $ |x| \ge M$—this ensures finite spatial propagation  $f(t,x,v) = 0$ for $|x| \ge \beta(t)$, critical to their characteristic-based proofs. However, extending this framework to the RVMB system faces a core obstacle: the Boltzmann collision operator’s momentum-variable non-locality eliminates finite propagation speed, and transport effects spread this non-locality to spatial variables, invalidating the classical Glassey-Strauss approach.
 
 
To address these limitations, recent research has focused on achieving vacuum stability for the RVM system without the compact support requirement via sophisticated analytical frameworks:\begin{itemize}
  \item {Wang \cite{wang2022propagation}} utilized modified vector fields, integrated with Fourier analytic techniques.
  
 \item {Bigorgne \cite{bigorgne2020sharp}} used modified vector fields and Lie algebra-based analysis within Minkowski space.
 
\item {Wei–Yang \cite{wei20213d}} developed an alternative approach to the method of characteristics.

\item {Bigorgne \cite{bigorgne2025global}} recently streamlined the theory by demonstrating that modified vector fields are not strictly necessary, providing a more direct proof.\end{itemize}
 While our current work draws inspiration from \cite{bigorgne2025global}, there are distinct differences in our approach. Whereas \cite{bigorgne2025global} operates in Minkowski space $(\mathbb{R}^{1+3}, \eta)$ with coordinates $(x_0=t, x_1, x_2, x_3)$ and metric $\eta = (-1, 1, 1, 1)$, we separate the time and space variables to explicitly track the dependence on the speed of light $c$. We provide an independent proof and anticipate that our framework for the Vlasov-Maxwell component remains parallel to \cite{bigorgne2025global} in the specific case  $c=1$.
 
 Additionally, global well-posedness for the RVM system has been established under various symmetry assumptions \cite{glassey1990one, glassey1997two, luk2016strichartz, wang2022global}. For a broader discussion on continuation criteria, we refer to \cite{glassey1986singularity, glassey1989large, bouchut2003classical, sospedra2010classical, patel2018three, kunze2015yet}, as well as \cite{lin2007nonlinear, lin2008sharp} for stability and instability criteria.
\subsubsection{Boltzmann and Landau Equations}

Global well-posedness of kinetic equations near Maxwellian equilibrium is well-established. In the non-relativistic regime, Ukai \cite{ukai1974existence} first proved existence for the cutoff Boltzmann equation (see also \cite{guo2003classical, strain2008exponential}), while the non-cutoff case was resolved later in \cite{gressman2011global, alexandre2011global, alexandre2012boltzmann}. Guo \cite{guo2002landau} treated the Landau equation analogously, with further advances in \cite{carrapatoso2017landau, duan2021global}. For the Vlasov-Maxwell-Boltzmann system near Maxwellian, the result was first obtained by Guo in \cite{guo2003vlasov}, see also \cite{guo2002vlasov, duan2017vlasov, guo2012vlasov} for other results on coupled systems (e.g., Vlasov-Poisson/Maxwell-Boltzmann and Landau systems)  near Maxwellians; recent Hilbert expansion work appears in \cite{lei2024global}.

In the relativistic regime, collision operator structural complexity hinders analysis. Glassey and Strauss \cite{glassey1993asymptotic, glassey1995asymptotic} first proved global well-posedness for the angular-cutoff relativistic Boltzmann equation near the Jüttner distribution, with the non-cutoff case resolved in \cite{jang2022asymptotic}. For coupled relativistic systems, Guo and Strain \cite{guo2012momentum} established RVMB system Maxwellian stability; analogous results for relativistic Landau and Vlasov-Maxwell-Landau equations appear in \cite{lemou2000linearized, strain2004stability, yang2012global}.
\paragraph{Stability Near Vacuum}
Maxwellian stability literature is extensive, but vacuum global stability poses unique challenges due to the absence of a dissipative spectral gap. In the non-relativistic setting, vacuum stability for the cutoff Boltzmann equation was first proved in \cite{illner1984boltzmann}, the Landau equation by Luk \cite{luk2019stability}, and the non-cutoff Boltzmann equation by Chaturvedi \cite{chaturvedi2021stability}. Guo \cite{guo2001vlasov} established stability for the angular-cutoff Vlasov-Poisson-Boltzmann system.
\subsubsection{Vector field method}
The vector field method, pioneered by Klainerman \cite{klainerman1985uniform} and extended to Maxwell equations by Christodoulou and Klainerman \cite{christodoulou1990asymptotic}, is a cornerstone of nonlinear PDE global analysis. Smulevici \cite{smulevici2016small} first adapted it to the Vlasov-Poisson system, followed by Fajman, Joudioux and Smulevici \cite{fajman2017vector}, who built a robust framework for relativistic transport equations via relativistic velocity vector field commutation properties.

More recently, these methods have been incorporated into collision-term kinetic models (e.g., Boltzmann and Landau equations). Unlike dissipation analysis near a global Maxwellian, vacuum global stability requires sharp time-decaying full collision operator estimates. Building on Luk’s \cite{luk2019stability} vector field application to the Landau equation, Chaturvedi \cite{chaturvedi2021stability} proved analogous stability results for the non-cutoff Boltzmann equation; Chaturvedi, Luk and Nguyen \cite{chaturvedi2023vlasov} further advanced this research by treating the Vlasov-Poisson-Landau system in the weak collision regime.\\

Compared to previous works, the main contributions of this paper are threefold.

\begin{itemize}
\item While the stability of Maxwellian for the non-relativistic Vlasov-Maxwell-Boltzmann system was pioneered in \cite{guo2003vlasov}, to the best of our knowledge, global stability near vacuum for the VMB system remains an open problem, in both relativistic and non-relativistic regimes,   This work aims to fill that gap by providing a comprehensive stability analysis for the RVMB system.

\item We introduce a framework of \textit{vector fields that depend explicitly on the speed of light \(c\)}. Using these fields, we prove a global stability result for the RVMB system that is uniform for all \(c \ge 1\). This constitutes the first application of a \(c\)-dependent commutator framework to achieve such uniformity. We believe this method will be instrumental in rigorously proving the Newtonian limit (\(c \to \infty\)), thereby providing a mathematical bridge between relativistic and non-relativistic kinetic theory.

\item We present, to the best of our knowledge, the \textit{first general chain rule} for the relativistic Boltzmann collision operator—a result previously absent from the literature. This new analytical tool is expected to be critical for future studies of general relativistic kinetic equations.
\end{itemize}
In addition to these contributions, our proof is presented in a self-contained, accessible manner. We have strived for clarity and simplicity, avoiding overly specialized techniques. We believe this approach makes the proof not only easier to follow but also adaptable for related problems in kinetic theory.
%
%
%
%

\subsection{Key ideas of proof}
In this subsection, we illustrate the main ideas of the proof.




\subsubsection{Vector field method and $c$-dependent vector field}
We  introduce and utilize a family of $c$-dependent vector fields designed to track the reliance on the speed of light $c \ge 1$.

\paragraph{1: The $c$-Dependent Formulation.}
 For a function of $(t,x)$ and $c\ge 1$, we denote $Z$ as the $c$-dependent vector field combined with the following operators
\[
\frac 1c\partial_t, \quad\partial_{x_i}, \quad S=\frac 1c\pare{t\partial_t+\sum_{i}x_i\partial_{x_i}},\quad\Omega_i=t\partial_{x_i}+\frac 1{c^2}x_i\partial_t,\quad\Omega_{jk}=x_j\partial_{x_k}-x_k\partial_{x_j},\quad i, j, k=1,2,3,\quad j \ne k,
\]
and we denote $\hat Z$ as the complete lift of $Z$, which is the $c$-dependent vector field combined with the following operators
\[
\frac 1c\partial_t, \,\,\partial_{x_i}, \,\, S=\frac 1c\pare{t\partial_t+\sum_{i}x_i\partial_{x_i}},\,\,\hat\Omega_i=t\partial_{x_i}+\frac 1{c^2}x_i\partial_t+\frac{v_0}c\partial_{v_i},\,\,\hat\Omega_{jk}=x_j\partial_{x_k}-x_k\partial_{x_j}+v_j\partial_{v_k}-v_k\partial_{v_j},\,\, i,j,k=1,2,3,\,\, j\ne k.
\]
$Z$ and $\hat Z$ are used in \cite{smulevici2016small, fajman2017vector} for $c=1$. In this paper, $\hat Z$ is applied to $f(t, x, v)$, and $Z$ is applied to $E(t,x), B(t,x)$. In particular, for any function $G(t,x)$ that only depends on $t, x$ we have  $\hat Z(G(t,x))=Z(G(t,x))$ for any function $G(t,x)$. In the Newtonian limit
 $c \to \infty$, $\hat{\Omega}_i$ converges to the non-relativistic lift $t\partial_{x_i} + \partial_{v_i}$, consistent with the vector fields used in non-relativistic studies \cite{luk2019stability, chaturvedi2021stability}. To the best of our knowledge, this is the first usage of a $c$-dependent commutator framework that is uniform for all $c \ge 1$, effectively bridging the relativistic and non-relativistic regimes.


\paragraph{2: Advantages for Decay and Commutation.}
A primary advantage of this framework is the retrieval of \textit{enhanced decay rates}. By expressing derivatives in terms of the scaling and boost operators, we observe:
\[
\frac 1 c \partial_t  = \frac { t} {(t+ \frac {|x|} c )( t -  \frac {|x|} {c} ) }S  -  \sum_{i} \frac {x_i} {c(t+\frac {|x|} c)(t- \frac {|x|} {c} ) }   \Omega_{i},\quad \partial_{x_i} =  \frac {t} {(t+ \frac {|x|} c )( t -  \frac {|x|} {c} )  }\Omega_{i}  -   \frac {x_i} { c (t+ \frac {|x|} c )( t -  \frac {|x|} {c} )}  S 
- \sum_{j} \frac {x_j} {c^2  (t+ \frac {|x|} c )( t -  \frac {|x|} {c} ) } \Omega_{ij}.
\]
Thus an extra $ \frac 1 {|t- \frac {|x|} {c} |} $ decay rate is obtained for $\partial_{x_i} ,\frac 1 c \partial_t$, together with \eqref{eq:change2} could imply an extra $(1+t)^{-1}$ decay. Moreover, recall the form of $\hat{\Omega}_{i} $ we have 
\[\hat{\Omega}_{i} = t \partial_{x_i} +\frac 1 {c^2} x_i \partial_t+\frac {v_0}{c }\partial_{v_i},\quad \Rightarrow \quad
\partial_{v_i} = \frac c {v_0} \pare{ \hat{\Omega}_{i} -  t \partial_{x_i}    -  \frac 1 {c^2} x_i \partial_t}.
\] 
Therefore, the term $(E +\frac {v} {v_0}\times B) \cdot \nabla_v f $ can also be written in the form of $\hat{Z}$. Also the $c$-dependent vector field commutes with the transport operator $T_0 := \partial_t + \hat{v} \cdot \nabla_x $ and the Maxwell equation. 
First for $T_0$  we have that 
\[
[v_0T_0,\hat{Z}] =0, \quad \hat{Z} = \frac 1 c \partial_t, \partial_{x_i},\hat{\Omega} _{i},\hat{\Omega} _{i,j} ,\quad [v_0T_0,S] = \frac 1 c v_0 T_0,
\]
where  $[A, B] =AB-BA$ denotes the Poisson bracket of two operators. 
Next, suppose $(f , E, B)$ solves the RVMB system\eqref{RVMB1}, \eqref{RVMB2}, denote
\[
h_1^{i} (f) :  = \int_{\R^3}(c^2   \partial_{x_i}  + \frac {c v_i} {v_0}  \partial_t) f \dd v ,  \quad h_2^{i, j}(f) : =\int_{\R^3} \left( \frac {c^2 v_i} {v_0} \partial_{x_j} - \frac {c^2  v_j} {v_0}\partial_{x_i}  \right)  f  \dd v ,
\]
then we can prove  that
\[
 ( \partial_t^2 - c^2 \Delta_x )  (Z^\alpha [E, B] )  = \sum_{\beta \le \alpha } \sum_{i, j,k=1}^3  C_{\alpha, \beta, i}  h_1^{i} (\hat{Z}^\beta f )+   C_{\alpha, \beta, j, k}  h_2^{j, k} (\hat{Z}^\beta f ),
\]
where $C_{\alpha, \beta, i} $ and $C_{\alpha, \beta, j, k} $ are uniformly bounded constants with respect to $c$. Detailed statements and proofs of these commutation identities are provided in \Cref{commutator of Z alpha E B}.

\subsubsection{Null structure of electromagnetic field}

The six-component electromagnetic field $(E(t,x), B(t,x))$ has four linear combinations $\beta_i(E,B)$ with enhanced decay rates.  We introduce an orthonormal basis  $\{e_1', e_2', e_3'\}$ associated with spherical coordinates  $x = r(\sin \theta \cos \phi, \sin \theta \sin \phi, \cos \theta)$:
\begin{equation*}
e_1' = (\sin \theta \cos \phi, \sin \theta \sin \phi, \cos \theta) =\frac {x} {r} ,\quad e_2'= (\cos \theta \cos \phi, \cos \theta \sin \phi, -\sin \theta),\quad e_3'= (-\sin \phi, \cos \phi, 0).
\end{equation*}
We define the "good" components: $\beta_i(E,B ) := E \cdot e_1', B \cdot e_1',  E \cdot e_2' + B \cdot e_3', E \cdot e_3' - B \cdot e_2', i=1,2,3,4$, which satisfy the fast decay rate  $(1+t)^{-(2-\varepsilon)}$ for any $\varepsilon \in (0,1)$. Our core strategy bounds the Lorentz force term as
\begin{equation}
\label{upper bound Maxwell term}
\frac {c} {v_0 }  \left| E + \frac{v}{v_0} \times B \right| \lesssim \left( 1 - \frac{v}{v_0} \cdot \frac{x}{|x|} \right) (|E| + |B|) + \sum_{i=1}^4 |\beta_i(E,B)|.
\end{equation}
with detailed derivations in \Cref{sec:commuvectorfields}.

\subsubsection{Weight functions in relativistic regime}
A central component of our analysis is the choice of weight function and its role in controlling the distribution function $f$. We define
$\bkk\cdot=(1+\av{\cdot}^2)^\frac 12$.

\paragraph{1: The polynomial weight function $\bkk v^k $.} In this paper, we will work with the polynomial velocity weight function $\langle v \rangle^k$, rather than an exponential weight of the form $e^{\langle v \rangle^\alpha}$.  This choice guarantees that $\frac {|\nabla_v m| } {m} $ is bounded by $\frac 1 {\langle v \rangle}$, which preserves the null structure discussed earlier. Using \eqref{upper bound Maxwell term}, the action of electromagnetic field term on $m$ can be estimated as 
\[
\left  |E + \frac{v}{v_0} \times B  \cdot \frac {\nabla_v m } {m} \right | \lesssim   \frac {1} {\langle v \rangle }  \left| E+ \frac{v}{v_0} \times B \right|   \lesssim   \frac {c} {v_0 }  \left| E + \frac{v}{v_0} \times B \right| \lesssim \left( 1 - \frac{v}{v_0} \cdot \frac{x}{|x|} \right) (|E| + |B|) + \sum_{i=1}^4 |\beta_i(E,B)|.
\]
Moreover, we have established the following optimal upper bounds for the relativistic Boltzmann operator in polynomially weighted spaces:
\[
\Vert \langle v \rangle^k Q_c(h, f) \Vert_{L^1_v} \lesssim \Vert\langle v \rangle^k h \Vert_{L^p_v}\Vert\langle v \rangle^k f \Vert_{L^1_v},\quad \Vert \langle v \rangle^k Q_c(h, f) \Vert_{L^\infty_v} \lesssim \Vert\langle v \rangle^k h \Vert_{L^p_v}\Vert\langle v \rangle^k f \Vert_{L^\infty_v}
\]
for any $p >\frac{3} {3+\gamma}$, where the exponent $\frac{3} {3+\gamma}$  is sharp. This optimal $L^p$ estimate, together with \eqref{eq:change1} below,  enables us to obtain the optimal decay rate $|Q_c(h,f)| \lesssim (1+t)^{-\lambda}$ for any $\lambda <  (3+\gamma) $.  

\paragraph{2: The weight function $\bkk {x-t\hat v}^k $.}
Another weight function of particular interest is $\bkk {x-t\hat v}$, which arises naturally in transport-dominated equations because it commutes with the free transport operator $\partial_t+\hat v\cdot\nabla_xf$.  To see this, define $h(t,x,v)=f(t,x+\hat vt,v)$. Then \begin{equation}
\label{eq:change}
\bkk xh(t,x,v)=(\langle x-\hat v t\rangle f  ) (t,x+t \hat{v} ,v).
\end{equation}
which shows that the weight $\bkk {x-t\hat v}$ corresponds, under this change of variables, simply to the spatial weight $\langle x \rangle$ applied to the translated profile.

The utility of this weight function is twofold. First, it allows us to recover sharp decay in the spatial density. As shown in \Cref{L 1 L infty estimate on x t v}, that  the $L^1_v$ norm of $f$ satisfies:
\begin{equation}
\label{eq:change1}
\Vert f\Vert_{L^1_v}\lesssim (1+t)^{-3}\Vert \bkk v^5 \bkk{x-\hat vt}^4f\Vert_{L^\infty_v}.
\end{equation}
Second, it provides a mechanism to relate the weights to the physical trajectories of the system. Specifically, from  \Cref{lem:extraesti} we have the following inequality:
\begin{equation}
\label{eq:change2} \quad 1+t+|x|\lesssim \pare{1+\av{t-\frac{|x|}c}} \bkk v^4 \bkk{x-\hat vt}^2. 
\end{equation}
This inequality is crucial for transforming the electromagnetic field estimates in \eqref{eq:theorem1.1-2} into integrable terms of order $(1+t)^{-2}$, at the expense of additional weight functions.

Next, to control the Boltzmann operator within our weighted framework, we require a robust upper bound for $\langle x - t\hat{v} \rangle Q_c(h, f)$. This estimate involves a significant technical hurdle. We begin by decomposing $Q_c$ into its gain and loss components, $Q_c^+$ and $Q_c^-$, respectively. The loss term $Q_c^-$ has the form $h(u)f(v)$, enabling straightforward bounding of $\langle x - t\hat{v} \rangle Q_c^-(h, f)$
 via the product of  $|h(u)|$ and $\langle x - t\hat{v} \rangle |f(v)|$. By contrast, the gain term $Q_c^+$, defined by post-collisional velocities $f(v')h(u')$
, is more difficult to handle; ideally, one would expect a sub-additive property for the weight, such as: 
\begin{equation}
\label{eq:introweight1}\langle x-t \hat{v} \rangle  \lesssim   \langle x-t\hat{v'}\rangle    +  \langle  x-t\hat{u'} \rangle.
\end{equation}
Such thing is true for the classical Boltzmann operator, however,  In the relativistic regime, however, \eqref{eq:introweight1} does not generally hold. Instead, we establish a weaker variant (see Lemma \ref{lem:estimateweight}) that incorporates a polynomial loss in the momentum variables:\begin{equation}\label{eq:introweight4}\langle x-t \hat{v} \rangle  \lesssim    \min\{   \langle v' \rangle^2    ,  \langle u' \rangle^2 \}   (\langle x-t\hat{v'}\rangle    +  \langle  x-t\hat{u'} \rangle).\end{equation}Because \eqref{eq:introweight4} is weaker than the standard estimate, it cannot be used directly to close the energy arguments. To circumvent this, we leverage the property $\langle v \rangle \lesssim \langle v' \rangle + \langle u' \rangle$ and define a composite weight function:$$\mathpzc n(v) \defeq \langle x-t\hat{v} \rangle^{k}\langle v \rangle^{4k+50} + \langle x-t\hat{v} \rangle^{k+10}\langle v \rangle^{2k+20},$$for a sufficiently large $k \in \mathbb{N}$. We demonstrate that this refined weight allows
satisfies the desired bound $\mathpzc n(v) \lesssim \mathpzc n(v') + \mathpzc n(u')$, thereby allowing
 us to close the argument by estimating $\left| \mathpzc n(v) Q_c(h, f) \right|$ for smooth functions $f$ and $h$. Detailed derivations are provided in \Cref{L infty estimates for Q f f z version}.

\subsubsection{The Distribution Function Estimate} In our framework, we classify terms of the form $\langle v \rangle^{N_1} \langle x - \hat{v}t \rangle^{N_2} f$ as higher-order terms. A primary objective of this work is to prove that these quantities—even when subjected to the commutation vector fields $\hat{Z}^\alpha$—grow no faster than logarithmically in time. For the precise formulation of these bounds, we refer the reader to Proposition \ref{main estimate for f}.

To establish the logarithmic bound on these higher-order terms, we recall that the distribution function $f$ satisfies \eqref{RVMB1}. Let $T_F$ denotes the Vlasov operator$
T_F = \partial_t +\hat{v} \cdot \nabla_x +\pare{ E + \frac{v}{v_0} \times B} \cdot \nabla_v,$
then we have $T_F(f)=Q_c(f,f)$. When commuting $T_F$ with the vector fields $\hat{Z}^\alpha$, we obtain:
\[
T_F(\hat Z^\alpha f)=\textnormal{Boltzmann term}  +\textnormal{commutator terms of}\,\, f,E,B.
\]
The Boltzmann terms, as talk before, decay at a rate of $(1+t)^{-\lambda}$ for $\lambda > 1$, making them integrable in time. Another main work is commutator terms involving $f, E, B$. A naive estimate gives:
\[
\textnormal{commutator terms of}\,\, f,E,B\sim  t^{-1}|\partial_{t,x}f|.
\]
The argument cannot be closed directly as $t^{-1}$ is non-integrable, requiring refined analysis of the vector fields $\hat Z$.

Decomposing the vector fields into "good terms" (e.g., $\frac{1}{t}\partial_t, \partial_x$) and "bad terms" (e.g., the rotations $\hat{\Omega}_i, \hat{\Omega}_{jk}$), we use inequality \eqref{eq:change2} and the electromagnetic field’s null structure to derive the improved estimates below:
\begin{itemize}
\item Commutators of good terms: $\lesssim \frac{1}{(1+t)\log^2(3+t)}$,
\item Commutators of bad terms: $\lesssim \frac{1 - \frac{\hat{v} \cdot x}{c|x|}}{\left( 1 + \left| t - \frac{|x|}{c} \right| \right) \log^2\left( 3 + \left| t - \frac{|x|}{c} \right| \right)}.$
\end{itemize}
By applying  Lemma \ref{T F g implies g formula}, we prove that these commutator terms are integrable in time. This ensures that the higher-order terms grow at most logarithmically, which, combined with \eqref{eq:change1}, allows us to conclude that $\left| \int_{\mathbb{R}^3} \hat{Z}^\alpha f \, dv \right| \lesssim t^{-3}$. We refer to \cite{bigorgne2020sharp} Section 2.8.2 for more talks on this part.  

\subsubsection{Chain rule for relativistic Boltzmann operator}

A central challenge in this analysis stems from a fundamental structural difference between the relativistic and classical theories. In the classical setting, it is well known that the classical (non-relativistic) Boltzmann operator $Q(f, f)$ satisfies a Leibniz-type rule with respect to velocity derivatives:
$$\partial_{v_i} Q(f, f) = Q (\partial_{v_i} f, f) + Q(f, \partial_{v_i}  f).$$
This property is a direct consequence of the collision kernel's structure. For the kernel of  classical Boltzmann operator, which is  of the form $|v-v_*|^n$, the following translational symmetry holds: $\partial_{v_i} ( |v - v_*|^n ) = - \partial_{v_{i,*}} ( |v - v_*|^n )$. This symmetry is essential for distributing velocity derivatives across the arguments of the bilinear operator via integration by parts.

However, this property does not generalize to the relativistic collision kernel. The lack of such symmetry prevents the derivative from commuting with the operator in a straightforward manner. This failure of the standard chain rule has historically restricted the study of the relativistic Boltzmann equation, with most existing literature focusing on the $H^k_x L^2_v$ or $L^2 \cap L^\infty$ frameworks,  making \Cref{thm:chainrule} necessary.

In this paper, by utilizing the relativistic Carleman representation (Lemma \ref{lem:Carleman}) established in \cite{jang2022asymptotic}, we overcome this obstruction and obtain the chain rule \Cref{thm:chainrule}.  Moreover, for any  smooth functions $h, f$, we also proved the following identity
\begin{equation}
\label{eq:CarlemanH2}(v_j \partial_{v_i}  - v_i \partial_{v_j})    Q_c(h, f) = Q_c (h, (v_j \partial_{v_i}  - v_i \partial_{v_j})  f  ) +Q_c ((v_j \partial_{v_i}  - v_i \partial_{v_j})   h,  f  ).
\end{equation}
The two chain rules imply that the relativistic Boltzmann operator is compatible with the vector field method, which strengthen the validity of our proof. Specifically, from \Cref{thm:chainrule} and \eqref{eq:CarlemanH2}, the following commutation identities hold
\begin{equation}
\label{eq:CarlemanH3}
\hat{Z} Q_c(f, g) = Q_c(\hat{Z}f, g ) + Q_c(f, \hat{Z} g ), \quad \hat{Z} =\partial_t, \partial_{x_i},\hat{\Omega}_{ij}, S,\quad
\hat{\Omega}_{i} Q_c(f, g) = Q_c( \hat{\Omega}_{i} f, g ) + Q_c(f, \hat{\Omega}_{i} g ) - \frac{v_i}{v_0} Q_c(f, g).
 \end{equation}
Since the factor $\frac{v_i}{v_0}$ and its velocity derivatives are uniformly bounded from above, these identities facilitate the estimation of higher-order derivatives.

\subsubsection{Rationale for the Collision Index $\gamma\in (-2,0]$}
\label{sub:gamma}
In this subsection, we will talk on the requirement for the  range $\gamma \in (-2, 0]$ in the proof  for the relativistic Boltzmann equation's global-in-time theory.

\begin{itemize}
\item
Upper bound ($\gamma \le 0$): This constraint ensures control over the collision integral. For $\gamma > 0$, the kernel $|u-v|^\gamma$ grows unbounded as $|u-v| \to \infty$, preventing the integral from being bounded by standard $L^p_v$ norms. While techniques exist in isolated Boltzmann theory to manage this growth, these methods produce extra unbounded terms when coupled with the RVM system. Thus, $\gamma \le 0$ is required to tame the collision term's growth.

It is important to note, however, that this restriction can be relaxed under specific scaling. If we set the speed of light $c = 1$, the constraint may be relaxed to $\gamma \le 1$. This is possible because the term $v_\phi$ includes the factor $\frac{c}{v_0}$, which behaves asymptotically like $\langle v \rangle^{-1}$ when $c = 1$, providing the necessary decay to offset the kernel's growth.

\item Lower bound ($\gamma > -2$): This restriction is dictated by the time-integrability of the nonlinear decay. As talked before, the nonlinear Boltzmann term satisfies $|Q_c(h,f)| \lesssim (1+t)^{-\lambda}$ for any $\lambda <  (3+\gamma) $. 
Global existence requires this rate to be integrable, implying $3+\gamma > 1$, or equivalently, $\gamma > -2$. This restriction coincides with the restriction for the global stability of Landau and Vlasov-Poisson-Boltzmann equation near vacuum, see \cite{guo2001vlasov, luk2019stability}.

\end{itemize}
Together, these conditions restrict our analysis to moderately soft potentials, ensuring the nonlinear contributions remain finite for all $t \in [0, \infty)$. 

\subsection{Structure of the paper}
In \Cref{sec2}, we introduce preliminary estimates and the algebraic properties of the commutator vector fields. In \Cref{sec4}, we define bootstrap assumptions and improve the estimate of distribution function $f$. In \Cref{sec5}, we improve the estimate of electromagnetic field functions $E,B$ from Glassey-Strauss representation formula for the relativistic Vlasov-Maxwell system. \Cref{sec3} establishes all the estimates for the relativistic Boltzmann operator and proves the associated chain rule, for readers  only interested in the relativistic Boltzmann part, we refer to Section \ref{sec3}.  For the proofs of the esimates about the weight function and some other technical lemmas, we refer to  \Cref{appendix}.  \\

{\bf Acknowledgement} The authors would thank Prof. Xuecheng Wang for introducing this problem and Prof. Renjun Duan for fruitful talks on the relativistic Boltzmann operator. CC would thank to The Chinese University of Hong Kong and The Hong Kong Polytechnic University for the support. XL is supported by PRIN 2022 ”Turbulent effects vs Stability in Equations from Oceanography” (TESEO), project number: 2022HSSYPN.

\section{Preliminaries }
\label{sec2}
We list some estimates about the weight functon required for our analysis and derive the algebraic properties of the commutator vector fields. For the proofs of the esimates about the weight function and some other technical lemmas, we refer to \Cref{appendix}.
\subsection{Notations and definitions}
Let us first list some notations and definitions that will be used in the following sections of the paper.
\begin{definition}
\label{def:vectorfields}
We denote that 
$\bkk\cdot=(1+\av{\cdot}^2)^\frac 12, \kappa(v) = 1- \frac {v\cdot x} {v_0   |x|} = 1- \frac {\hat{v}\cdot x} {c   r}$, where $r:=|x|$ and we will not distinguish $r$ and $|x|$ in the following paper. 
Denote the transport operators
\[
T_0 = \partial_t  +\hat{v} \cdot \nabla_x  ,\quad T_F = \partial_t +\hat{v} \cdot \nabla_x +\pare{ E + \frac{v}{v_0} \times B} \cdot \nabla_v.
\]
we define, on the domain $\{ y \in \R^3  |  |y|<c \}$, the operator $\check{\; \; }$ as 
\begin{equation}
\label{eq:opeatorcheck}
y \to\check{y} = \frac{y}{\sqrt{1-\frac {|y|^2}{c^2 } }},  \quad \forall  |y| <c,  \quad \hat{\check{y}}=y, \quad \check{\hat{v}}=v,\quad \forall \quad v \in \R^3_v.
\end{equation}
For any $N_1,N_2\in\mathbb N$, we define the \textit{weight function}
\begin{equation}
\label{eq:weight}
\mathpzc W^{N_1}_{N_2}(t,x,v,c)\defeq \langle v\rangle^{N_1} \langle x-t\hat v\rangle^{N_2},
\end{equation}
and we will write $\mathpzc W^{N_1}_{N_2}$ for short.\\
We denote  
$e_1 = (1, 0, 0), e_2=(0, 1, 0), e_3 = (0, 0, 1)$.
For $x=(x_1,x_2,x_3)$, under the sphere coordinates $(r,\theta,\phi)$, we have  $x = r(\sin \theta \cos \phi, \sin \theta \sin \phi, \cos \theta)$,
and we will denote that 
\begin{equation}
\label{eq:e1'e2'e3'}
e_1' := (\sin \theta \cos \phi, \sin \theta \sin \phi, \cos \theta) =\frac {x} {r} ,\quad e_2' := (\cos \theta \cos \phi, \cos \theta \sin \phi, -\sin \theta),\quad e_3':= (-\sin \phi, \cos \phi, 0).
\end{equation}
We can see that $e_1' , e_2' , e_3' $ also forms a new orthogonal base.
From direct calculations,
\begin{equation}
\label{eq:e1'e2'e3'1}
\partial_r =  \sum_{i=1}^3 \frac {x_i} r \partial_{x_i} = \frac 1 r  x \cdot \nabla = e_1' \cdot \nabla.
\end{equation}
For $i , j, k = 1,2,3 ,j \neq k$, let $\mathbb K$ be the set composed by the vector fields  
\[
\frac 1 c \partial_t ,\quad \partial
_{x_i} ,\quad \Omega_{jk} = x_j\partial_{x_k}  - x_k \partial_{x_j} ,\quad \Omega_{ i} = t \partial_{x_i} +\frac 1 {c^2} x_i \partial_t,\quad S =\frac 1 c \pare{ t \partial_t + \sum_{i}  x_i \partial_{x_i} } ,
\]
and  let $\tilde{\mathbb K}$ be the set composed by the vector fields  
\[
\frac 1 c\partial_t ,\quad \partial
_{x_i} ,\quad \hat{\Omega}_{jk} = x_j\partial_{x_k}  - x_k \partial_{x_j} +v_j\partial_{v_k}  - v_k \partial_{v_j},\quad \hat{\Omega}_{ i} = t \partial_{x_i}    +  \frac 1 {c^2} x_i \partial_t+\frac {v_0}{c }\partial_{v_i},\quad S =\frac 1 c \pare{ t \partial_t + \sum_{i}  x_i \partial_{x_i} }.
\]
These vector fields have good commutation properties with the linear transport operator $T_0=\partial_t+\hat v\cdot\nabla_x$ (see \Cref{Representation for first T 0} for more details). In order to consider higher order derivatives, we introduce an ordering on $\mathbb K=\{Z^i, 1\le i\le 11\}$ and $\tilde{\mathbb K}=\{\hat Z^i, 1\le i\le 11\}$.For a multi-index $\beta$, we denote $H(\beta) (T(\beta))$ as the number of homogeneous vector fields, i.e. $S,\hat\Omega_i, \hat\Omega_{ij}$ (respectively translations, i.e. $\frac 1c\partial_t, \partial_{x_i}$) composing $\hat Z^\beta$. Sometimes we will denote $H(\hat{Z}^\alpha) : = H(\alpha)$.\\
For any vector functions $F(t, x), G(t, x)\in\mathbb R^3$, denote \textit{null decomposition}
\[
\rho ([F, G]) =F \cdot e_1' \quad \sigma ([F, G]) =G \cdot e_1' ,\quad \alpha_1 ([F, G]) = F \cdot e_2' + G \cdot e_3'  ,\quad  \alpha_2([F, G]) = F \cdot e_2'- G \cdot e_3',
\]
and $\mathpzc D_Z(F, G)(x)\in\mathbb R^6$ as
\[
 \mathpzc D_Z(F, G)=(\mathpzc D^{(1)}_Z(F, G), \mathpzc D^{(2)}_Z(F, G))  : =\begin{cases} (ZF, ZG),\quad Z = \frac 1 c \partial_t, \partial_{x_i} ,S\\
\pare{ ZF + \frac 1 c (0,   -G_3, G_2), ZG+ \frac 1 c(0,  F_3, -F_2)},\quad Z=\Omega_{1} \\
\pare{ZF+ (F_2 , - F_1 , 0), ZG + (G_2 , -G_1, 0 )}, \quad Z=\Omega_{12} ,
 \end{cases}
\]
and $Z=\Omega_{2},\Omega_{3},\Omega_{13}, \Omega_{23}$ can be defined similarly.\\

\end{definition}
\begin{remark}
In Minkowski spacetime $(\mathbb R^{1+3},\eta)$, the vector fields in $\mathbb K, \tilde{\mathbb K}$ are categorized as follows:

\begin{itemize}
\item Lorentz Boosts ($\Omega_{i}$) and Rotations ($\Omega_{jk}$): These, along with translations $\frac 1t\partial_t$ and $\partial_{x_i}$, form the set $\mathbb K \setminus \{S\}$ and are Killing vector fields. They generate isometries that preserve the spacetime metric $\eta$.
\item Scaling ($S$): Also called dilatation, this is a conformal Killing vector field. It preserves the light-cone structure but alters the metric scale.
\item Complete Lifts ($\hat{\Omega}_{i}, \hat{\Omega}_{jk}$): These represent the extension of the Lorentz generators to the tangent bundle, allowing the symmetries to act on both positions and velocities. \\

Moreover, $ \mathpzc D_Z(E(t,x), B(t,x))$ is equivalent to the Lie derivative of the Maxwell field with respect to $Z$.
\end{itemize}
\end{remark}
\begin{remark}
Remind that $G(t,x)$ do not depend on $v$, so $Z^iG=\hat Z^iG (1\le i\le 11)$.
\end{remark}


\subsection{Estimates related to the weight function}
In this subsection, we list some estimates related to the weight function $\mathpzc W^{N_1}_{N_2}$. We refer to \Cref{app:proof} for the proofs.
\begin{lemma}\label{L 1 L infty estimate on x t v}
For any smooth function $f$, for any $x \in \R^3, t \ge 0, c \ge 1, p\in [1,\infty)$,  we have that
\begin{align*}
\Vert f \Vert_{L^p_v}   \lesssim (1+t)^{-\frac 3p}\Vert \langle v \rangle^{5}  \langle  x -t \hat{v} \rangle^{4 }  f\Vert_{L^\infty_v} , 
\end{align*}
where the constant is independent of $x, t, c$.

\end{lemma}

\begin{lemma}\label{basic lemma for L p estimate Hardy}
For any smooth function $f$, $\gamma \in (-3, 0)$ fixed, and any $ p>1$ that satisfies $ p >\frac {3 } {3+\gamma}$ i.e. $\frac { p} { p-1}  \frac {-\gamma } {3} <  1$, for any $\lambda <3+\gamma$ and for any $t \ge 0, x \in \R^3$ we have
\[
\int_{\R^3} |u-v|^\gamma f(u) \dd u \lesssim \Vert f \Vert_{L^1_v} + \Vert f \Vert_{L^{ p}_v} \lesssim (1+t)^{-\lambda}  \Vert \langle v \rangle^{5}  \langle  x -t \hat{v} \rangle^{4 }  f\Vert_{L^\infty_v}   .
\]
\end{lemma}

\begin{lemma}
\label{lem:extraesti}
For any $t \ge 0, x, v \in \R^3, c \ge 1$. 
Then we have that
\begin{equation}
\label{inequality for hat v L}
|v \cdot e_2'|+  |v \cdot e_3'| \lesssim |v\times \frac x r|,\quad     \frac { |v \times \frac {x} {r} | } {v_0 } \lesssim \sqrt{\kappa(v) },\quad \frac {c } {v_0 } \lesssim \sqrt{\kappa(v) },
\end{equation}
and
\begin{equation}
\label{main inequality from 1 t - r c to 1 + t + r}
1+t +|x|  \lesssim  (1 +  |t-\frac {|x| } {c} | )   \langle  v \rangle^4  \langle  x-t \hat{v} \rangle^2.
\end{equation}
\end{lemma}

\begin{lemma}\label{lem:xgect}
For any $t \ge 0, x, v \in \R^3, c \ge 1$, if $|x| \ge c t$, then 
\begin{equation}\label{eq:xgect}
1+t + |x|  \lesssim \langle v \rangle^2  \langle x-t\hat v  \rangle.
\end{equation}
\end{lemma}

\subsection{Commutation vector fields}
\label{sec:commuvectorfields}
In this subsection we elaborate on the commutators for the Maxwell equations and the ones for the relativistic Vlasov part of the equation.

\begin{lemma}\label{basic commutator formula}
For any $Z_1, Z_2 \in \mathbb{K} $ and $\hat{Z}^1, \hat{Z}^2 \in \tilde{\mathbb K} $ we have that 
\[
[Z_1, Z_2] = \sum_{|\alpha| \le 1} C_\alpha Z^\alpha ,\quad [\hat{Z}^1, \hat{Z}^2] = \sum_{|\beta| \le 1} C_\beta \hat{Z}^\beta,
\]
where the constants $C_{\alpha}, C_{\beta}$ are uniformly (in $t, x, c$) bounded from above.(May be a linear combination of $\frac 1 {c}, \frac 1 {c^2}$).
\end{lemma}
\begin{proof}
We first compute the commutators. We have 
\begin{equation}
\label{eq:basic commutator formula1}
\begin{aligned}
&[\partial_t , \hat{\Omega}_{jk}] = 0,\,\, [\partial_{x_i} , \hat{\Omega}_{jk}] = \delta_{ij }\partial_{x_k} -\delta_{ik} \partial_{x_j} ,\,\,  [\partial_t, \Omega_{ j} ]  = \partial_{x_j},\,\, [\partial_{x_i}, \Omega_{ j} ]  = \frac 1 {c^2} \delta_{ij} \partial_{t},\,\,  [\frac 1 c \partial_t, S] =\frac 1 {c^2}  \partial_t ,\,\, [\partial_{x_i}, S ] = \partial_{x_i},\\
&[\hat{\Omega}_{i} , S]=
[\hat{\Omega}_{jk}, S] = 0,\,\, [\hat{\Omega}_{ i},\hat{\Omega}_{ j}] = \frac 1{c^2} \hat{\Omega}_{ij},\,\, [\hat{\Omega}_{i} , \hat{\Omega} _{jk}]=\delta_{ij} \hat{\Omega}_{k} - \delta_{ik} \hat{\Omega}_{j} ,\,\, [\Omega_{ i} , \Omega _{jk}] =  \delta_{ij} \Omega_{k} - \delta_{ik} \Omega_{j}, \\
& [\Omega_{jk} , \Omega_{il}]  =\delta_{jl}\Omega_{ki} + \delta_{ik}\Omega_{jl}+ \delta_{kl} \Omega_{ij} +  \delta_{ij} \Omega_{lk},\,\,[\hat\Omega_{jk} , \hat\Omega_{il}]  =\delta_{jl}\hat\Omega_{ki} + \delta_{ik}\hat\Omega_{jl}+ \delta_{kl} \hat\Omega_{ij} +  \delta_{ij}\hat \Omega_{lk},
\end{aligned}
\end{equation}
and the lemma is thus proved by \eqref{eq:basic commutator formula1}.
\end{proof}
\begin{lemma}\label{commutator of Z J}
For $J_i, i=0, 1, 2, 3$ defined in \eqref{definition J 0 J i}, for any $Z \in \mathbb K$ we have that 
\[
Z (J_i (f) )= \sum_{j=0}^3 \sum_{|\gamma| \le 1} C_{i, j, \gamma}J_j (  \hat{Z}^\gamma f  ) 
\]
where the constants are uniformly in $c$ bounded from above. 
\end{lemma}
\begin{proof} For $i, j, k=1, 2, 3$, using integration by parts, we can compute that 
\begin{align*}
J_0(\hat{Z} f) = Z( J_0( f))  ,\quad J_k(\hat{Z} f) = Z (J_i( f) ),\quad Z =\partial_t, \frac 1 c\partial_{x_i} S ,\quad J_0(\hat{\Omega}_{i} f)= \Omega_{i} J_0( f)- \frac 1 {c^2} J_i (f),\quad J_0(\hat{\Omega}_{ij} f)= \Omega_{ij} J_0( f),
\end{align*}
as well as
\begin{align*}
J_k  (\hat{\Omega}_{i} f)= \Omega_{i} J_k( f)- \delta_{ik} J_0 (f),\quad J_k(\hat{\Omega}_{ij} f)= \Omega_{ij} J_0( f) + \delta_{ik} J_j(f) -\delta_{jk } J_i (f),\quad  
\end{align*}
so the proof is thus finished. 
\end{proof}
\begin{lemma}\label{commutator of Z alpha E B}
Suppose that $[f, E, B]$ solves the relativistic Vlasov-Maxwell-Boltzmann system \eqref{RVMB1}, \eqref{RVMB2}. Denote that 
\[
h^i_1(f) : =\int_{\mathbb R^3}\pare{c^2   \partial_{x_i}  + \frac {c v_i} {v_0}  \partial_t}f\dd v, \quad h^{j,k}_2(f) : = \int_{\mathbb R^3}\pare{\frac {c^2 v_j} {v_0} \partial_{x_k} - \frac {c^2  v_k} {v_0} \partial_{x_j}}f\dd v,
\]
then there exist uniformly bounded constants $C_{\alpha,\beta,i,l}$ and $C_{\alpha,\beta,j, k,l}$, such that
\begin{equation}
\label{eq:comm0}
 ( \partial_t^2 - {{c^2}} \Delta_x )  (Z^\alpha [E, B] )_l  = \sum_{\beta \le \alpha }\sum_{i=1}^3  C_{\alpha, \beta, i, l}  h^i_1 (\hat{Z}^\beta f )  + \sum_{\beta \le \alpha } \sum_{j,k=1}^3 C_{\alpha, \beta, j,k, l}  h^{j,k}_2 (\hat{Z}^\beta f ), \quad l=1,2,...6 .
\end{equation}
\end{lemma}

\begin{proof}
Taking another $t$ derivative on \eqref{RVMB2},  using that  $\nabla \times (\nabla \times f) =  \nabla (\nabla \cdot f) - \Delta f $
we have that
\begin{equation}
\label{eq:commuE}
\partial_t^2 E 
= -c^2 \nabla \times (\nabla \times E) - \partial_t J = c^2 \Delta E   -c^2 \nabla J_0 -\partial_t J,\quad
\partial_t^2 B 
= -c^2 \nabla \times (\nabla \times B ) +  c \nabla \times J   = c^2 \Delta B    +c  \nabla \times J,
\end{equation}
and
\[
(c^2 \nabla_x J_0 +\partial_t J)_i = h^i_1 (f) ,\quad c(\nabla \times J)_i = \int_{\R^3} (\frac {c^2 v_{i+2}} {v_0} \partial_{x_{i+1}} - \frac {c^2  v_{i+1}} {v_0} \partial_{x_{i+2}} )f \dd v = h_2^{i+2, i+1}(f),
\]
where $v_{i+3}=v_i,  x_{i+3}=x_{i}, i=1, 2$. Next, using 
\[
[\partial_t^2 - c^2 \Delta_x, Z] = 0,\quad Z\in \{\frac 1 c \partial_t, \partial_{x_i}, \Omega_{ij}, \Omega_{i} \},\quad [\partial_t^2 - c^2 \Delta_x, S   ] = \frac 2 c  (\partial_t^2 - c^2 \Delta_x).
\]
Combine the estimates of commutators above together and remind \eqref{eq:commuE}, we have
\begin{equation}
\label{eq:commu1}
 ( \partial_t^2 -{{c^2}} \Delta_x )(Z^\alpha [E,B] )_l =  \sum_{\beta \le \alpha} C_{\alpha, \beta} Z^\beta    ( \partial_t^2 -{{c^2}}  \Delta_x )([E,B] )_l=\sum_{\beta \le \alpha}\sum_{i, j, k=1}^3  C_{\alpha, \beta, i, l} Z^\beta h_1^i(f) . + C_{\alpha, \beta,  j, k, l} h_2^{j, k}(f).
\end{equation}
We next calculate the right hand side of \eqref{eq:commu1}. 
\begin{enumerate}
\item 
 From integration by parts, we compute that 
\begin{equation}
\label{eq:commchange1}
\begin{aligned}
h^i_1
\pare{\hat{\Omega}_{j}f}  
=  & \Omega_{j} \pare{ h^i_1 (f )} + \int_{\R^3 }(\delta_{ij} \partial_{t}  + \frac {c v_i} {v_0} \partial_{x_j}  )f \dd v    +\int_{\R^3 }(c^2 \partial_{x_i}  + \frac {c v_i} {v_0}  \partial_t)  (\frac {v_0} {c}  \partial_{v_j} ) f \dd v    
\\
= &   \Omega_{j} \pare{ h^i_1 (f )}  + \int_{\R^3 }(\delta_{ij} \partial_{t}  + \frac {cv_i} {v_0} \partial_{x_j}  ) f \dd v   -  \int_{\R^3 }(\delta_{ij} \partial_{t}  + \frac {cv_j} {v_0} \partial_{x_i}  ) f  \dd v  
\\
 = &  
 \Omega_{j} \pare{ h^i_1 (f )}     + \frac 1 ch^{i, j}_2 (f),
\end{aligned}
\end{equation}
and 
\begin{equation}
\label{eq:commchange2}
    \begin{aligned}
 h^i_1\pare{\hat \Omega_{jk} f}
= & \Omega_{jk} \pare{ h^i_1 (f)}  + \int_{\R^3 }c^2 (\delta_{ij} \partial_{x_k}  - \delta_{ik} \partial_{x_j}  )f \dd v  +\int_{\R^3 } (c^2 \partial_{x_i}  + \frac {c v_i} {v_0}  \partial_t) (v_j\partial_{v_k } - v_k \partial_{v_j}  ) f \dd v    
\\
 = & \Omega_{jk}  \pare{ h^i_1 (f)}  + \int_{\R^3 }c^2 (\delta_{ij} \partial_{x_k}  - \delta_{ik} \partial_{x_j}  )f \dd v  +  \int_{\R^3}  (-   \delta_{ik} \frac {cv_j} {v_0}  +  \delta_{ij} \frac {cv_k} {v_0} )  \partial_t    f \dd v    
\\
 = &\Omega_{jk}  \pare{ h^i_1 (f)}     -   \delta_{ik}  
 h^j_1  (f)  +  \delta_{ij} 
 h^k_1 (f)  ,
\end{aligned}
\end{equation}
\begin{equation}
\label{eq:commchange3}
 h^i_1\pare{ Sf}  = S\pare{ h^i_1 (f)}  +\frac 1c  h^i_1
(f) .
\end{equation}
\item
Similarly as \eqref{eq:commchange1} to \eqref{eq:commchange3}, we have
\begin{equation}
\label{eq:commchange4}
h^{i,j}_2\pare{\hat\Omega_{k} f}=\Omega_{k}\pare{h^{i,j}_2 (f)} +\frac 1 {c}\pare{-   \delta_{ik} h^j_1 (  f ) +  \delta_{jk} h^i_1 ( f ) },  
\end{equation}
\begin{equation}
\label{eq:commchange5}
h^{i,j}_2\pare{\hat\Omega_{lk} f}=\Omega_{lk}\pare{h^{ij}_2 (f)  }+  \delta_{ik} h^{j,l}_2   (f)  +  \delta_{jk} h^{l,i}_2  (f)+ \delta_{jl} h^{i,k}_2 (  f ) +  \delta_{il} h^{k,j}_2 ( f ) ,
\end{equation}
\begin{equation}
\label{eq:commchange6}
h^{i,j}_2\pare{Sf}=S\pare{h^{i,j}_2 (f)}+ \frac 1c  h^{i,j}_2 (f). 
\end{equation}
\end{enumerate}
Equations \eqref{eq:commchange1}- \eqref{eq:commchange6}  imply that there exist uniformly bounded constants $C_{\alpha, \beta, i} $ and $C_{\alpha, \beta,  j,k} $, such that
\begin{equation}
\label{eq:commu2}
Z^\alpha h^i_1  (f), \quad Z^\alpha h^{j,k}_2  (f)    = \sum_{\beta  \le \alpha}  C_{\alpha, \beta, i} h^{i}_1 (\hat{Z}^\beta f )+\sum_{\beta  \le \alpha} C_{\alpha, \beta, j,k } h^{j,k}_2 (\hat{Z}^\beta f ),
\end{equation}
and \eqref{eq:comm0} is proved after gathering \eqref{eq:commu1} and \eqref{eq:commu2}.
\end{proof}


\begin{lemma}\label{basic representation for e 1 ' e 2 ' e 3 '}
 For the $e_i'$ defined in \eqref{eq:e1'e2'e3'}, there exist uniformly bounded constants $C_{i, j,k} (\theta, \phi)$, such that 
 \begin{equation}
     \label{eq:e1'e2'e3'2}
e_1' \times \nabla =\frac {(\Omega_{23},\Omega_{31},\Omega_{12})}r,\quad     e_i' \cdot \nabla = \frac 1 r \sum_{j,k=1, j<k}^3  C_{i, j,k}(\theta, \phi) \Omega_{jk},\quad i=2,3.
 \end{equation}
 \end{lemma}
 \begin{proof}Remind \eqref{eq:e1'e2'e3'1}, we have $e_1' \times \nabla  = \frac 1 r (x \times \nabla_x)$, thus the first equality is proved. The second equality is the direct result of
\[
e_2'   \cdot \nabla = ( e_3' \times e_1') \cdot \nabla =   e_3' \cdot (e_1 '\times \nabla) ,\quad e_3'   \cdot \nabla = ( e_1' \times e_2') \cdot \nabla =  - e_2' \cdot (e_1 '\times \nabla).
\]
\end{proof}

\begin{lemma} For any vector function $E$, for any $ j =1 ,2, 3$ we have the following identities  
\begin{equation}
\label{equality for partial x j E e 1'}
\partial_{x_j } E \cdot e_1' = \frac 1 r \partial_{x_j } E \cdot  x = \frac {x_j} {r} (\nabla \cdot E) + \frac 1 r \sum_{ i \neq j} \Omega_{ij}  E_i,
\end{equation}
\begin{equation}
\label{equality for partial r E e 1'}
\partial_{r} (E )\cdot e_1 '     = \nabla \cdot E  + \frac 1 {r^2} \sum_{i, j,  i \neq j}   x_j  \Omega_{ij}  E_i = \nabla \cdot E + \frac 1 r \sum_{i, j} C_{i, j} (\theta, \phi  )\Omega_{ij}  E_i,
\end{equation}
and
\begin{equation}
\label{equality for nabla times E e 3' -E e 2'}
(\nabla \times E ) \cdot e_3'  -  \partial_{r} (E) \cdot e_2 ' = -e_1'    \cdot ( e_2'    \cdot  \nabla ) E,\quad  (\nabla \times E ) \cdot e_2'  +   \partial_{r} (E) \cdot e_3 ' = e_1'    \cdot ( e_3 '    \cdot  \nabla ) E.
\end{equation}
\end{lemma}
\begin{proof} For any $j=1, 2, 3$ we compute that 
\begin{equation}
\label{equality for partial x j E e 1'-1}
\partial_{x_j} E \cdot x =x_j  \partial_{x_j} E_j  +  \sum_{ i \neq j} \partial_{x_j} E_i  x_i =  x_j \sum_{i=1}^3 \partial_{x_i} E_ i  -  x_j \sum_{i \neq j} \partial_{x_i} E_ i  + \sum_{ i \neq j} \partial_{x_j} E_i    x_i  = x_j (\nabla \cdot E) +\sum_{ i \neq j} \Omega_{ij}  E_i   .
\end{equation}
Thus taking the summation, we have that
\begin{multline}
\label{equality for partial x j E e 1'-2}
\partial_{r} (E \cdot e_1 ' )  =  \frac 1 r \sum_{j} x_j \partial_{x_j } E \cdot \frac {x} r = \frac 1 {r^2} [\sum_{j} x_j  \partial_{x_j} E \cdot x]=  \frac 1 {r^2} [\sum_{j}| x_j |^2 (\nabla \cdot E) + x_j \sum_{ i \neq j} \Omega_{ij}  E_i  ] = \nabla \cdot E  + \frac 1 {r^2} \sum_{i, j,  i \neq j}   x_j  \Omega_{ij}  E_i,
\end{multline}
and \eqref{equality for partial x j E e 1'}, \eqref{equality for partial r E e 1'} can be obtained from \eqref{equality for partial x j E e 1'-1}, \eqref{equality for partial x j E e 1'-2}.
For \eqref{equality for nabla times E e 3' -E e 2'}, 
remind \eqref{basic identities for cross product}, 
we have that 
\begin{multline*}
(\nabla \times E ) \cdot e_3'  -  \partial_{r} (E) \cdot e_2 '     = (e_1' \times e_2' ) \cdot   (\nabla \times E ) - e_2 ' \cdot  (e_1' \cdot \nabla ) E =    e_2'    \cdot( e_1' \cdot \nabla  ) E  -  e_1'    \cdot ( e_2'    \cdot  \nabla ) E -   e_2'    \cdot( e_1' \cdot \nabla  ) E = - e_1'    \cdot ( e_2'    \cdot  \nabla ) E, 
\end{multline*}
and the other equality in \eqref{equality for nabla times E e 3' -E e 2'} can be proved similarly, so the proof is thus finished. 
\end{proof}
Next, we show that
 \begin{lemma}\label{partial t - partial r derivative of E B}
For $\xi = \alpha_1, \alpha_2, \rho, \sigma$, suppose that  $[f, E, B]$ satisfies the Vlasov-Maxwell-Boltzmann system \eqref{RVMB1}, \eqref{RVMB2} and $J_i$ is defined in \eqref{definition J 0 J i}. Then we have
\[
 (\frac 1 c \partial_{t} - \partial_r  )\xi( [E, B])  =\sum_{i=0}^3 C_i(\theta, \phi) J_i (f) +\frac 1 r \sum_{ |\gamma| \le 1} \sum_{i=1}^3   C_{\gamma, i, 1} (\theta, \phi) Z^\gamma E_i +  C_{\gamma, i, 2}(\theta, \phi) Z^\gamma B_i,
\]
where the constants are uniformly bounded from above.
\end{lemma}
\begin{proof}

First, we compute 
$\partial_t (\rho [E, B]  )= e_1 ' \cdot  (\partial_t E ) = c e_1 ' \cdot (\nabla \times B )  +  e_1' \cdot J = c( e_1' \times \nabla) \cdot B +  e_1' \cdot J ,
$. Together with \eqref{equality for partial r E e 1'}, we have
\begin{align*}
\partial_r \rho = \partial_r (E )\cdot e_1 '  = \nabla \cdot E + \frac 1 r \sum_{i, j} C_{i, j} (\theta, \phi  )\Omega_{ij}  E_i  = J_0 + \frac 1 r \sum_{i, j} C_{i, j} (\theta, \phi  )\Omega_{ij}  E_i .
\end{align*}
Thus by Lemma \ref{basic representation for e 1 ' e 2 ' e 3 '} the case $\xi=\rho$ is proved. The case $\xi=\sigma$ is similar. For $\xi=\alpha_1$, 
from \eqref{equality for nabla times E e 3' -E e 2'}, we have
\begin{align*}
\frac 1 c \partial_t (\alpha_1[E, B] ) -  \partial_r (\alpha_1[E, B] )  =  & -(\nabla \times E )\cdot e_2' + (\nabla \times B )\cdot e_3'  + \frac  1 c J \cdot e_3'  - \partial_r (B \cdot e_2 ' +E \cdot e_3 ' )
=  e_1'  \cdot (   e_3' \cdot  \nabla) E  + e_1'    \cdot ( e_2'    \cdot  \nabla ) B + \frac  1 c J \cdot e_3' ,  
\end{align*}
so by using Lemma \ref{basic representation for e 1 ' e 2 ' e 3 '} we have finished the proof for $\alpha_1$. Similarly for $\xi=\alpha_2$, we compute that 
\[
\frac 1 c \partial_t (\alpha_2   [E, B] ) -  \partial_r (\alpha_2   [E, B] ) 
=  - e_2'  \cdot (   e_3' \cdot  \nabla) B   + e_1'    \cdot ( e_2'    \cdot  \nabla ) E,
\]
so the lemma is thus proved. 
\end{proof}

\begin{lemma}\label{partial t - partial r derivative of Z E B}
For $\xi = \alpha_1, \alpha_2, \rho, \sigma$ and $Z =S, \Omega_i, \Omega_{i j} (1\le i, j\le 3)$, suppose that  $[f, E, B]$ satisfies the Vlasov-Maxwell-Boltzmann system \eqref{RVMB1}, \eqref{RVMB2} and $J$ is defined in \eqref{definition J 0 J i}. Then
\[
(\frac 1 c \partial_t -\partial_r)\xi(  \mathpzc D_Z([E, B])  )  =\frac 1 r \left( \sum_{|\alpha| \le 2} \sum_{i=1}^3 C_{\alpha, i, 1} (\theta, \phi) Z^\alpha E_i + C_{\alpha, i, 2} (\theta, \phi) Z^\alpha B_i \right)  + \sum_{|\beta| \le 1} \sum_{i=0}^3 C_{\beta, i} (\theta, \phi)J_i(\hat{Z}^\beta f ) ,
\]
where the constants are uniformly bounded from above. 
\end{lemma}

\begin{proof}
We split into three different cases of $Z$.
\begin{enumerate}
    \item {If ${Z}  = S$}. Then $\mathpzc D_Z([E, B]) = Z [E, B]$, and $[\partial_t - c\partial_r, S] = \frac  1 c (\partial_t -c \partial_r)$. Since $S$ only contains $\partial_r, \partial_t$, which commutes with $e_2', e_3', e_1'$, thus commutes with $\xi$, 
\begin{align*}
(\frac 1 c \partial_t  -  \partial_r )[  \xi(S [E, B] ) ]  ] = (\frac 1 c \partial_t  - \partial_r ) S [  \xi([E, B] ) ]  ] = S (\frac 1 c \partial_t  - \partial_r )[   \xi([E, B] ) ]+ \frac 1 c  ( \frac 1 {c} \partial_t  - \partial_r )[   \xi( [E, B] ) ] .
\end{align*}
Together with Lemma \ref{partial t - partial r derivative of E B} and the fact that $S$ commutes with $\theta, \phi$, the case ${Z}  = S$ is thus proved. 

 \item {If ${Z}  =\frac 1 c \partial_t $}. This case is similar, we have  $\mathpzc D_Z([E, B]) = Z [E, B]$. Since $ \partial_t$ commutes with $e_2', e_3', e_1'$, thus commutes with $\xi$, we have
\begin{align*}
(\frac 1 c \partial_t  -  \partial_r )[  \xi(Z [E, B] ) ]  ] = (\frac 1 c \partial_t  - \partial_r )Z [  \xi([E, B] ) ]  ] = Z (\frac 1 c \partial_t  - \partial_r )[   \xi([E, B] ) ]
\end{align*}
Together with Lemma \ref{partial t - partial r derivative of E B} and the fact that $\partial_t$ commutes with $\theta, \phi$, the case ${Z}  = \frac 1 c \partial_t$ is thus proved. 
\item {If ${Z}  =\partial_{x_i}$}. This case is similar as Lemma \ref{partial t - partial r derivative of E B}, we also have $D_Z([E, B]) = Z [E, B]$. First, we compute 
\[
\partial_t (\rho (\mathpzc D_Z([E, B]))   ) = \partial_t (\rho (Z[E, B]   )   )= e_1 ' \cdot  (\partial_t ZE ) = c e_1 ' \cdot (\nabla \times ZB )  +  e_1' \cdot ZJ = c( e_1' \times \nabla) \cdot ZB +  e_1' \cdot ZJ ,
\]
Together with \eqref{equality for partial r E e 1'}, we have
\begin{align*}
\partial_r \rho = \partial_r (ZE )\cdot e_1 '  = \nabla \cdot ZE + \frac 1 r \sum_{i, j} C_{i, j} (\theta, \phi  )\Omega_{ij}  ZE_i  = ZJ_0 + \frac 1 r \sum_{i, j} C_{i, j} (\theta, \phi  )\Omega_{ij}  ZE_i .
\end{align*}
using Lemma \ref{commutator of Z J} on $J$ term, by Lemma \ref{basic representation for e 1 ' e 2 ' e 3 '} the case $\xi=\rho$ is proved. The case $\xi=\sigma$ is similar. For $\xi=\alpha_1$, 
from \eqref{equality for nabla times E e 3' -E e 2'}, we have
\begin{align*}
\frac 1 c \partial_t (\alpha_1  (Z[E, B]) ) -  \partial_r (\alpha_1 (Z[E, B]) )  =  & -(\nabla \times ZE )\cdot e_2' + (\nabla \times ZB )\cdot e_3'  + \frac  1 c ZJ \cdot e_3'  - \partial_r (ZB \cdot e_2 ' +ZE \cdot e_3 ' )
\\
= & e_1'  \cdot (   e_3' \cdot  \nabla) (ZE)  + e_1'    \cdot ( e_2'    \cdot  \nabla ) (Z B) + \frac  1 c ZJ \cdot e_3' ,  
\end{align*}
so by using Lemma \ref{commutator of Z J} and Lemma \ref{basic representation for e 1 ' e 2 ' e 3 '} we have finished the proof for $\alpha_1$. The case $\xi=\alpha_2$ can be proved similarly. 

\item {If $Z=\Omega$}.
We prove the case $Z = {\Omega}_1$ and ${\Omega}_{2}, {\Omega}_{3}$ can be proved similarly.

\underline{Case 1: $\xi=\rho,\sigma$.} We only proof the $\rho$ term as the $\sigma$ term is similar. Recall that $Z= {\Omega}_1 = t \partial_{x_1}  + \frac 1 {c^2} x_1 \partial_t
$, we have
\begin{equation}
\label{commutator for Z Omega 1}
\begin{aligned}
&\mathpzc D^{(1)}_Z(E, B)    = ZE + \frac 1 c (0,   -B_3, B_2)   = ZE+ \frac 1 c e_1 \times  B    ,\quad \mathpzc D^{(2)}_Z(E, B)   : = ZB+ \frac 1 c(0,  E_3, -E_2) =ZB - \frac 1 c e_1 \times E,  \\
&\partial_t Z = Z \partial_t + \partial_{x_1},\quad  \partial_{x_1} Z    = Z\partial_{x_1} + \frac 1 {c^2} \partial_t  , \quad Z\partial_{x_2} = \partial_{x_2} Z,\quad Z\partial_{x_3} = \partial_{x_3} Z,\quad  \nabla   Z = Z \nabla +  e_1 \frac 1 {c^2} \partial_t  .
\end{aligned}
\end{equation}
Using \eqref{equality for partial x j E e 1'} we have that 
\begin{align*}
\frac 1 c \partial_t (Z E + \frac 1 c  e_1 \times B)   \cdot e_1'   =&  \frac 1 c  ( \partial_{t} Z E )\cdot  e_1'   + \frac 1 {c^2}  (  e_1 \times  \partial_t  B)   \cdot e_1'    
= \frac 1 c Z\partial_t E   \cdot e_1'  +\frac  1 c  ( \partial_{x_1} E \cdot e_1'  )   +   \frac 1 {c^2}    (  e_1 \times  \partial_t  B)   \cdot e_1'\\    
=& (  Z(\nabla \times B)  \cdot e_1'  +\frac 1 c Z  J \cdot e_1' ) + \frac 1 c  ( \partial_{x_1} E \cdot e_1'  )  + \frac 1 {c^2}  (  e_1 \times  \partial_t  B)   \cdot e_1'     
\\
=& (  (\nabla Z \times B)  \cdot e_1'  -  \frac 1 {c^2}  \partial_t (e_1 \times B) \cdot e_1'  + \frac 1 c Z  J \cdot e_1' )+  \frac 1 c ( \partial_{x_1} E \cdot e_1'  ) +    \frac 1 {c^2}    (  e_1 \times  \partial_t  B)   \cdot e_1'       
\\
=&    (e_1' \times \nabla)  \cdot  ZB +  \frac 1 c   Z  J \cdot e_1' +  \frac 1 c  \partial_{x_1} E \cdot e_1'  
\\
= &    \frac 1 r \sum_{i, j} C_{i, j} (\theta, \phi  )\Omega_{ij}  E_i  + \frac 1 c  Z  J \cdot e_1' +  \frac 1 c  \frac {x_1 } {r} J_0    + \frac 1 r \sum_{ i \neq 1} \Omega_{i1}  E_i   .
\end{align*}
Together with $\nabla \cdot  (ZE  ) = Z \nabla \cdot E +  \frac 1 {c^2}  \partial_t(e_1 \cdot E) =Z J_0 + \frac 1 {c^2} e_1 \cdot J  +  \frac 1 c (e_1 \cdot (\nabla \times B) )$, \eqref{basic identities for cross product} and  \eqref{equality for partial r E e 1'}, we compute that 
\begin{align*}
\partial_r( Z E + \frac 1 c  e_1 \times B )\cdot e_1'   =& (\partial_{r} (Z E_i)) \cdot e_1'  +\frac 1 c \partial_r ( e_1 \times B )\cdot e_1' 
\\
=&   \nabla \cdot (ZE) + \frac 1 r \sum_{i, j} C_{i, j} (\theta, \phi  )\Omega_{ij} (Z E_i)  +\frac 1 c \nabla \cdot  ( e_1 \times B)  + \frac 1 r \sum_{i, j, k}  C_{i, j, k} (\theta, \phi  )\Omega_{ij}  B_k
\\
=&Z J_0 + \frac 1 {c^2} e_1 \cdot J   +  \frac 1 c (e_1 \cdot (\nabla \times B) )   + \frac 1 r \sum_{i, j} C_{i, j} (\theta, \phi  )\Omega_{ij} (Z E_i)   +\frac 1 c \nabla \cdot  ( e_1 \times B)  + \frac 1 r \sum_{i, j, k}  C_{i, j, k} (\theta, \phi  )\Omega_{ij}  B_k 
\\
=&Z J_0   + \frac 1 r \sum_{i, j} C_{i, j} (\theta, \phi  )\Omega_{ij} (Z E_i)   + \frac 1 r \sum_{i, j, k}  C_{i, j, k} (\theta, \phi  )\Omega_{ij}  B_k ,
\end{align*}
then we can finish the proof for $\rho$ and $\sigma$ by using Lemma \ref{basic representation for e 1 ' e 2 ' e 3 '} and \eqref{eq:commu2}.\\
\underline{Case 2: $\xi=\alpha_1,\alpha_2$.}
We just prove the $\alpha_1$ term 
\[
\alpha_1(\mathpzc D_Z[E, B])   = (Z B -  \frac 1 c  e_1 \times E )\cdot e_2'+  (Z E + \frac 1 c  e_1 \times B )  \cdot e_3'=   Z B  \cdot e_2' +ZE \cdot e_3'   -  \frac 1 c (e_1 \times E) \cdot e_2' + \frac 1 c  ( e_1 \times B) \cdot e_3'  ,
\]
and the $\alpha_2$ term can be proved similarly. Remind \eqref{basic identities for cross product-2}.
Using \eqref{commutator for Z Omega 1} we compute that 
\begin{align*}
\frac 1 c \partial_t \alpha_1 =  &  \frac 1 c  Z \partial_t B \cdot e_2'  + \frac 1 c \partial_{x_1} B \cdot e_2'   + \frac 1 c  Z \partial_t E \cdot e_3'  +   \partial_{x_1} E  \cdot e_3' -  \frac 1 {c^2}  (e_1 \times \partial_t  E) \cdot e_2' + \frac 1 {c^2}   ( e_1 \times \partial_t B) \cdot e_2'
\\
= &      - Z( \nabla \times E ) \cdot e_2'    + \frac 1 c  \partial_{x_1} B \cdot e_2'   +   Z ( \nabla \times  B) \cdot e_3'  +   \partial_{x_1} E  \cdot e_3'+ \frac 1 c  J \cdot e_3' -  \frac 1 {c^2}  (e_1 \times \partial_t  E) \cdot e_2' + \frac 1 {c^2}   ( e_1 \times \partial_t B) \cdot e_2'      
\\
= &       - ( \nabla \times ZE ) \cdot e_2'    +\frac 1 {c^2}   ( e_1 \times \partial_t E  ) \cdot e_2'   +\frac 1 c  \partial_{x_1} B \cdot e_2'   +      ( \nabla \times  Z B) \cdot e_3'  + \frac 1 c  J \cdot e_3' 
\\
&-  \frac 1 {c^2}  (e_1 \times \partial_t  B  ) \cdot e_2'   +  \frac 1 c   \partial_{x_1} E  \cdot e_3 '    -  \frac 1 {c^2}  (e_1 \times \partial_t  E) \cdot e_2' + \frac 1 {c^2}   ( e_1 \times \partial_t B) \cdot e_2'
\\
 = &    - ( \nabla \times Z E ) \cdot e_2'  +  \frac 1 c \partial_{x_1} B \cdot e_2'   +  c  ( \nabla \times  Z B) \cdot e_3'  + \frac 1 c  \partial_{x_1} E  \cdot e_3'+ \frac 1 c J \cdot e_3'  ,
\end{align*}
moreover, we have $\partial_r (\alpha_2) = \partial_r(Z B \cdot e_2' + ZE \cdot e_3'  ) - \frac 1 c (e_1 \times \partial_r  E) \cdot e_2' + \frac 1 c  ( e_1 \times  \partial_r B) \cdot e_3'$. Since $e_1$ commutes with $\partial_{x_i}$, we obtain from \eqref{basic identities for cross product} that
$e_2' \cdot (\partial_{x_1} B) = e_2' \cdot (e_1 \cdot \nabla) B =     (e_2' \times \nabla) \cdot  (B \times e_1) + (e_2' \cdot e_1) (\nabla \cdot B) 
$.And from \eqref{basic identities for cross product} , we have
\begin{multline*}
     ( e_1 \times  \partial_r B) \cdot e_3' = e_3' \cdot  (e_1' \cdot \nabla ) (e_1 \times B)  = ((  e_1' \times e_3') \times \nabla )   (e_1 \times B) +  e_1' \cdot  (e_3' \cdot \nabla ) (e_1 \times B)  =  ( e_2' \times \nabla )   ( B \times e_1 ) + e_1' \cdot  (e_3' \cdot \nabla ) (e_1 \times B) .
\end{multline*}
So by Lemma \ref{basic representation for e 1 ' e 2 ' e 3 '}, we have
\begin{align*}
&e_2' \cdot (\partial_{x_1} B)  -  ( e_1 \times  \partial_r B) \cdot e_3'  =  (e_2' \cdot e_1) (\nabla \cdot B) - e_1' \cdot  (e_3' \cdot \nabla ) (e_1 \times B)  =     \frac 1 r \sum_{i, j, k} C_{i, j, k} (\theta, \phi  )\Omega_{ij} ( E_k),\\   
&\partial_{x_1} E  \cdot e_3'+  (e_1 \times \partial_r  E) \cdot e_2'    = (e_3' \cdot e_1) J_0  +  \frac 1 r \sum_{i, j, k} C_{i, j, k} (\theta, \phi  )\Omega_{ij} ( B_k)  .
\end{align*}
Thus  we compute that 
\begin{align*}
(\frac 1 c \partial_t -\partial_r ) (\alpha_1) =&  ( \nabla \times Z B ) \cdot e_3'  -  ( \nabla \times  Z E) \cdot e_2'  - \partial_r(Z B \cdot e_2' + ZE \cdot e_3'  )+ \frac 1 c J \cdot e_3'   
\\
& + \frac 1 c \partial_{x_1} B \cdot e_2'    - \frac 1 c (e_1 \times \partial_r  B) \cdot e_3' +    \frac  1 c \partial_{x_1} E  \cdot e_3'  + \frac 1 c  ( e_1 \times  \partial_r E) \cdot e_2'   ,
\end{align*}
so the proof is thus finished by using \eqref{equality for nabla times E e 3' -E e 2'}.

\item{If $Z=\Omega_{ij}$}.
We just prove for $Z = \Omega_{12}$, the cases $\Omega_{13}, \Omega_{23}$ are similar.\\
\underline{Case 1: $\xi=\rho,\sigma$.} We only proof the $\rho$ term as the $\sigma$ term is similar.
Remind $Z = \Omega_{12} = x_1\partial_{x_2} - x_2 \partial_{x_1}  $, we have
 \begin{equation}
\label{commutator for Z Omega 1 2}
\begin{aligned}
&\mathpzc D^{(1)}_Z[E, B ]   = ZE+ (E_2 , - E_1 , 0)  =  ZE + E \times e_3 ,\quad \mathpzc D^{(2)}_Z[E, B ]  = ZB + (B_2 , -B_1, 0 ) =ZB +  B\times e_3,\\ 
&\partial_t Z = Z \partial_t,\quad Z\partial_{x_1} = \partial_{x_1} Z-\partial_{x_2}  , \quad Z\partial_{x_2} = \partial_{x_2} Z+  \partial_{x_1} ,\quad Z\partial_{x_3} = \partial_{x_3} Z,\quad  Z \nabla =\nabla   Z  -   \nabla \times e_3  =\nabla   Z  +    e_3 \times \nabla .
\end{aligned}
\end{equation}
From direct calculations, we have
\begin{align*}
\frac 1 c\partial_t(ZE + E \times e_3) \cdot e_1'  = &\frac 1 c Z   (\partial_t     E) \cdot e_1'  + \frac 1 c (\partial_t E  \times e_3) \cdot e_1' = \frac 1 c Z   (c \nabla \times B +J ) \cdot e_1'  +( (\nabla \times B ) \times e_3 ) \cdot  e_1'
\\
= &  (  \nabla \times Z  B ) \cdot e_1'  +  ((e_3 \times \nabla )\times B )\cdot e_1'  +  \frac 1 c J \cdot e_1'  +( (\nabla \times B ) \times e_3 ) \cdot  e_1'
\\
= &  (  \nabla \times Z  B ) \cdot e_1'  +e_1'   \cdot  ( \nabla \times (B \times e_3 ) ) +  \frac 1 c J \cdot e_1' 
\\
= &  (  \nabla \times Z  B ) \cdot e_1'  + (e_1'   \times    \nabla) \cdot (e_3 \times B ) +  \frac 1 c J \cdot e_1' .
\end{align*}
By using \eqref{equality for partial r E e 1'}, we compute that 
\begin{align*}
\partial_r (ZE + E \times e_3 )\cdot e_1'  =
& (\partial_{r} (Z E)) \cdot e_1'  +\partial_r ( E \times e_3 )\cdot e_1' 
\\
=&   \nabla \cdot (ZE) + \frac 1 r \sum_{i, j} C_{i, j} (\theta, \phi  )\Omega_{ij} (Z E_i) + \nabla \cdot ( E \times e_3 ) + \frac 1 r \sum_{i, j, k}  C_{i, j, k} (\theta, \phi  )\Omega_{ij}  E_k
\\
=&Z (\nabla \cdot E) -( e_3 \times \nabla  )\cdot E  + \frac 1 r \sum_{i, j} C_{i, j} (\theta, \phi  )\Omega_{ij} (Z E_i)  + \nabla \cdot ( E \times e_3 ) + \frac 1 r \sum_{i, j, k}  C_{i, j, k} (\theta, \phi  )\Omega_{ij}  E_k
\\
=&Z J_0   + \frac 1 r \sum_{i, j} C_{i, j} (\theta, \phi  )\Omega_{ij} (Z E_i)  + \frac 1 r \sum_{i, j, k}  C_{i, j, k} (\theta, \phi  )\Omega_{ij}  E_k.
\end{align*}
So we have proved it by gathering the terms above and using \eqref{eq:commu2}.\\
\underline{Case 2: $\xi=\alpha_1,\alpha_2$.}
Next we come to prove $\alpha_1$, the $\alpha_2$ term can be proved similarly. 
Remind \eqref{basic identities for cross product-1}.
Together with \eqref{commutator for Z Omega 1 2} and \eqref{basic identities for cross product} we have
\begin{align*}
\frac 1 c  \partial_t \alpha_1 =&\frac 1 c  \partial_t ( ZB  \cdot e_2' ) + \frac 1 c\partial_t (  B \times e_3) \cdot e_2' +\frac 1 c  \partial_t ( ZE  \cdot e_3' ) + \frac 1 c\partial_t (  E \times e_3) \cdot e_3' 
\\
 = &  -  Z (\nabla \times E)   \cdot e_2'  - (  (\nabla \times E) \times e_3) \cdot e_2' +   ( Z (\nabla \times B)  \cdot e_3' ) +  (   (\nabla \times B)  \times e_3) \cdot e_3' 
 \\
 = &   -  (\nabla \times ZE)   \cdot e_2'  -  ((e_3 \times \nabla ) \times E)   \cdot e_2'    - (  (\nabla \times E) \times e_3) \cdot e_2' +    (\nabla \times ZB)  \cdot e_3'  +   ( (e_3 \times \nabla) \times B)  \cdot e_3'  + (   (\nabla \times B)  \times e_3) \cdot e_3' 
  \\
 = &  -  (\nabla \times Z E   )   \cdot e_2'  - e_2'   \cdot  ( \nabla \times (E \times e_3 ) )    +    (\nabla \times ZB)  \cdot e_3'  +e_3'   \cdot  ( \nabla \times (B \times e_3 ) )
 \\
 = &   -   (\nabla \times ZE  )   \cdot e_2'  -  (e_2'   \times    \nabla) \cdot (E \times e_3 )  +   (\nabla \times ZB)  \cdot e_3' + (e_3'   \times    \nabla) \cdot (B   \times e_3 ) ,
\end{align*}
and we compute 
$\partial_r \alpha_1    = \partial_r(Z B) \cdot e_2'  + \partial_r(Z E  )  \cdot e_3' +   ( \partial_r  B \times e_3 ) \cdot e_2'  +  (   \partial_r E \times e_3 ) \cdot e_3'$. Together with \eqref{basic identities for cross product}, we have 
\begin{multline*}
    (\partial_r B \times e_3 ) \cdot e_2' = e_2' \cdot  (e_1' \cdot \nabla ) (B \times e_3)  = ( (e_1'  \times e_2' )\times \nabla )  \cdot   (B \times e_3 ) +   e_1' \cdot  (e_2' \cdot \nabla ) (B \times e_3 ) =  ( e_3' \times \nabla )  \cdot  ( B \times e_3 ) +  e_1' \cdot  (e_2' \cdot \nabla ) (B \times e_3   ) .
\end{multline*}
Similarly, $
 (\partial_r E \times e_3 ) \cdot e_3' 
 =  -( e_2' \times \nabla )  \cdot  ( E \times e_3 ) +  e_1' \cdot  (e_3' \cdot \nabla ) (E \times e_3   ) .$
Thus
\begin{multline*}
( \frac 1 c  \partial_t  -\partial_r ) \alpha_1 =    -   (\nabla \times ZE  )   \cdot e_2'  +   (\nabla \times ZB)  \cdot e_3'  - (\partial_r(Z B) \cdot e_2'  + \partial_r(Z E  )  \cdot e_3')  -  e_1' \cdot  (e_2' \cdot \nabla ) (B \times e_3   ) -  e_1' \cdot  (e_3' \cdot \nabla ) (E \times e_3   ) ,
\end{multline*}
so the proof is thus finished by using  \eqref{equality for nabla times E e 3' -E e 2'} and Lemma \ref{basic representation for e 1 ' e 2 ' e 3 '}. 
\end{enumerate}
\end{proof}
\begin{lemma}\label{inequality for xi partial x i}
For any $t \ge 0, x \in \R^3, c \ge 1$ satisfies that $|x|  \le 10  c (1+t)$, denote $\xi = \alpha_1, \alpha_2, \rho, \sigma$, for $Z_1 =\frac 1 c \partial_t, \partial_{x_i}$  we have 
\[
| \xi (   \mathpzc D_{Z_1} [E, B]  )  |\lesssim  \frac 1 {1+|t-\frac {|x|} {c}| } \left [ | \xi ( [E, B]) |  +    \sum_{Z \in \mathbb{K} } | \xi (  \mathpzc D_{Z}   [E, B]) |  \right]  + \frac 1 {1+t+\frac {|x|} {c}}  \sum_{|\gamma| \le 1 }  |Z^\gamma   [E, B]|  .
\]
\end{lemma}

\begin{proof} The case $t \le 1$ is obvious and we assume later $t \ge 1$. 
We first prove for $Z_1 = \frac 1  c\partial_t$. First, one can check that 
\begin{equation}\label{basic representation for partial t partial x i}
\partial_t  = \frac {c t} {(t+ \frac {r} c )( t -  \frac {r} {c} ) }S  -  \sum_{i} \frac {x_i} {(t+\frac {r} c)(t- \frac {r} {c} ) }   \Omega_{i},\quad \partial_{x_i} = \frac {\Omega_{i} } t - \frac {x_i} {c^2 t} \partial_t  = \frac {\Omega_{i} } t - \frac {x_i} {c t}   \frac 1 c\partial_t ,
\end{equation}
 Recall the definition of $D^{( 1 )}_{\Omega_{ i }} , D^{( 2 )}_{\Omega_{ i }}$ we have that 
\begin{align*}
 \mathpzc D^{( 1 )}_{\Omega_{ i }} [E, B ]     - \Omega_{i} E = \frac 1 c  e_i \times B,\quad  \mathpzc D^{(2)}_{\Omega_{ i }}[E, B ]  - \Omega_{ i } B =  -\frac 1 c e_i \times E .
\end{align*}
Thus we have 
\begin{equation}
\label{basic equality for difference between D Omega_i and Omega i}
\sum_{i } x_i    ( \mathpzc D^{(1)}_{\Omega_{i}}[E, B ]-\Omega_{i} E)  =  \frac 1 c \sum_{i}x_i e_i \times B =\frac 1 c (x \times B) = \frac r c (e_1' \times B), \quad \sum_{i } x_i    ( \mathpzc D^{(2)}_{\Omega_{i}}[E, B ]-\Omega_{i} E)  = -\frac r c (e_1' \times E),
\end{equation}
Next, for $\xi=\rho$, recall that 
\[
\rho  (\mathpzc D_{\frac 1 c \partial_t } [E, B]  ) = \frac 1 c \partial_t E \cdot e_1',   \quad\rho  (\mathpzc D_{S} [E, B]  ) = S E \cdot e_1' ,\quad \rho( \mathpzc D_{\Omega_i} [E, B ])  = \mathpzc D^{(1)}_{\Omega_i} [E, B ]   \cdot e_1'   ,
\]
using \eqref{basic representation for partial t partial x i}, \eqref{basic equality for difference between D Omega_i and Omega i} and the fact that  $(e_1' \times B) \cdot e_1' = 0$ we have 
\begin{align*}
\rho  (\mathpzc D_{\frac 1 c \partial_t } [E, B]  )  =&\left [  \frac {t} { (t+ \frac {r} c )( t -  \frac {r} {c} ) }  S E  +  \sum_{i} \frac {x_i} {c (t+\frac {r} c)(t- \frac {r} {c} ) }   \Omega_{i}  E \right] \cdot e_1'  .
\\
 =&   \frac {t} { (t+ \frac {r} c )( t -  \frac {r} {c} ) } \rho  (\mathpzc D_{S} [E, B]  )   -  \sum_{i} \frac {x_i} {c (t+\frac {r} c)(t- \frac {r} {c} ) }  \rho(\mathpzc D_{\Omega_i} [E, B ])    - [ \sum_{i} \frac {x_i} {c (t+\frac {r} c)(t- \frac {r} {c} ) }   (  \mathpzc D^{(1)}_{\Omega_{i}}[E, B ] -\Omega_{i} E)  ]\cdot e_1'  
 \\
 =&   \frac {t} { (t+ \frac {r} c )( t -  \frac {r} {c} ) } \rho  (\mathpzc D_{S} [E, B]  )   -  \sum_{i} \frac {x_i} {c (t+\frac {r} c)(t- \frac {r} {c} ) }  \rho(\mathpzc D_{\Omega_i} [E, B ])  .
\end{align*}
so the proof for $\rho$ is thus finished, the proof for $\sigma$ can be similar. \\
Next we come to compute the term $\alpha_1$, the case of $\alpha_2$ would be similar. Recall that $\alpha_1 ([E, B] ) = E \cdot e'_2 + B \cdot e'_3$, thus still using \eqref{basic representation for partial t partial x i}, \eqref{basic equality for difference between D Omega_i and Omega i} and the fact that $(e_1' \times B )\cdot e_2'= -  B \cdot e_3', -( e_1' \times E )\cdot  e_3 '= - E \cdot e_2'$ we  have that 
\begin{align*}
\alpha_1 (\mathpzc D_{\frac 1 c \partial_t } [E, B] ) = &  \frac {t} {c (t+ \frac {r} c )( t -  \frac {r} {c} ) }  \alpha_1(\mathpzc D_{S} [E, B]) -  \sum_{i} \frac {x_i} {c (t+\frac {r} c)(t- \frac {r} {c} ) }   \alpha_1 (  \mathpzc D_{\Omega_i} [E, B ]  )  
\\
&+  \frac {1} {c (t+\frac {r} c)(t- \frac {r} {c} ) } \left[\sum_{i}  [  x_i  (\mathpzc D^{(1)}_{\Omega_i} [E, B ]    -\Omega_{i} E) ] \cdot e'_2 +   [  x_i  ( \mathpzc D^{(2)}_{\Omega_i} [E, B ] -\Omega_{i} B) ] \cdot e'_3 \right].
\\
= &  \frac {t} {c (t+ \frac {r} c )( t -  \frac {r} {c} ) }  \alpha_1(\mathpzc D_{S} [E, B])-  \sum_{i} \frac {x_i} {c (t+\frac {r} c)(t- \frac {r} {c} ) }   \alpha_1 (\mathpzc D_{\Omega_i} [E, B ])   
-   \sum_{i} \frac {r} {c^2 (t+\frac {r} c)(t- \frac {r} {c} ) }   \alpha_1 ([E, B]),
\end{align*}
so the proof is thus finished for $Z=\frac 1 c \partial_t$. 
\\For the case $Z=\partial_{x_i}$,  using \eqref{basic representation for partial t partial x i} we have
\[
\xi (\mathpzc D_{\partial_{x_i} } ([E, B]) )  =\xi (\partial_{x_i} ([E, B]) ) =  \frac {1} t \xi (\Omega_{i} ([E, B]) ) - \frac {x_i} {c t} \xi(   \frac 1 {c} \partial_t  [E, B]  ) )=  \frac {1} t \xi (\Omega_{i} ([E, B]) )- \frac {x_i} {c t} (\mathpzc D_{\frac 1 c\partial_{t} } ([E, B]) ) .
\]
Since we assume $r \le 10c( 1+t)$,  the second term can be directly deduced from  $Z=\frac 1 c \partial_t$. For the first term, since $t \ge \frac 1 {10} \frac {|x|} {c}$, we have that 
\[
| \frac {1} t \xi (\Omega_{i} ([E, B]) ) | \lesssim \frac 1 {1+t+\frac {|x|} {c}} |\Omega_{i}[E, B] | \lesssim \sum_{|\gamma| \le 1 } \frac 1 {1+t+\frac {|x|} {c}} |Z^\gamma [E, B] | ,
\]
and the lemma is thus proved by combining the estimates above. 
\end{proof}

\begin{lemma}\label{basic decomposition of the v 0 E v B term}
For any $t\ge 0, x, v \in \R^3, c \ge 1$ and any smooth functions $E(t, x), B(t, x)$, we have that 
\[
\frac c {v_0 }|E +  \frac {v} {v_0} \times B | \lesssim \kappa(v) |  (|B| +|E|) +   \sqrt{\kappa(v) } (| \alpha([B, E]) | +|\rho([B, E] ) |) .
\]
\end{lemma}
\begin{proof}

Recall the definition of $e_1', e_2', e_3'$ in \eqref{eq:e1'e2'e3'}. We have $e_1' \times e_2' =e_3'$ and \eqref{basic identities for cross product-2}, 
thus we compute that 
\begin{align*}
v_0 E +  v \times B  =&  v_0   [ (E \cdot e_1') e_1' + (E \cdot  e_2' )e_2'  +(E \cdot   e_3' )e_3' ]  +   v \times [ (B \cdot e_1') e_1' + (B \cdot e_2' )e_2'  +(B \cdot  e_3' )e_3' ]  
 \\
 = &  \sum_{i=1}^3 [  - (v \cdot e_{i+2}') (B\cdot e_{i+1}')   +  (v \cdot e_{i+1}') (B\cdot e_{i+2}')  +  v_0(E \cdot e_i')     ]e_i', \quad e_{i+3}'=e_i', \quad i=1,2,3.
\end{align*}
Recall that $ E_1 \cdot e_1' = \rho([E, B])$. 
Thus using \eqref{inequality for hat v L} and \eqref{basic identities for cross product-3}, we have 
\begin{align*}
| v_0 E + v \times B | \lesssim & | (  - (v \cdot e_1') (B\cdot e_3')    +  v_0(E \cdot e_2')     ) | +  |    (v \cdot e_1') (B\cdot e_2')  +  v_0(E \cdot e_3')      |  + v_0 |\rho| + |v\times \frac x r|  ( |B| +|E|)
\\
\lesssim &  |  (v_0 - (v \cdot e_1') )  ( B\cdot e_3' + E \cdot e_2'  )  | +  |  (v_0 + (v \cdot e_1') )  ( B\cdot e_3' - E \cdot e_2'  )  | 
\\
+& |  (v_0 +  (v \cdot e_1') )  ( B\cdot e_2' + E \cdot e_3'  )  | +  |  (v_0 - (v \cdot e_1') )  ( B\cdot e_2' - E \cdot e_3 ' )  | + v_0 |\rho| + |v\times \frac x r|  ( |B| +|E|)
\\
\lesssim &  v_0 (| \alpha([B, E]) | +|\rho([B, E] ) |)  +  |v\times \frac x r|  ( |B| +|E|) + v_0 \kappa(v)  (|B| +|E|),
\end{align*}
where we use the fact that  $B\cdot e_2' + E \cdot e_3' = \alpha_1 ([E, B]) ,\,\, B\cdot e_3' - E \cdot e_2' = -\alpha_2([E, B] ) ,\,\, v_0 - (v \cdot e_1') =  v_0 \kappa(v)$. Thus by \eqref{inequality for hat v L}, we have 
\[
\frac c {v_0 }|E +  \frac {v} {v_0} \times B |   =  \frac c {v_0 }   \frac 1 {v_0}   |v_0E +  v \times B |   |  \lesssim \kappa(v) |  (|B| +|E|) +   \sqrt{\kappa(v) }(| \alpha([B, E]) | +|\rho([B, E] ) |) ,
\]
the proof is thus finished. 
\end{proof}

\subsection{Estimates of the transport operator}

We recall that 
\[
T_0 = \partial_t  +\hat{v} \cdot \nabla_x  ,\quad T_F = \partial_t +\hat{v} \cdot \nabla_x +\pare{ E + \frac{v}{v_0} \times B} \cdot \nabla_v,\quad \hat{v} = \frac {cv} {v_0} ,
\]
and we write $F(t,x,v,c)=E+ \frac{v}{v_0} \times B$.  For $T_0$ and $T_F$ we have the following lemmas. 
\begin{lemma}\label{Representation for first T 0}
Recall that $T(\alpha)$ defined in \Cref{def:vectorfields}. Then
\begin{equation}
\label{eq:Representation for first T 0}
 v_0T_0\hat{Z}^\alpha =  \sum_{\beta <  \alpha,  T(\beta) = T(\alpha)} C_{\alpha, \beta}\hat{Z}^\beta (v_0T_0)  +   \hat{Z}^\alpha (v_0T_0).
\end{equation}
where the constants are uniformly bounded from above. 
\end{lemma}
\begin{proof}
From direct calculations, we have
\begin{equation}
\label{eq:Representation for first T 0-4}
[v_0T_0,\hat{Z}] =0, \quad \hat{Z} = \frac 1 c \partial_t, \partial_{x_i},\hat{\Omega} _{i},\hat{\Omega} _{i j} ,\quad [v_0T_0,S] =\frac 1c v_0 T_0,
\end{equation}
and the lemma is thus proved after iterating \eqref{eq:Representation for first T 0-4}. 
\end{proof}

\begin{lemma}\label{Exact formula for L Z E B}
Remind $\mathpzc D_Z$ defined in Definition \ref{def:vectorfields}. For any vector functions $F(t, x), G(t, x)$, for any $Z \in \mathbb K$ we have
\begin{equation}
\label{eq:Exact formula for L Z E B}
\hat{Z} (  (v_0 F + v \times G )\cdot \nabla_v f  ) =  (  v_0 \mathpzc D^{(1)}_Z(F, G)  + v \times  \mathpzc D^{(2)}_Z(F, G)    )\cdot \nabla_v f   + (  v_0  F + v \times   G )\cdot \nabla_v \hat{Z} f.  
\end{equation}

\end{lemma}

\begin{proof}
Since $\frac 1 c \partial_t, \partial_{x_i} ,S $ commutes with $\partial_{v_i} $, the case $\hat Z = \frac 1 c \partial_t, \partial_{x_i} ,S $ is direct. 
For $\hat Z=\hat\Omega_{1}$, we have
\begin{align*}
&\hat{Z} [(v_0 F +v \times G ) \cdot  \nabla_v f ] -  (v_0 F +v \times  G )  \cdot \nabla_v \hat{Z}  f 
 =   (v_0  ZF +v \times   ZG ) \cdot \nabla_v  f+ [ \frac {v_0} c \partial_{v_1},  v_0F \cdot \nabla_v + (v\times G) \cdot \nabla_v ]     f .
\end{align*}
Moreover, recall that 
\begin{equation}\label{eq:Maxwellcomm1}
(v \times G )\cdot \nabla_v 
=( v_2 \partial_{v_1} - v_1\partial_{v_2})G_3 + ( v_3 \partial_{v_2} - v_2\partial_{v_3})G_1 + ( v_1\partial_{v_3} - v_3\partial_{v_1})G_2,
\end{equation}we compute that
\begin{align*}
 [ \frac {v_0} {c} \partial_{v_1},  v_0 F \cdot \nabla_v +  v\times G \cdot \nabla_v ]    &= [  \frac {v_0} {c}  \partial_{v_1},  v_0\partial_{v_1} F_1 + v_0\partial_{v_2} F_2+ v_0\partial_{v_3}F_3 ]
 \\
&+ [  \frac {v_0} {c}   \partial_{v_1} , ( v_2 \partial_{v_1} - v_1\partial_{v_2})G_3 + ( v_3 \partial_{v_2} - v_2\partial_{v_3})G_1 + ( v_1\partial_{v_3} - v_3\partial_{v_1})G_2 ]  
\\
&= - \frac 1 c (v_2\partial_{v_1} -v_1\partial_{v_2} )F_2 +  \frac 1 c (v_1\partial_{v_3} -v_3\partial_{v_1} )F_3 -\frac {v_0} {c}  \partial_{v_2}G_3 
+\frac {v_0} {c}   \partial_{v_3}G_2 ,
\end{align*}
thus we have proved \eqref{eq:Exact formula for L Z E B} for $\hat Z=\hat\Omega_{1}$. The cases $\hat Z=\hat\Omega_{2},\hat\Omega_{3}$ are similar. Next, for $\hat Z=\hat \Omega_{12}$, 
remind \eqref{eq:Maxwellcomm1}, we compute 
\begin{align*}
 [v_1\partial_{v_2} - v_2 \partial_{v_1}  ,  v_0 F\cdot \nabla_v + v\times G \cdot \nabla_v ]     &= [v_1\partial_{v_2} - v_2 \partial_{v_1}  ,  v_0\partial_{v_1} F_1 + v_0\partial_{v_2} F_2+ v_0\partial_{v_3}F_3 ]
 \\
&+ [v_1\partial_{v_2} - v_2 \partial_{v_1}  , ( v_2 \partial_{v_1} - v_1\partial_{v_2})G_3 + ( v_3 \partial_{v_2} - v_2\partial_{v_3})G_1 + ( v_1\partial_{v_3} - v_3\partial_{v_1})G_2 ]  
\\
& =  v_0\partial_{v_1} F_2  - v_0 \partial_{v_2} F_1   - (v_1\partial_{v_3} -v_3\partial_{v_1}) G_1 +  ( v_3 \partial_{v_2} - v_2\partial_{v_3}) G_2 ,
\end{align*}
thus we have proved \eqref{eq:Exact formula for L Z E B} for $\hat Z=\hat\Omega_{12}$, and the proof is thus finished. 
\end{proof}
The following lemma show that the structure of the RVMB system \eqref{RVMB1},\eqref{RVMB2}  are preserved by
commutation.
\begin{lemma} Suppose $(f, E, B)$ solves the relativistic Vlasov-Maxwell-Boltzmann system \eqref{RVMB1},\eqref{RVMB2}, then 
\begin{align}\label{Representation for T 0 Z alpha f}
v_0T_0(\hat{Z}^\alpha f)
=& \sum_{|\beta| \le  |\alpha|} C_{\alpha, \beta}   \hat{Z}^\beta [v_0 Q_c(f, f)   ]\\
+&\sum_{i,  j, k=1}^3 \sum_{\alpha_1 +\alpha_2 \le \alpha} ( Z^{\alpha_1} (C_1 E_k  +C_2B_k) v_0\partial_{v_j } \hat{Z}^{\alpha_2} f 
+   Z^{\alpha_1} (C_3 E_k  +C_4B_k)    (v_i \partial_{v_j }  -v_j \partial_{v_i} )\hat{Z}^{\alpha_2} f ,
\end{align}
and we have a more precise estimate for $T_F$:
\begin{align}\label{Representation for T F Z alpha f}
\nonumber
v_0T_F (  \hat{Z}^\alpha f    ) =& \sum_{|\beta| \le  |\alpha|} C_{\alpha, \beta} \hat{Z}^\beta [ v_0 Q_c(f, f) ]+  \sum_{ |\alpha_2| = |\alpha| -1, T(\alpha_2)  = T(\alpha) } C_{\alpha, \alpha_2} (  v_0 E + v \times   B     )\cdot \nabla_v \hat{Z}^{\alpha_2}  f  
\\ \nonumber
&+  \sum_{\alpha =\alpha_1+  \alpha_2 ,|\alpha_2| = |\alpha| -1 } C_{\alpha_1, \alpha_2} (  v_0 \mathpzc D^{(1)}_{Z^{\alpha_1}} (E, B)   + v \times   \mathpzc D^{(2)}_{Z^{\alpha_1}} (E, B)   )\cdot \nabla_v \hat{Z}^{\alpha_2}  f 
\\
&+\sum_{i,  j, k=1}^3 \sum_{\substack{\alpha_1 +\alpha_2 \le \alpha,|\alpha_2| \le |\alpha|-2,\\ T(\alpha_1) +T(\alpha_2) =T(\alpha) }} ( Z^{\alpha_1} (C_1 E_k  +C_2B_k) v_0\partial_{v_j } \hat{Z}^{\alpha_2} f 
+   Z^{\alpha_1} (C_3 E_k  +C_4B_k)   (v_i \partial_{v_j }  -v_j \partial_{v_i} )\hat{Z}^{\alpha_2} f  .
\end{align}
where $C_1, C_2, C_3, C_4$ are constants only depend on $\alpha_1, \alpha_2,\alpha, i, j, k$ uniformly bounded from above. 
\end{lemma}
\begin{proof}
After iterating Lemma \ref{Exact formula for L Z E B} with $F=E(t,x), G=B(t,x)$, we have that 
\begin{align*}
&\hat{Z}^\alpha (  (v_0 E + v \times B )\cdot \nabla_v f  ) 
\\
= &\sum_{i,  j, k=1}^3 \sum_{\substack{\alpha_1 +\alpha_2 \le \alpha,|\alpha_2| \le |\alpha|-2,\\ T(\alpha_1) +T(\alpha_2) =T(\alpha) }}   Z^{\alpha_1} (C_1 E_k  +C_2B_k) v_0\partial_{v_j } \hat{Z}^{\alpha_2} f 
+   Z^{\alpha_1} (C_3 E_k  +C_4B_k)    
(v_i \partial_{v_j }  -v_j \partial_{v_i} )\hat{Z}^{\alpha_2} f 
\\
&+  \sum_{ \alpha_1 +\alpha_2 =\alpha,|\alpha_2| = |\alpha| -1 } C_{\alpha, \alpha_2} (   v_0 \mathpzc D^{(1)}_{Z^{\alpha_1}} (E, B)   + v \times   \mathpzc D^{(2)}_{Z^{\alpha_1}} (E, B)  )   \cdot \nabla_v \hat{Z}^{\alpha_2}  f        + (  v_0  F + v \times   G )\cdot \nabla_v \hat{Z}^\alpha f  
\\
= &\sum_{i,  j, k=1}^3 \sum_{\substack{\alpha_1 +\alpha_2 \le \alpha,|\alpha_2| \le |\alpha|-1,\\ T(\alpha_1) +T(\alpha_2) =T(\alpha) }}   Z^{\alpha_1} (C_1 E_k  +C_2B_k) v_0\partial_{v_j } \hat{Z}^{\alpha_2} f 
+   Z^{\alpha_1} (C_3 E_k  +C_4B_k)        (v_i \partial_{v_j }  -v_j \partial_{v_i} )\hat{Z}^{\alpha_2} f 
 + (  v_0  F + v \times   G )\cdot \nabla_v \hat{Z}^\alpha f  
\\
 = & \sum_{i,  j, k=1}^3 \sum_{\alpha_1 +\alpha_2 \le \alpha, T(\alpha_1) +T(\alpha_2) =T(\alpha)}  Z^{\alpha_1} (C_1 E_k  +C_2B_k) v_0\partial_{v_j } \hat{Z}^{\alpha_2} f 
+   Z^{\alpha_1} (C_3 E_k  +C_4B_k)   (v_i \partial_{v_j }  -v_j \partial_{v_i} )\hat{Z}^{\alpha_2} f .
\end{align*}
Using $v_0T_F(f) = v_0Q_c(f, f), v_0 T_0(f) = -(v_0E + v \times B )\cdot \nabla_v f+ v_0 Q_c(f, f)$
and Lemma \ref{Representation for first T 0}, we have that 
\begin{align*}
 &v_0T_0(\hat{Z}^\alpha f)=  \sum_{\beta <  \alpha, T(\beta)= T(\alpha) } C_{\alpha, \beta}\hat{Z}^\beta ( v_0T_0 f )+  \hat{Z}^\alpha  (v_0T_0 f)\\
 =&\sum_{|\beta| \le  |\alpha|} C_{\alpha, \beta}     \hat{Z}^\beta[ v_0 Q_c(f, f) ]  
+\sum_{i,  j, k=1}^3 \sum_{\alpha_1 +\alpha_2 \le \alpha}  Z^{\alpha_1} (C_1 E_k  +C_2B_k) v_0\partial_{v_j } \hat{Z}^{\alpha_2} f 
+   Z^{\alpha_1} (C_3 E_k  +C_4B_k)    {(v_i \partial_{v_j }  -v_j \partial_{v_i}) } \hat{Z}^{\alpha_2} f ,
\end{align*}
and we compute that
\begin{align*}
v_0T_F (  \hat{Z}^\alpha f    ) =& v_0T_0(\hat{Z}^\alpha f  )     +v_0 (    E +\frac { v } {v_0}\times   B )\cdot \nabla_v \hat{Z}^\alpha f   
\\
=&  \sum_{|\beta| \le   |\alpha| -1, T(\beta ) =T(\alpha) } C_{\alpha, \beta}\hat{Z}^\beta( v_0T_0f  )+  \hat{Z}^\alpha( v_0 T_0 f) + (   v_0 E +v\times   B )\cdot \nabla_v \hat{Z}^\alpha f  
\\
= & \sum_{|\beta| \le  |\alpha|} C_{\alpha, \beta} \hat{Z}^\beta [ v_0 Q_c(f, f) ]  +  \sum_{|\beta| \le   |\alpha| -1, T(\beta ) =T(\alpha) } C_{\alpha, \beta}\hat{Z}^\beta (  (v_0 E + v \times B )\cdot \nabla_v f  )  
\\ +&
\left[(   v_0 E +v\times   B )\cdot \nabla_v \hat{Z}^\alpha f  - \hat{Z}^\alpha (  (v_0 E + v \times B )\cdot \nabla_v f  )    \right]
 \\
=&\textnormal{r.h.s. of}\,\,\eqref{Representation for T F Z alpha f},
\end{align*}
so the proof is thus finished. 
\end{proof}

We also have a lemma on the transport equation.
\begin{lemma}\label{T F g implies g formula}
Let $\mathfrak{g}(t, x, v),\mathfrak h(t, x , v)  $ be two smooth non-negative functions, and there exists a uniform constant $C>0$, such that 
\[
T_F (\mathfrak g)  \le C\pare{\frac {1} {(1+t) \log^2(1+t) }\mathfrak g + \frac {\kappa(v)} {(1+| t-\frac {|x|} c| ) \log^2(1+|t-\frac {|x|} c | ) }\mathfrak g}+  \frac {1} {(1+t) \log^2(1+t) }\mathfrak h,\quad \kappa(v) = 1 - \frac {\hat{v}}{c} \cdot \frac {x} {r}, \quad \forall t, x, v, c
\]
then for any $t\ge 0$, we have 
\begin{equation}
\label{eq:Gronwall}
  \Vert \mathfrak g \Vert_{L^\infty_{t, x, v}} \le   e^{6C}  [\Vert \mathfrak g(0, x, v) \Vert_{L^\infty_{x, v}}  +  3\Vert\mathfrak h \Vert_{L^\infty_{t, x, v}} ].
\end{equation}
\end{lemma}

\begin{proof} The case $c=1$ is proved in \cite{bigorgne2025global}. The proof is similar, we write here for completeness. Remind that $F=E+\frac v{v_0}\times B$. Denote $(X(t), V(t))$ the characteristic of $T_F$ satisfying  
\[
\frac {\dd  } { \dd t }  X(t) = -\hat{V}(t),\quad  \frac {\dd   } { \dd t }  V(t)=-F(t),\quad X(0) = x ,\quad V(0 ) = v.
\]
Then using characteristic method, we can easily have that 
\[
\mathfrak g(t-t', X(t') , V(t')))  = \mathfrak g(t, x, v) - \int_{0}^{t'} T_F(\mathfrak g) (t-z, X(z), V(z) ) \dd z,\quad \forall t' \le t.
\]
Choose $t' =t$, we have 
\[
\mathfrak g(0, X(t) , V(t)))  = \mathfrak g(t, x, v) - \int_{0}^t T_F(\mathfrak g) (t-z, X(z), V(z) ) \dd z.
\]
Denote that 
\[
\phi_1(t' ) = (1+t')^{-1} \log^{-2}  (3+t') ,\quad \phi_2(t') = \kappa (X(t'), V(t')) \left(1+\av{t-t'-\frac {X(t')} {c}}\right)^{-1} \log^{-2}  \left(3+   \av{t-t' -\frac {X(t')} {c}}\right ),
\]
then $\int_{0}^\infty \phi_1(t' ) \dd t' \le 3$, and for any $x, v$, after changing of variable $ \mathpzc s = t- t' - \frac {|X(t')|} {c} $, we have 
\[
\frac {\dd \mathpzc s} {\dd t'}  =  -   1 + \frac {\hat{V}(t') \cdot X(t') } {c |X(t')|} =  - \kappa(X(t'), V(t')) .
\] 
Thus we have that 
\begin{align*}
\int_0^t \phi_2(t')\dd t' = \int_{t-\frac {|x|} c }^{\mathpzc s(t)} \frac 1 {(1+|\mathpzc s|)  \log^2(3+|\mathpzc s|)} \dd \mathpzc s \le  \int_{-\infty }^{\infty} \frac 1 {(1+|\mathpzc s |)  \log^2(3+|\mathpzc s|) }\dd \mathpzc s  \le 6,
\end{align*}
which implies that 
\begin{multline*}
\mathfrak g(t, x, v) \le \Vert\mathfrak g(0, X(t), V(t)) \Vert_{L^\infty_{x, v}} + \Vert\mathfrak h \Vert_{L^\infty_{t, x, v}} \int_0^t \phi_1(t- t') \dd t' + C \int_0^t [\phi_1 (t-t') +\phi_2(t', X(t'), V(t'))  ]  \mathfrak g(t-t', X(t'), V(t'))   \dd t' ,
\end{multline*}
and \eqref{eq:Gronwall} is proved by Gr\"onwall's lemma.
\end{proof}

\section{Bootstrap assumptions and basic convergence rate of distribution function}
\label{sec4}
In this section, we initiate the bootstrap argument by establishing our primary assumptions by leveraging the technical lemmas derived in the preceding sections, and we obtain improved decay and regularity estimates for the distribution function $f$.
\subsection{Bootstrap assumptions}
Based on a standard local well-posedness argument, there exists a unique maximal solution $(f, E, B)$ to the RVMB system \eqref{RVMB1}--\eqref{RVMB2} emanating from the given initial data.\footnote{ If the initial data is of size $\epsilon_0$, then the lifespan of the solution is at least $\epsilon_0^{-1/2}$. For a detailed derivation, we refer the reader to Section 4.5 of \cite{wang2022propagation}.}Let $T_{\text{max}} \in \mathbb{R}^+ \cup \{+\infty\}$ denote the maximal time such that the solution is defined on the interval $[0, T_{\text{max}})$. By continuity, there exists a largest time $T \in [0, T_{\text{max}})$ and a uniform constant $M_1$ such that the following bootstrap assumptions hold for all $(t, x) \in [0, T) \times \mathbb{R}^3$:
\begin{equation}
\tag{BA1}
\label{basic assumption 1}
|Z^\gamma [E, B] |(t, x)\le   \frac { M_1  } {(1 +   t + |x|  ) (1+   |t - \frac {|x|} c| )    },\quad \forall |\gamma| \le N-1,
\end{equation}
and
\begin{equation}
\tag{BA2}\label{basic assumption 2}
|\frac 1 c \partial_t  Z^\gamma [E, B]  |(t, x)  + |\nabla_{x}   Z^\gamma [E, B] |(t, x) \le  \frac { M_1 \log(3+ |t - \frac {|x|} c|  )} {(1+t+|x|   ) (1+ |t - \frac {|x|} c| )^2    },\quad \forall |\gamma| \le N-1,
\end{equation}
as well as 
\begin{equation}
\tag{BA3}\label{basic assumption 3}
\sum_{i=0}^3 |J_i(\hat{Z}^\beta f) |  =  \left |\int_{\R^3}      \hat{Z}^\beta f (t, x, v)  \dd v \right | +  \left |\int_{\R^3}  \hat{v}    \hat{Z}^\beta f (t, x, v)  \dd v \right |\le   \frac { M_1 } {(1+t+|x|  )^3     },\quad \forall |\beta| \le N-2,
\end{equation}
where we recall that $J_i$ is defined in \eqref{definition J 0 J i} and we assume $M_1\defeq C_{\textnormal{bo}} M$, where $C_{\textnormal{bo}}$ is a universal constant to be determined later.
Under these assumptions we first  have that 
\begin{cor}
Under the assumptions \eqref{basic assumption 1}-\eqref{basic assumption 3}, for any $|\gamma| \le N$, we have 
\begin{equation}
\label{rough estimate for Z gamma E B}
| Z^\gamma [E, B] |(t, x) 
 \lesssim  M_1\frac {\log(3+ |t - \frac {|x|} c|  )} {(1+ |t - \frac {|x|} c| )^2    }    \lesssim M_1  \frac {\log(3+ |t - \frac {|x|} c|  )} {(1+ t  +  |x| )^2     }  \langle v \rangle^8 \langle x-t\hat v \rangle^4,
\end{equation}
and for any $|\gamma  |   \le N-1$, 
\begin{equation}
\label{rough estimate for partial x i t E B}
|\partial_{x_i }  Z^\gamma [E, B]   | +  |\frac 1 c \partial_{t }  Z^\gamma [E, B]   |  \lesssim M_1  \frac {\log(3+ |t - \frac {|x|} c|  )}   {(1+ t+  |x|)^3     }  \langle v \rangle^8 \langle x-t\hat v \rangle^4,
\end{equation}
\end{cor}
\begin{proof}
This is the direct result of $|Z| \lesssim  (1+t+|x|) (|\frac 1 c\partial_t |+|\nabla_x |) $ for all $Z$ and \eqref{main inequality from 1 t - r c to 1 + t + r}. 
\end{proof}

\begin{lemma}\label{better estimate for xi F}
Denote that  $F = [E, B]$ or $F = \mathpzc D_Z(E,  B)$ with $Z = S, \Omega_i, \Omega_{ij}$.  Under the assumption \eqref{basic assumption 1}-\eqref{basic assumption 3} , denote also $\xi = \alpha_1, \alpha_2, \rho, \sigma$ , we have a better estimate for $\xi(F)$:
\[
|\xi(F)| \lesssim M_1  \frac {\log (3+t)} { (1+t+|x| ) (1+t+\frac {|x|} {c}) }.
\]
\end{lemma}
\begin{proof} 
First, under the assumption \eqref{basic assumption 1},  if $  \frac {|x|} {c} \ge 2(1+t) $ or $\frac {|x|} {c} \le \frac 1 2(1+t) $, we have 
\[
\sum_{|z| \le 1} |Z^\gamma [E , B ] |  \lesssim  \frac {M_1} {(1 + t +  |x|) (1+t+\frac {|x|} c) }.
\] 
So we suppose $  \frac 1 2 (1+t) \le    \frac {|x|} {c} \le   2(1+t)  $. Denote that $\mathfrak u = t- \frac r c , \mathfrak u'=  t+ \frac r c$, then $t= \frac {\mathfrak u+\mathfrak u'} 2, r= \frac {c(\mathfrak u'-\mathfrak u)} 2$, and we define that 
\[
\phi(\mathfrak u,\mathfrak u', \omega ) := \xi(F) (\frac {\mathfrak u+\mathfrak u'} 2 ,   \frac {c(\mathfrak u'-\mathfrak u)} 2 \omega  ) ,\quad \xi = \alpha_1, \alpha_2, \sigma ,\rho,
\]
where $\phi$ is defined on the region $\mathfrak u   \le \mathfrak u',\mathfrak u' \ge 0$. Then we have that 
\[
\xi(F) (t, r\omega) =  \phi(t-\frac r c , t+\frac r c, \omega) = \phi(-t -\frac r c,  t+\frac r c ,\omega    ) + \int_{-t-\frac r c}^{t-\frac r c} \partial_u \phi ( z, t+\frac r c, \omega )\dd z.
\]
For the first term, using \eqref{initial data requirement} we easily have that 
\[
\phi(-t -\frac r c,  t+\frac r c ,\omega    )  = \xi(F) (0, (  c  t+ r)\omega   )  \lesssim \sum_{|\gamma| \le 1}  |Z^\gamma [E, B] | (0, (ct+r) \omega ) \lesssim M_1   \frac {1} {(1 + t +  |x|) (1+t+\frac {|x|} c)  } ,
\]
for all $z \in [-t-\frac r c,  t-\frac r c  ]$, we have 
$\phi(z, t +\frac r c, \omega) =   \xi(F) (\frac {t +\frac r c +z} {2} , \frac {c(t +\frac r c -z)} {2}    \omega)$.
Since $\frac r c \ge \frac 1 2 (1+t) $, for all $z \in [-t-\frac r c,  t-\frac r c  ]$,
\[
\quad (t +\frac r c -z)  \ge \frac {2r} {c} \ge \frac 1 4 (1+t+\frac {r} {c}) ,\quad c (t +\frac r c -z)  \ge \frac 1 4 c (1+t+\frac {r} {c})\ge \frac 1 4 (1+t+r)
\]
Also using Lemma \ref{partial t - partial r derivative of E B} and Lemma \ref{partial t - partial r derivative of Z E B} we have that 
\begin{align*}
[ \partial_{t} -c \partial_{r} ] \xi (F ) =  \frac {c} {r} \sum_{|\gamma| \le 2} \sum_{k=1}^3 [C_{\gamma, k, 1} (\theta, \phi) Z^\gamma E_k +  C_{\gamma, k, 2} (\theta, \phi)  Z^\gamma B_k ] + c \sum_{|\gamma| \le 1} \sum_{j=0}^3 C_{j, \gamma} (\theta, \phi) J_{j}(\hat{Z}^\gamma f) .
\end{align*}
Since $\partial_{u} \phi = (\partial_{t}  -c \partial_{r} )\xi $, by \eqref{basic assumption 1} and \eqref{basic assumption 3} we compute that 
\begin{align*}
& |\partial_{u}\phi(z, t +\frac r c, \omega)  |     =\left| (\partial_t - c \partial_{r} )\xi(F) (\frac {t +\frac r c +z} {2} , \frac {c(t +\frac r c -z)} {2}    \omega)  \right|
\\
 \lesssim  &(  \frac {c} {c(t +\frac r c -z) }       \sum_{|\beta | \le 2} |\hat{Z}^\beta [E, B] |  (\frac {t +\frac r c +z} {2} , \frac {c(t +\frac r c -z)} {2}    \omega)   + c  \sum_{|\gamma| \le 1 } \sum_{i=0}^3 |J_i(\hat{Z}^\gamma f)  |(\frac {t +\frac r c +z} {2} , \frac {c(t +\frac r c -z)} {2}    \omega)   )   
 \\
 \lesssim  &M_1\frac 1 {(1+t+\frac r c) (1+t+r) } \frac 1 {1+|z|}  + M_1 \frac {c} {(c (1+t+\frac r c))^3 } \lesssim M_1   \frac 1 {(1+t+\frac r c) (1+t+r) } \frac 1 {1+|z|}  +M_1 \frac 1 {(1+t+r)^2(1+t+\frac rc) },
\end{align*}
and since $\frac r c\le 2(1+t )$,  we have 
\begin{align*}
\left | \int_{-t-\frac r c}^{t-\frac r c} \partial_u \phi ( z, t+\frac r c, \omega )\dd z  \right|\lesssim &  \frac {M_1 \int_{-t-\frac r c}^{t-\frac r c}\frac 1 {1+|z| }\dd |z|} {(1+t+\frac r c) (1+t+r) }    + \frac{ M_1} {(1+t+r)^2 } 
\lesssim  \frac {M_1\log (3+t+\frac r c) } {(1+t+\frac r c) (1+t+r) } \lesssim    \frac {M_1\log (3+t) } {(1+t+\frac r c) (1+t+r) }, 
\end{align*}
the lemma is thus proved. 
\end{proof}

From \Cref{inequality for xi partial x i} and \Cref{basic decomposition of the v 0 E v B term} we have the following two colloraries.
\begin{cor}\label{estimate for xi partial x i}
Under the assumption \eqref{basic assumption 1}-\eqref{basic assumption 3}, for any $t \ge 0, x \in \R^3, c \ge 1$, denote $\xi = \alpha_1, \alpha_2, \rho, \sigma$, for $Z_1 =\frac 1 c \partial_t, \partial_{x_i}$  we have
\[
| \xi (   \mathpzc D_{Z_1} [E, B]  )|   \lesssim  M_1 \frac {\log (3+t + \frac {|x|} c )} {(1+t+|x|) (1+t+\frac {|x|} {c}) (1+|t-\frac {|x|} {c}| )}.
\]
 \end{cor}
\begin{proof}The case $|x| \ge 10  c (1+t)$, is direct from Assumption (\ref{basic assumption 1}) and (\ref{basic assumption 2}).  For $|x|  \le 10  c (1+t)$, by Lemma \ref{inequality for xi partial x i} and Lemma \ref{better estimate for xi F} we have that 
\[
| \xi (   \mathpzc D_{Z_1} [E, B]  ) | \lesssim  \frac 1 {1+|t-\frac {|x|} {c}| }  \sum_{|\gamma| \le 1 }| \xi (  \mathpzc D_Z   [E, B])  |+ \frac 1 {1+t+\frac {|x|} {c}}  \sum_{|\gamma| \le 1 }  |Z^\gamma   [E, B]|   \lesssim M_1  \frac {\log (3+t + \frac {|x|} c )} {(1+t+|x|) (1+t+\frac {|x|} {c}) (1+|t-\frac {|x|} {c}| )},
\]
this completes the proof of corollary.
\end{proof}
\begin{cor}
Under the assumption \eqref{basic assumption 1}-\eqref{basic assumption 3}, for $Z = \Omega_i,\Omega_{i, j},S$, we have that 
\begin{equation}
\label{estimate for v 0 E v B}
\frac c {v_0}|E +  \frac {v} {v_0 } \times B |   + \frac c {v_0 }| \mathpzc D^{(1)}_Z(E, B)+ \frac{ v} {v_0} \times \mathpzc D^{(2)}_Z(E,  B)|  \lesssim M  _1 \kappa(v) \frac {1 }{  ( 1+ t+ |x| ) (1+ |t-\frac {|x|} c | ) } + M_1   \frac {\log (3+t) }{  ( 1+ t+ |x| )(1+t+\frac {|x|} c)    }, 
\end{equation}
and for $Z = \partial_{x_i}, \frac 1 c \partial_{t}$ we have that
\begin{equation}
\label{estimate for partial x i v 0 E v B}
\frac c {v_0 }| \mathpzc D^{(1)}_Z(E, B)+ \frac{ v} {v_0} \times \mathpzc D^{(2)}_Z(E,  B)|     \lesssim M_1  \kappa(v)  \frac { \log (3+|t- \frac r c| ) } {(1+t+|x| ) ( 1+ |t- \frac {|x|} c| )^2  } +   M_1  \frac {1  } {(1+t+|x| )  (1+t+\frac {|x|} c)^{\frac 3 2} }.
\end{equation}
\end{cor}
\begin{proof}
Under the assumption \eqref{basic assumption 1}-\eqref{basic assumption 3}, by Lemma \ref{better estimate for xi F}, Lemma \ref{basic decomposition of the v 0 E v B term}, we have that 
\begin{align*}
\frac c {v_0}|E +  \frac {v} {v_0 } \times B |   \lesssim  &M _1 \kappa(v) \frac {1 }{  ( 1+ t+|x| ) (1+ |t-\frac {|x|} c | ) } + M_1 \sqrt{\kappa(v)}  \frac {\log (3+t) }{  ( 1+ t+ |x| )(1+t+\frac {|x|} c)     } 
\\
\lesssim &  M_1  \kappa(v) \frac {1 }{  ( 1+ t+ |x| ) (1+ |t-\frac {|x|} c | ) } +  M_1  \frac {\log (3+t) }{  ( 1+ t+ |x| )(1+t+\frac {|x|} c)      } ,
\end{align*}
and the $\mathpzc D_Z([E, B])$ term is similar. 
For the $\partial_{x_i} $ term, by Corollary \ref{estimate for xi partial x i}, Lemma \ref{basic decomposition of the v 0 E v B term} and Assumption \eqref{basic assumption 2} we have that
\begin{align*}
\frac c {v_0^2 }|v_0 \partial_{x_i}  E +  v \times \partial_{x_i }B    |   \lesssim  &  M _1 \kappa(v)  \frac { \log (3+|t- \frac {|x|} c| ) } {(1+t+|x| ) ( 1+ |t- \frac {|x|} c| )^2  }  +  M_1 \sqrt{ \kappa(v) }  \frac { \log (3+t+\frac {|x|} c) } {(1+t+|x| )   (1+t+\frac {|x|} c) ( 1+ |t- \frac {|x|} c| ) } 
\\
\lesssim &  M_1  \kappa(v)  \frac { \log (3+|t- \frac {|x|} c| ) } {(1+t+|x| ) ( 1+ |t- \frac {|x|} c| )^2  } +  M_1 \frac {\log^2 (3+t+\frac {|x|} c)  } {(1+t+|x| )  (1+t+\frac {|x|} c)^2 }  ,
\end{align*}
and $\frac 1 c \partial_t$ can be proved similarly.  The proof is thus finished since $\mathpzc D_Z([E, B]) = Z[E, B] $ for $Z= \partial_{x_i}, \frac 1 c \partial_t$.
\end{proof}
\subsection{Improvement of the bootstrap assumptions on the distribution function}
We are now ready to improve the bootstrap assumption on the distribution function $f$. We start by proving the following lemma about $\hat Z^\beta f$ in phase space.
\begin{lemma}\label{main estimate for f}
Recall weight function $\mathpzc W$ defined in \eqref{eq:weight}. For any $M_1>0$ fixed under Assumptions \eqref{basic assumption 1} to \eqref{basic assumption 3} ,  for any $t \ge 0, x, v \in \R^3$,  $|\beta| \le N$, for any $k \le 300$, there exist a constant $\epsilon_0\in (0, 1)$ such that for any $\epsilon_1 \in (0, \epsilon_0)$, if the initial data satisfies \eqref{initial data requirement},  we have that 
\[
\mathpzc W^{100}_k  |\hat{Z}^\beta f |    (t, x, v)  \le C_1 \epsilon_1  e^{DM_1} \log^{3H(\beta)   + 3k  }   (3+t) .
\]
where $ C_1, D $ are two universal constant that does not depend on $M_1, \epsilon_1, M , \epsilon_0$, and $\epsilon_0$ satisfies  $\epsilon_0 e^{DM_1} \le 1 $  we also recall that $H(\beta) (T(\beta))$ as the numbers of $S,\hat\Omega_i, \hat\Omega_{i j}$ (respectively $\frac 1c\partial_t, \partial_{x_i}$) composing $\hat Z^\beta$.  
\end{lemma}
\begin{proof}
Denote $k_0=10^5$. We define that 
\begin{align*}
\mathcal E (f)(t, x, v) = &\sum_{|\beta | \le N }     \frac {\langle v \rangle^{4(k_0-20|\beta|)+50  }|\hat{Z}^\beta f |    (t, x, v)   } {  \log^{3H(\beta)   }   (3+t)  } +  \sum_{|\beta | \le N }     \frac {\mathpzc W^{4(k_0-20|\beta|)+50}_{k_0-20|\beta|} |\hat{Z}^\beta f |    (t, x, v)   } {  \log^{3H(\beta)   + 3(k_0-20|\beta|)  }   (3+t)  } 
+ \sum_{|\beta | \le N }     \frac {\mathpzc W^{2(k_0-20|\beta|)+20  }_{k_0-20|\beta|+10}  |\hat{Z}^\beta f |    (t, x, v)   } {  \log^{3H(\beta)   + 3(k_0-20|\beta|+10)  }   (3+t)  },
\end{align*}
We also need to make the priori assumption that 
\begin{equation}
\label{priori estimate for E f}
|\mathcal E (f)   |    (t, x, v) \le  \sqrt{\epsilon_1} (1+t)^{\frac{2+\gamma}{12}} ,
\end{equation}
we will then prove that 
\[
T_F(\mathcal E (f)) \lesssim\frac {M_1} {(1+t) \log^2(1+t) } \mathcal E (f) + M_1 \frac {\kappa(v)} {(1+| t-\frac {|x|} c| ) \log^2(1+|t-\frac {|x|} c | ) } \mathcal E (f) +  \frac {1} {(1+t) \log^2(1+t) } \epsilon_1,
\]
then by Lemma \ref{T F g implies g formula} the theorem is thus proved. 
Since $T_F(\log^{-1} (3+t)) \le 0 $, we compute that for any $p,q\in\mathbb N^+$,
\begin{align*}
T_F    \left ( \frac{\mathpzc W^p_q |\hat{Z}^\beta f |}      {  \log^{3q + 3H(\beta)   }   (3+t)  }       \right  )\le&  p T_F (\langle v \rangle )\frac {\mathpzc W^{p-1}_q  |\hat{Z}^\beta f |   } {  \log^{3q + 3H(\beta)   }   (3+t)  }    + qT_F (\langle z \rangle )  \frac {\mathpzc W^p_{q-1}  |\hat{Z}^\beta f |   } {  \log^{3q + 3H(\beta)   }   (3+t)  }   
+  T_F (\hat{Z}^\beta f  ) \frac {\hat{Z}^\beta f  } {|\hat{Z}^\beta f | } \frac {\mathpzc W^p_q } {  \log^{3q + 3H(\beta)   }   (3+t)  } .
\end{align*}
For the first term, since $\nabla_v( \langle v \rangle) \lesssim 1$,  using \eqref{estimate for v 0 E v B} we have that 
\begin{align*}
\frac 1 {\langle v \rangle } T_F (\langle v \rangle)  = \frac 1 {\langle v \rangle }  | (E+c^{-1} \hat{v} \times B) \cdot     \nabla_v (\langle v \rangle))| \lesssim \frac {c} {v_0} |E+ \frac {v} {v_0}\times B|  \lesssim M_1 \kappa(v) \frac {1 }{  ( 1+ t+ \frac {|x| } {c} ) (1+ |t-\frac {|x|} c | ) } +  M_1 \frac {\log (3+t) }{  ( 1+ t+ \frac {|x| } {c} )^2    } ,
\end{align*}
so the first term is proved. For the second term, recall that $T_0(x-t\hat v) =0$ and $|\nabla_v (\hat{v})  |\lesssim \frac {c} {v_0}$, using \eqref{estimate for v 0 E v B}, we have  
\begin{align*}
|T_F (  \langle x-t\hat v \rangle)| =& |\frac {(x-t\hat v) \cdot  T_F (x-t\hat v)} {\langle x-t\hat v\rangle} |\lesssim |T_F(x-t\hat v)| \lesssim  | E+ \frac {v} {v_0}\times B| |\nabla_v (x-t\hat v)| \lesssim t |\nabla_v (\hat{v}) | | E+ \frac {v} {v_0}\times B|\\  =&\frac {t c} {v_0 } | E+ \frac {v} {v_0}\times B|  \lesssim M_1\kappa(v) \frac {t }{  ( 1+ t+ \frac {|x| } {c} ) (1+ |t-\frac {|x|} c | ) } + M_1  \frac {t\log (3+t) }{  ( 1+ t+ \frac {|x| } {c} )^2    }  .
\end{align*}
For any $a >0$ we use the Young inequality 
\begin{align*}
\frac {\langle x-t\hat v \rangle^{a-1} } {\log^{3a} (3+ t)} \lesssim \frac {\langle x-t\hat v \rangle^{a} } {\log^{3} (3+t) \log^{3a} (3+ t)}  +   \frac {1  } {\log^{3} (3+ t)}  \lesssim  \frac {1  } {\log^{3} (3+ t)}   \left(1+ \frac {\langle x-t\hat v \rangle^{a} } { \log^{3a} (3+ t)} \right),
\end{align*}
so the second term is thus proved. For the third term, the lower order case, remind
$\frac {v_0} {c} \partial_{v_i} = \hat{\Omega}_{i} -t \partial_{x_i } -\frac 1 {c^2} x_i \partial_t$.
Using \eqref{Representation for T F Z alpha f}, the $T_F (  \hat{Z}^\alpha f    ) $ term is bounded by $T_1+T_2+T_3+T_4$, where
\begin{equation}
\begin{aligned}
&T_1 = \sum_{|\beta| \le  |\alpha|} | \frac 1 {v_0}  \hat{Z}^\beta [ v_0 Q_c(f,f) ] |,\\
&T_2 = \sum_{ |\alpha_2| = |\alpha| -1, T(\alpha_2 ) = T(\alpha) }\frac c {v_0} |E+ \frac {v} {v_0} \times B|  \left[ | \hat{\Omega}_{i } \hat{Z}^{\alpha_2} f | +  (t+\frac {|x|} {c})  (| \nabla_{ x} \hat{Z}^{\alpha_2}  f | + | \frac 1 c \partial_{t} \hat{Z}^{\alpha_2}  f  | ) \right],\\
&T_3 = \sum_{\alpha =\alpha_1+  \alpha_2 ,|\alpha_2| = |\alpha| -1 } \frac c {v_0 }  \av{ \mathpzc D^{(1)}_{Z^\alpha_1}  (E, B)+ \frac{ v} {v_0} \times \mathpzc D^{(2)}_{Z^\alpha_1} (E,  B)}   \left[ | \hat{\Omega}_{i } \hat{Z}^{\alpha_2} f | +  (t+\frac {|x|} {c})  (| \nabla_{ x} \hat{Z}^{\alpha_2}  f | + | \frac 1 c \partial_{t} \hat{Z}^{\alpha_2}  f  | ) \right],\\
&T_4= \frac c {v_0} \sum_{\alpha_1+  \alpha_2  \le \alpha, |\alpha_2| \le |\alpha|-2, T(\alpha_1) +T(\alpha_2) = T(\alpha)}  | Z^{\alpha_1} [E,  B]  |    \left[ | \hat{\Omega}_{i } \hat{Z}^{\alpha_2} f | +  (t+\frac {|x|} {c})  (| \nabla_{ x} \hat{Z}^{\alpha_2}  f | + | \frac 1 c \partial_{t} \hat{Z}^{\alpha_2}  f  | ) \right].
\end{aligned}
\end{equation}
\begin{enumerate}
\item  For the term $T_1$, using  \eqref{accumulated chain rule for the Z b Q f f} and by \Cref{L infty estimates for Q f f z version}  we have that for any $\lambda <\frac{5+\gamma}3 $ we have
\begin{align*}
\av{\mathpzc W^{4(k_0-20|\alpha| )+50 }_{k_0- 20|\alpha|}   + \mathpzc W^{2(k_0-20|\alpha|)+20  }_{k_0- 20|\alpha|+10} } T_1 
\lesssim & \sum_{|\alpha_1| +|\alpha_2| \le |\alpha|}
 \av{\mathpzc W^{4(k_0-20|\alpha|)+50  }_{k_0- 20|\alpha|}   + \mathpzc W^{2(k_0-20|\alpha|)+20  }_{k_0+10- 20|\alpha|} } |Q_c(\hat{Z}^{\alpha_1} f,  \hat{Z}^{\alpha_2} f)|   
\\
\lesssim    &     \sum_{|\alpha_1| +|\alpha_2| \le |\alpha||}   (1+t)^{-\lambda }  [    \mathpzc W^{4(k_0-20|\alpha| )+50 }_{k_0- 20|\alpha|}    \hat{Z}^{\alpha_1} f     \Vert_{L^\infty_{x, v}}     + \Vert   \mathpzc W^{2(k_0-20|\alpha|)+20  }_{k_0- 20|\alpha|+10}   \hat{Z}^{\alpha_1}  f      \Vert_{L^\infty_{x, v}}  ] 
\\
&\times [      \Vert  \mathpzc W^{4(k_0-20|\alpha|)+50  }_{k_0- 20|\alpha|}   \hat{Z}^{\alpha_2} f     \Vert_{L^\infty_{x, v}}     + \Vert     \mathpzc W^{2(k_0-20|\alpha|)+20  }_{k_0- 20|\alpha|+10}  \hat{Z}^{\alpha_2} f        \Vert_{L^\infty_{x, v}}  ] 
\\
\lesssim &(1+t)^{-  \lambda   } \Vert \mathcal E^2(f)\Vert_{L^\infty_{t, x, v} } \lesssim \epsilon_1  (1+t)^{-\frac {8+\gamma}6 } .
\end{align*}
\item For the $T_2$ term, first since either $|t-\frac {|x|} c |  \le t$ or $|t-\frac {|x|} c |  \ge t$, we have that 
\[
\frac {1}  { (1+ |t-\frac {|x|} c | ) \log^2(3+t  )  }  \lesssim  \frac 1 {(1+t)\log^2 (3+t)  } +  \frac {1}  { (1+ |t-\frac {|x|} c | ) \log^2(3+ |t-\frac {|x|} c | )  },
\]
and since$|\alpha_2| < |\alpha|$ and $T(\alpha_2) =  T(\alpha)$, we have $H(\alpha_2) <H(\alpha)$, recall that $\kappa(v) \lesssim 1$, using that $H(\hat{\Omega}_{i} \hat{Z}^{\alpha_2} ) = H(\alpha_2)+1 \le H(\alpha), H(\partial_{x_i}\hat{Z}^{\alpha_2} ) =H(\frac 1 c \partial_{t}\hat{Z}^{\alpha_2} )=  H(\alpha_2) \le H(\alpha)-1$ and  \eqref{estimate for v 0 E v B} we have for any $a ,b\ge 0$,
\begin{align*}
&T_2 \frac {\hat{Z}^\alpha f  } {|\hat{Z}^\alpha f | }\frac {\mathpzc W^{a-80|\alpha|  }_{b- 20|\alpha|} } {  \log^{3\pare{b- 20|\alpha|} + 3H(\alpha)  }   (3+t)  } 
\\
\lesssim & M_1 \pare{ \kappa(v) \frac {1 }{  ( 1+ t+ \frac {|x| } {c} ) (1+ |t-\frac {|x|} c | ) } +   \frac {\log (3+t) }{  ( 1+ t+ \frac {|x| } {c} )^2    } }     \pare{ | \hat{\Omega}_{i } \hat{Z}^{\alpha_2}  f | +  (t+\frac {|x|} {c})  (| \nabla_{ x} \hat{Z}^{\alpha_2}  f | + | \frac 1 c \partial_{t}  \hat{Z}^{\alpha_2}  f  | )}  \frac {\mathpzc W^{a-80|\alpha|  }_{b- 20|\alpha|} } {  \log^{3\pare{b- 20|\alpha|} + 3H(\alpha)  }   (3+t)  } 
\\
\lesssim &M_1 \pare{ \kappa(v) \frac {1 }{  ( 1+ t+ \frac {|x| } {c} ) (1+ |t-\frac {|x|} c | ) } +   \frac {\log (3+t) }{  ( 1+ t+ \frac {|x| } {c} )^2    } }   \frac{\mathpzc W^{a-80 (|\alpha_2 | + 1)  } _{b- 20(|\alpha_2 | + 1)  }  }{ \log^{3\pare{b- 20|\alpha|}  }   (3+t)}
\\
&\times \pare{ \frac {| \hat{\Omega}_{i } \hat{Z}^{\alpha_2}  f   | }    {  \log^{ 3(H(\alpha_2 )+1)  }   (3+t) }+ \frac 1 {\log^3 (3+t) }\frac { (t+\frac {|x|} {c})  (| \nabla_{ x} \hat{Z}^{\alpha_2}  f   | +  | \frac 1 c \partial_{t}  \hat{Z}^{\alpha_2}  f   | )    } {  \log^{3\pare{H(\alpha_2)}   }   (3+t)  } }
\\
\lesssim &M_1\mathcal E (f) \pare{ \frac 1 {(1+t)\log^2 (3+t)  } +  \frac {\kappa(v)}  { (1+ |t-\frac {|x|} c | ) \log^2(3+ |t-\frac {|x|} c | )  }   + \frac {\kappa(v)}  { (1+ |t-\frac {|x|} c | ) \log^2(3+t  )  }  }
\\
\lesssim &M_1\mathcal E (f) \pare{ \frac 1 {(1+t)\log^2 (3+t)  } +  \frac {\kappa(v)}  { (1+ |t-\frac {|x|} c | ) \log^2(3+ |t-\frac {|x|} c | )  }     }.
\end{align*}
\item For the term $T_3$ with $\hat{Z}^{\alpha_1}  =S, \hat{\Omega}_i, \hat{\Omega}_{ij}  $, we still have $H(\alpha_2) <H(\alpha)$, so the proof is the same as $T_2$ case. For $\hat{Z}^{\alpha_1} =\frac 1 c \partial_t, \partial_{x_i}$, this time $H(\alpha_2) = H(\alpha)$, using \eqref{estimate for partial x i v 0 E v B} we have that for any $a,b\ge 0$,
\begin{align*}
&T_3 \frac {\hat{Z}^\alpha f  } {|\hat{Z}^\alpha f | } \frac {\mathpzc W^{a-80|\alpha|  }_{b- 20|\alpha|} } {  \log^{3\pare{b- 20|\alpha|} + 3H(\alpha)  }   (3+t)  }   \\
\lesssim & M_1\frac {\mathpzc W^{a-80|\alpha|  }_{b- 20|\alpha|} } {  \log^{3\pare{b- 20|\alpha|} + 3H(\alpha)  }   (3+t)  }    \pare{  \frac {\kappa(v)  \log(3+  |t- \frac {|x|} c|  ) } {(1+t+|x| ) ( 1+  |t-  \frac {|x|} c| )^2 } +    \frac { 1 } { (1 + t + |x| )   ( 1+t+\frac {|x|}c)^{\frac 3 2 }  }     }\\&\times 
\pare{ | \hat{\Omega}_{i } \hat{Z}^{\alpha_2}  f | +  (t+\frac {|x|} {c})  (| \nabla_{ x} \hat{Z}^{\alpha_2}  f | + | \frac 1 c \partial_{t}  \hat{Z}^{\alpha_2}  f  | )}  
\\
\lesssim &M_1\frac{ \mathpzc W^{a-80 (|\alpha_2 | + 1)  } _{b- 20(|\alpha_2 | + 1)  } }{ \log^{3\pare{b- 20|\alpha|}  }   (3+t)}  \pare{   \frac {\kappa(v)  \log(3+  |t- \frac {|x|} c|  ) } {(1+t+|x| ) ( 1+  |t-  \frac {|x|} c| )^2 } +    \frac { 1 } { (1 + t + |x| )   ( 1+t+\frac {|x|} c)^{\frac 3 2 }  }      }    \\
&\times  \pare{ \frac {\log ^3 (3+t) | \hat{\Omega}_{i } \hat{Z}^{\alpha_2}  f   | }    {\log^{ 3(H(\alpha_2 )+1)  }   (3+t) }+ \frac { (t+\frac {|x|} {c})  (| \nabla_{ x} \hat{Z}^{\alpha_2}  f   | + | \frac 1 c \partial_{t}  \hat{Z}^{\alpha_2}  f   | )    } {  \log^{ 3H(\alpha_2 )   }   (3+t)  } }
\\
\lesssim &M_1\mathcal E (f)  \pare{ \frac 1 {(1+t)\log^2 (3+t)  } +  \frac {\kappa(v)}  { (1+ |t-\frac {|x|} c | ) \log^2(3+ |t-\frac {|x|} c | )  }}.
\end{align*}
\item For the $T_4$ term,
since $|\alpha_2| \le |\alpha|-2$, we have that $H(\alpha_2) \le H(\alpha)  -2 $, then using \eqref{rough estimate for Z gamma E B} and $\frac {\log (3+t+|x|) } {1+t+|x|} \lesssim \frac {\log (3+t) } {1+t}$ and using that $H(\hat{\Omega}_{i} \hat{Z}^{\alpha_2} ) = H(\alpha_2)+1 \le H(\alpha)-1, H(\partial_{x_i}\hat{Z}^{\alpha_2} ) =H(\frac 1 c \partial_{t}\hat{Z}^{\alpha_2} )=  H(\alpha_2) \le H(\alpha)-2$ we have that for any $a,b\ge 0$,
\begin{align*}
 &T_4 \frac {\hat{Z}^\alpha f  } {|\hat{Z}^\alpha f | } \frac {\mathpzc W^{a-80|\alpha|  }_{b- 20|\alpha|} } {  \log^{3\pare{b- 20|\alpha|} + 3H(\alpha)  }   (3+t)  }  \\
\lesssim &    M_1 \frac {\log(3+ |t - \frac {|x|} c|  )} {(1+ t+  |x| )^2     } \mathpzc W^8_4   \pare{ | \hat{\Omega}_{i } \hat{Z}^{\alpha_2}  f | +  (t+\frac {|x|} {c})  (| \nabla_{ x} \hat{Z}^{\alpha_2}  f | + | \frac 1 c \partial_{t}  \hat{Z}^{\alpha_2}  f  | )}\frac {\mathpzc W^{a-80|\alpha|  }_{b- 20|\alpha|} } {  \log^{3\pare{b- 20|\alpha|} + 3H(\alpha)  }   (3+t)  } 
\\
\lesssim &M_1\frac {\log(3+ |t - \frac {|x|} c|  )} {(1+ t+  |x| )^2     } \frac 1 {\log^3 (3+t)} \frac{\mathpzc W^{a- 80 (|\alpha_2 | +1  )} _{b- 20 (|\alpha_2 | +1  )   } } {\log^{3\pare{b- 20|\alpha|}}  (3+t) }  \pare{ \frac {| \hat{\Omega}_{i } \hat{Z}^{\alpha_2}  f   | }    {  \log^{ 3(H(\alpha_2 )+1)  }   (3+t) }+ \frac { (t+\frac {|x|} {c})  (| \nabla_{ x} \hat{Z}^{\alpha_2}  f   | + | \frac 1 c \partial_{t}  \hat{Z}^{\alpha_2}  f   | )    } {  \log^{  3(H(\alpha_2 )+1)   }   (3+t)  } } 
\\
\lesssim &M_1\mathcal E (f) \frac 1 {(1+t)\log^2 (3+t)  }.
\end{align*}
The similar estimate holds if we replace $ \mathpzc W^{a-80|\alpha|  }_{b- 20|\alpha|}$ by $\mathpzc W^{a-40|\alpha|  }_{b- 20|\alpha|}$.
\end{enumerate}
Combine the estimates above together, by using Lemma \ref{T F g implies g formula} we get our result, which also close the priori assumption \eqref{priori estimate for E f} by assuming \begin{equation}
\label{extension1}
C_1\sqrt{\epsilon_0} e^{DM_1} \le \frac 1 2 ,
\end{equation} 
here $C$ is a universal constant, and the proof is thus finished. 
\end{proof}

We note that in the following two sections we also need to assume \eqref{extension1} holds to use Theorem \ref{main estimate for f}.  Next we come to prove that

\begin{lemma}
\label{estimate for f dx}
For any $|\alpha| \le N-1$, for any $t\in \R^+, v \in \R^3$, we have 
\begin{equation}
\label{eq:betterestimate1}
\langle v \rangle^{30}   |\partial_t \int_{\R^3} \hat{Z}^\alpha f(t, x, v)  
\dd x| \lesssim  \epsilon_1 M_1 e^{DM_1}   (1+t)^{-\frac {8+\gamma} 6} .
\end{equation}
As a consequence, 
\begin{equation}
\label{eq:betterestimate2}
\langle v \rangle^{30}   | \int_{\R^3} \hat{Z}^\alpha f(t, x, v)  \dd x| \lesssim \epsilon_1 M_1 e^{DM_1} .
\end{equation}

\end{lemma}

\begin{proof} First, recall that $T_0 = \partial_t + \hat{v } \cdot \nabla_x $, from integration by parts, we have 
\[
\partial_t \int_{\R^3} \hat{Z}^\alpha f(t, x, v)  \dd x =\int_{\R^3} T_0 \hat{Z}^\alpha f(t, x, v)  \dd x -  \int_{\R^3} \hat{v } \cdot \nabla_x \hat{Z}^\alpha f(t, x, v)  \dd x = \int_{\R^3} T_0 (\hat{Z}^\alpha f)(t, x, v)  \dd x ,
\]
and using \eqref{Representation for T 0 Z alpha f} and \eqref{accumulated chain rule for the Z b Q f f} we have that 
\begin{align}
\nonumber
 v_0T_0(\hat{Z}^\alpha f)
=& \sum_{|\beta| \le  |\alpha|} C_{\alpha, \beta}    \hat{Z}^\beta [v_0 Q_c(f,f)   ]
+\sum_{i,  j, k=1}^3 \sum_{\alpha_1 +\alpha_2 \le \alpha} ( Z^{\alpha_1} (C_1 E_k  +C_2B_k) v_0\partial_{v_j } \hat{Z}^{\alpha_2} f 
+   Z^{\alpha_1} (C_3 E_k  +C_4B_k)   (v_i \partial_{v_j }  -v_j \partial_{v_i} )\hat{Z}^{\alpha_2} f .
\end{align}
Using that 
\[
\frac {v_0} {c} \partial_{v_i} = \hat{\Omega}_{i} -t \partial_{x_i } -\frac 1 {c^2} x_i \partial_t = \hat{\Omega}_{i} -t \partial_{x_i } -\frac 1 {c^2} (x_i  - \hat{v}_i t )\partial_t  - \frac 1 {c} \hat{v}_i S + \frac 1 {c^2} \hat{v}_i \sum_{j} x_j \partial_{x_j} ,
\]
then using integration by parts in $x$ for $\partial_{x_i}$, we have
\begin{equation}
\label{eq:better0}
\begin{aligned}
|\partial_t \int_{\R^3} \hat{Z}^\alpha f(t, x, v)  \dd x |  \lesssim& \sum_{|\alpha_1|+|\alpha_2|\le |\alpha|+1,| \alpha_1| \le | \alpha| } \left|  \int_{\R^3} \langle v\rangle^2 |  Z^{\alpha_1}[E, B]  (t, x) |   |\langle x-t\hat v  \rangle \hat{Z}^{\alpha_2}  f |   (t, x, v) \dd x  \right|
\\
& +  \sum_{|\alpha_1|+|\alpha_2|\le |\alpha| } \left|    \int_{\R^3} (t+\frac {|x|} c) | \nabla_{x}  Z^{\alpha_1}[E, B]  (t, x) |    |\hat{Z}^{\alpha_2}  f |   (t, x, v) \dd x  \right|\\
&+  \sum_{|\alpha_1|+|\alpha_2|\le |\alpha|}  \left|  \int_{\R^3} |Q_c(\hat{Z}^{\alpha_1} f,  \hat{Z}^{\alpha_2} f)|(t, x, v)   \dd x\right|\defeq T_1+T_2+T_3.
\end{aligned}
\end{equation}
Using Lemma \ref{L infty estimates for Q f f z version} and Lemma \ref{main estimate for f} we have that 
\begin{equation}
\label{eq:better0.5}
\mathpzc W^{30}_{10} |Q_c(\hat{Z}^{\alpha_1} f,  \hat{Z}^{\alpha_2} f)|(t, x, v)  \lesssim (1+t)^{-\frac{5+\gamma} 3  }  \Vert \mathpzc W^{50}_{50} \hat{Z}^{\alpha_1} f  \Vert_{L^\infty_{t, x, v} } \Vert  \mathpzc W^{50}_{50}  \hat{Z}^{\alpha_2} f  \Vert_{L^\infty_{t, x, v} } \lesssim (1+t)^{-\frac {8+\gamma } 6 } \epsilon_1 e^{DM_1} ,
\end{equation}
together with that $\int_{\R^3}   \langle x - t \hat{v} \rangle^{-5} dx \lesssim 1$, we have
\begin{equation}
\label{eq:better1}
T_3\lesssim (1+t)^{-\frac {8+\gamma } 6 } \epsilon_1 e^{DM_1} .
\end{equation}
Recall that $|\alpha|\le N-1$, for  $T_1$ we use \eqref{rough estimate for Z gamma E B}, for $T_2$ we use \eqref{rough estimate for partial x i t E B}, and by Lemma \ref{main estimate for f} we compute that 
\begin{equation}
\label{eq:better2}
\begin{aligned}
T_1+T_2&\lesssim \epsilon_1M_1 e^{DM_1} \av{\int_{\mathbb R^3}\frac{\log\pare{3+\av{t-\frac{\av x}{c}}}}{(1+t+|x|)^2}
\frac{\log^{1000} (3+t)}{\mathpzc W^{20}_{10}} \dd x}\\&\lesssim (1+t)^{-\frac{5+\gamma} 3} \epsilon_1M_1 e^{DM_1}\left|  \int_{\R^3} \langle x-t\hat v \rangle^{-10}  \dd x\right| \lesssim (1+t)^{-\frac {5+\gamma}3}\epsilon_1M_1 e^{DM_1}.
\end{aligned}
\end{equation}
We plug \eqref{eq:better1} and \eqref{eq:better2} into \eqref{eq:better0} to obtain \eqref{eq:betterestimate1}. To show \eqref{eq:betterestimate2},
we integrate  from $0$ to $t$ and obtain that 
\[
\langle v \rangle^{30}   | \int_{\R^3} \hat{Z}^\alpha f(t, x, v)  \dd x| \lesssim \langle v \rangle^{30}   | \int_{\R^3} \hat{Z}^\alpha f(0, x, v)  \dd x|  +  \int_0^t  \langle v \rangle^{30}   |\partial_t  \int_{\R^3} \hat{Z}^\alpha f(t', x, v)  \dd x|  \dd t'  \lesssim \epsilon_0  +\epsilon_1M_1 e^{DM_1},
\]
and the lemma is thus proved. 
\end{proof}

\begin{lemma}
\label{lem:betterdecayrate}
Let $N_x>7$. For any $|\beta| \le N$, for any $t \in [0, T], x \in \R^3$, we have that 
\begin{equation}\label{basic estimate for f dv basic version}
\left| \int_{\R^3} 
\mathpzc W^{20}_{10}|\hat{Z}^\beta f |(t, x , v) \dd v \right|    \lesssim \epsilon_1e^{DM_1} \frac {\log^{100000}(  3+t) } {(1+t + |x| )^{3}}.
\end{equation}
Moreover we have a better decay depending on different cases 
\begin{equation}\label{basic estimate for f dv improved version x le c t1}
\forall |x| \le c t, \quad 
\int_{\R^3} 
\mathpzc W^{20}_{10}|\hat{Z}^\beta f |(t, x , v) \dd v \lesssim \epsilon_1e^{DM_1} \frac {\log^{3N_x+3N}(1+t)   (1 +  |t-\frac {|x|} {c} | )^2    } {(1+t+  |x| )^{   5 }  },
\end{equation}
and 
\begin{equation}\label{basic estimate for f dv improved version x ge c t2}
\forall |x| \ge c t, \quad 
\int_{\R^3} 
\mathpzc W^{20}_{10}|\hat{Z}^\beta f |(t, x , v) \dd v \lesssim  \epsilon_1e^{DM_1} \frac {1} {(1+ t + |x| )^{8}}.
\end{equation}
\end{lemma}
\begin{proof}
First, to prove \eqref{basic estimate for f dv basic version}, we have
\begin{equation}\label{basic estimate for f dv improved version x ge c t-1}
\begin{aligned}
\left| \int_{\R^3} 
\mathpzc W^{20}_{10}   |\hat{Z}^\beta f |(t, x , v) \dd v \right|  \overset{\eqref{main inequality from 1 t - r c to 1 + t + r-2}}{\lesssim} & \frac {(1+t)^3} {(1+t+|x| )^3   }    \left| \int_{\R^3} 
\mathpzc W^{40}_{20}  |\hat{Z}^\beta f |(t, x , v) \dd v \right|  
\\
\overset{\textnormal{Lemma} \ref{L 1 L infty estimate on x t v}}{\lesssim} & \frac {1} {(1+t+|x| )^3   }  \Vert 
\mathpzc W^{50}_{30} |\hat{Z}^\beta f | (t, x, v) \Vert_{L^\infty_{x, v}}     \overset{\textnormal{Lemma} \ref{main estimate for f}} {\lesssim}  \epsilon_1e^{DM_1} \frac {\log^{100000}(  3+t) } {(1+t + |x| )^{3}}.
\end{aligned}
\end{equation}
If $|x| \le c t $, using \eqref{main inequality from 1 t - r c to 1 + t + r} and similarly as \eqref{basic estimate for f dv improved version x ge c t-1},
we have that
\begin{align*}
\av{\int_{\R^3}
\mathpzc W^{20}_{10}|\hat{Z}^\beta f |(t, x , v) \dd v }&\lesssim  \frac {(1 +  |t-\frac {|x|} {c} | )^5   } {(1+t+|x| )^5   }    \left| \int_{\R^3}
\mathpzc W^{40}_{20} |\hat{Z}^\beta f |(t, x , v) \dd v \right|  \\
&\lesssim  \frac {(1 +  |t-\frac {|x|} {c} | )^2} {(1+t+|x| )^5  }  \Vert 
\mathpzc W^{50}_{30}  |\hat{Z}^\beta f | (t, x, v) \Vert_{L^\infty_{x, v}}      {\lesssim}  \epsilon_1e^{DM_1} \frac {\log^{10000}(1+t)   (1 +  |t-\frac {|x|} {c} | )^2    } {(1+t+  |x| )^{   5}  }.
\end{align*}
If $|x| \ge ct$, using \Cref{lem:xgect} we have that
\[
\av{\int_{\R^3} 
\mathpzc W^{20}_{10}  |\hat{Z}^\beta f |(t, x , v) \dd v} \lesssim \frac {\epsilon_1e^{DM_1}} {(1+t+|x| )^8   }    \left| \int_{\R^3} 
\mathpzc W^{40}_{20} |\hat{Z}^\beta f |(t, x , v) \dd v \right|     \lesssim \epsilon_1e^{DM_1} \frac {1} {(1+ t + |x| )^{8}},
\]
the proof is thus finished. 
\end{proof}

\begin{cor}\label{CORbasic estimate for Q f f dv}
For any $|\beta| \le N$, for any $t \in [0, T], x \in \R^3$, we have that 
\begin{equation}
\label{eq:basic estimate for Q f f dv-1}
\left| \int_{\R^3} \frac {\langle v\rangle^{40} }{v_0} \left  |\hat{Z}^\beta  [v_0 Q_c(f, f)]   \right  | (t, x , v) \dd v \right|    \lesssim \epsilon_1e^{DM_1}\frac {1 }  {(1+t + |x| )^{4+\frac{2+\gamma}{12}}}.
\end{equation}

\begin{proof} Using \Cref{cor:upper bound estimate for Q + -} and  \eqref{accumulated chain rule for the Z b Q f f} we have for any $p>\frac 3{3+\gamma}$,
\begin{equation*}
\textnormal{l.h.s. of} \,\, \eqref{eq:basic estimate for Q f f dv-1}
\lesssim  (1+t)^{-3-\frac 3p} \sum_{|\beta_1| +|\beta_2| \le |\beta|}     \Vert  \mathpzc W^{k+55}_4 |\hat{Z}^{\beta_1}f | (t, x, v)  \dd v \Vert_{L^\infty_v}  \Vert \mathpzc W^{k+50}_4 |\hat{Z}^{\beta_2}f | (t, x, v)  \dd v \Vert_{L^\infty_v}.
\end{equation*}
If $|x|\le ct$, from \Cref{main inequality from 1 t - r c to 1 + t + r}, we have
\begin{align*}
& (1+t)^{-3-\frac 3p} \sum_{|\beta_1| +|\beta_2| \le |\beta|}     \Vert  \mathpzc W^{k+55}_4 |\hat{Z}^{\beta_1}f | (t, x, v)  \dd v \Vert_{L^\infty_v}  \Vert \mathpzc W^{k+50}_4 |\hat{Z}^{\beta_2}f | (t, x, v)  \dd v \Vert_{L^\infty_v}\\
\lesssim & \frac{\pare{1+t-\frac{|x|}c}^{3+\frac 3p}}{(1+t)^{3+\frac 3p}(1+t+|x|)^{3+\frac 3p}} \sum_{|\beta_1| +|\beta_2| \le |\beta|}     \Vert  \mathpzc W^{k+71}_{12} |\hat{Z}^{\beta_1}f | (t, x, v)  \dd v \Vert_{L^\infty_v}  \Vert \mathpzc W^{k+66}_{12} |\hat{Z}^{\beta_2}f | (t, x, v)  \dd v \Vert_{L^\infty_v}\\
\lesssim &\frac {\epsilon_1e^{DM_1}}{ (1+t+|x|)^{4+\frac{2+\gamma}2}}\cdot  (1+t)^{\frac{2+\gamma}3}\lesssim\frac {\epsilon_1e^{DM_1} }  {(1+t + |x| )^{4+\frac{2+\gamma}{12}}}.
\end{align*}
If $|x|\ge ct$, from \Cref{lem:xgect}, we have
\begin{align*}
& \sum_{|\beta_1| +|\beta_2| \le |\beta|}     \Vert  \mathpzc W^{k+55}_4 |\hat{Z}^{\beta_1}f | (t, x, v)  \dd v \Vert_{L^\infty_v}  \Vert \mathpzc W^{k+50}_4 |\hat{Z}^{\beta_2}f | (t, x, v)  \dd v \Vert_{L^\infty_v}\\
\lesssim & (1+t+|x|)^{-5} \sum_{|\beta_1| +|\beta_2| \le |\beta|}     \Vert  \mathpzc W^{k+65}_{9} |\hat{Z}^{\beta_1}f | (t, x, v)  \dd v \Vert_{L^\infty_v}  \Vert \mathpzc W^{k+60}_{9} |\hat{Z}^{\beta_2}f | (t, x, v)  \dd v \Vert_{L^\infty_v}\\
\lesssim & (1+t+|x|)^{-5}\cdot \epsilon_1e^{DM_1} (1+t)^{\frac{2+\gamma}3}\lesssim  \epsilon_1e^{DM_1} \frac {1 }  {(1+t + |x| )^{4+\frac{2+\gamma}{12}}}.
\end{align*}
So the proof is thus finished. 
\end{proof}

\end{cor}

\begin{lemma} For any smooth function $h$, for any $t \ge 0, x, v \in \R^3, c \ge 1$ satisfy that $|x| \le ct$, we have that 
\begin{multline}
\label{eq:smoothg1}
    \left |\int_{\R^3} t^3 h(t, x -t \hat{v}   , v) \dd v  - \int_{\R^3}  \frac {v_0^5} {c^5} h(t, y, \frac {\check{x} }t  ) \dd y  \right| \lesssim    \frac {1} {t} \sup_{(y, v) \in \R^3 \times \R^3}  \left (  |\frac {v_0} {c}|^{10}  \langle y \rangle^{10} |h|  \right)   (t, y, v) +  \left ( |\frac {v_0} {c}|^{10}   \langle y \rangle^{10} |  \right)\nabla_v h|(t, y, v).
\end{multline}
\end{lemma}
\begin{proof}
Recall \eqref{change of variable x - t hat v}, we have
\begin{align*}
\textnormal{l.h.s. of} \,\, \eqref{eq:smoothg1}
= &  \pare{\int_{|y-x| < c t}  \frac {v_0^5} {c^5}  h (   t, y, \widecheck{\frac {x-y} {t}} ) \dd y -   \int_{|x-y| \le ct}  \frac {v_0^5} {c^5}  h  (t, y, \frac {\check{x} }t )  \dd y}  -   \int_{|x-y| \ge c t }  \frac {v_0^5} {c^5}  h (t, y, \frac {\check{x} }t  ) \dd y :=   I_1 +I_2.
\end{align*}
For the term  $I_1$, we have
\[
|\partial_{y_j}  \check{y}_i | = \left| \frac {\delta_{ij}}  {   \sqrt{1-\frac {|y|^2}{c^2 } }  }  +  \frac {y_i y_j}  {c^2 (1-\frac {|y|^2}{c^2 })^{3/2} } \right |\lesssim    \frac {1} {(1-\frac {|y|^2}{c^2 })^{3/2} }  = \pare{ \frac {c^2 + \check{y}^2 } {c^2} }^{\frac 3 2 } = \pare{\frac {v_0} {c} }^3 (\hat{y}).
\]
From mean value theorem, we have that 
\begin{equation}
\label{eq:I1}
\begin{aligned}
I_1&\lesssim  \int_{|x-y|\le c t}\langle y\rangle^{-4} \dd y\cdot \frac {1} {t} \sup_{(y, v) \in \R^3 \times \R^3} \left(  |\frac {v_0} {c}|^{10}  \langle y \rangle^{10} |h |  \right) (t, y, v) + \left (   |\frac {v_0} {c}|^{10}   \langle y \rangle^{10} | \nabla_v h |  \right)    (t, y, v) 
\\
&\lesssim \frac {1} {t} \sup_{(y, v) \in \R^3 \times \R^3}  \left ( |\frac {v_0} {c}|^{10}  \langle y \rangle^{10} |h|  \right)  (t, y, v) + \left (  |\frac {v_0} {c}|^{10}   \langle y \rangle^{10} | \nabla_v h |   \right) (t, y, v).
\end{aligned}
\end{equation}
For $I_2$, denote $v =\frac {\check{x} }t $. From $|x| \le ct$, we have 
\[
c^2 = v_0^2(1- \frac {|x|^2 } {c^2t^2}) =  v_0^2(\frac {c^2 t^2 -  |x|^2 } {c^2t^2})  \le   v_0^2(\frac {|y| (c t+|x|) } {c^2t^2}) \le \frac {2 v_0^2|y| } { c t},
\]
which implies that 
\begin{equation}
\label{eq:I2}
I_2 \lesssim \frac{1}{ct} \sup_{y,v\in\mathbb R^3\times\mathbb R^3} \left ( {  \av{\frac{v_0}c}^{10}\langle y\rangle ^{10}\av  h }    \right)
(t,y,v)\int_{|x-y|\ge ct}\langle y\rangle ^{-4} \dd y \lesssim \frac{1}{ct}    \sup_{y,v\in\mathbb R^3\times\mathbb R^3}   \left (   {\av{\frac{v_0}c}^{10}\langle y\rangle ^{10}\av h}  \right) (t,y,v) ,
\end{equation}
and the lemma is thus proved after gathering \eqref{eq:I1} and \eqref{eq:I2}. 
\end{proof}

\begin{lemma}\label{basic lemma for x z f  h x+vt}
 Let $h (t, x, v)= f (t, x+t \hat{v}, v), z_i   (t,x,v)=x_i-t\hat  v_i$, then 
\[
|\langle v \rangle^b \langle x \rangle^{a} h(t, x, v) | \lesssim\langle v \rangle^b ( | \langle    z \rangle^a f | ) (t, x+t \hat{v}, v) ,\quad \frac {v_0} {c}  |\nabla_v h (t, x, v) | \lesssim  \sum_{|\gamma|=1} ( | \langle z \rangle \hat{Z}^\gamma  f| ) (t, x+t \hat{v}, v).
\]
\end{lemma}
\begin{proof}
The first inequality just follows by $  z_i(t, x+t \hat{v}, v) = x_i  + t \hat{v}_i - t \hat{v}_i = x_i$. For the second, we compute that 
\[
\frac {v_0} {c} \partial_{v_j } g  = \frac {v_0} {c} \partial_{v_j } f +t \partial_{x_j} f - t\sum_{i} \frac {v_i v_j } {v_0^2} \partial_{x_i}  f,
\]
then using that $\frac {v_0} {c} \partial_{v_j }  = \hat{\Omega}_{ j }  -t \partial_{x_j}   -\frac 1 {c^2} x_j \partial_t$ and $\hat{v} = \frac {cv} {v_0}$, we have that 
\begin{align*}
&\frac {v_0} {c} \partial_{v_j }  +t \partial_{x_j}  - \sum_{i} \frac {v_i v_j } {v_0^2} \partial_{x_i}   = \hat{\Omega}_{ j }  -\frac 1 {c^2} x_j \partial_t  - t\sum_{i} \frac {v_i v_j } {v_0^2} \partial_{x_i} 
= \hat{\Omega}_{ j }  - \frac 1 c (x_j - t \hat{v}_j ) \frac 1 c \partial_t   - \frac 1 {c^2} \hat{v}_j t \partial_t- \frac t {c^2} \sum_{i} \hat{ v}_i \hat{v}_j \partial_{x_i} 
\\
= &\hat{\Omega}_{ j }- \frac 1 c (x_j - t \hat{v}_j ) \frac 1 c \partial_t  - \frac 1 {c} \hat{v}_j S + \sum_{i} \frac 1 {c^2} \hat{v}_j x_i \partial_{x_i}  - \frac t {c^2} \sum_{i} \hat{ v}_i \hat{v}_j \partial_{x_i} 
=  \hat{\Omega}_{ j }   - \frac 1 c (x_j - t \hat{v}_j ) \frac 1 c \partial_t  - \frac 1 {c} \hat{v}_j S + \frac {  \hat{v}_j } {c^2}\sum_{i}  (x_i - t \hat{v}_i )\partial_{x_i}  ,
\end{align*}
so the proof follows. 
\end{proof}  
\begin{lemma}\label{main improved estimate for f dv}
For any $|\alpha| \le N-1$, 
and $\Psi (\omega, v), \omega \in \mathbb{S}^2$ is sufficient smooth and satisfies that  
\begin{equation}
\label{eq:Phi}
\Vert\Psi(\omega, v) \Vert_{L^\infty_\omega} + \Vert \nabla_v \Psi(\omega, v) \Vert_{L^\infty_\omega} \lesssim \langle v \rangle^{10},
\end{equation}
then there exists a universal constant $C_4$, such that for any $t\in \R^+, x \in \R^3$,
\begin{equation}
\label{improvement assumption 3}
| \int_{\R^3}  \Psi(\omega, v)   \hat{Z}^\alpha f(t, x, v)  \dd v| \le C_4 \epsilon_1 M_1 e^{DM_1}  \frac 1  { (1+|t| +|x|    )^3}.
\end{equation}
\end{lemma}
\begin{proof}
First we have if $t\le \max\pare{1,\frac{\av x}c}$, then the result follows by \eqref{basic estimate for f dv basic version} and \eqref{basic estimate for f dv improved version x ge c t2}. For the case $|x|\le c t$, if $\log^{3N_x+3N}(1+t)(1+|t-\frac{\av x}{c}|)^2\le (1+t+|x|)^2$, then result  follows by \eqref{basic estimate for f dv improved version x le c t1}.
we now come to prove the case \[t\ge \max\pare{1,\frac{\av x}c}, \quad\log^{3N_x+3N}(1+t)(1+|t-\frac{\av x}{c}|)^2\ge (1+t+|x|)^2.\]
Then we have
\begin{equation}
\label{eq:nottoforget}
1\lesssim \frac{\log^{\frac 92\pare{N_x+N}}(1+t)(1+|t-\frac{\av x}c|)^3}{(1+t+\av x)^3} \lesssim \frac{\log^{\frac 92\pare{N_x+N}}(1+t)(1+t)^3}{(1+t+\av x)^3}.
\end{equation}
Denote that $g(t, x, v)=   \Psi(\omega, v) \hat{Z}^\alpha f(t, x+t\hat{v}, v)$. Since $|\alpha| \le N-1$, first by Lemma \ref{basic lemma for x z f  h x+vt} and Lemma \ref{main estimate for f} we  have that 
\begin{equation}
\label{eq:nottoforget1}
\begin{aligned}
&\sup_{(y, v) \in \R^3 \times \R^3} \left(  |\frac {v_0} {c}|^{10}  \langle y \rangle^{10} |g | \right)(t, y, v) + \left( |\frac {v_0} {c}|^{10}   \langle y \rangle^{10} | \nabla_v g | \right)(t, y, v)
\\
\lesssim&\sup_{(y, v) \in \R^3 \times \R^3}\pare{ |\Psi(\omega, v) | + |\nabla_v\Psi(\omega, v)| }\langle v \rangle^{20} \bkk{x-\hat vt}^{15} \sum_{\gamma \le 1} |Z^\gamma Z^\alpha f(t, x, v)| \lesssim \epsilon_1e^{DM_1} (1+t)^{\frac{2+\gamma}{12}}.
\end{aligned}
\end{equation}
Then from \eqref{eq:smoothg1}, \eqref{eq:nottoforget} and  \eqref{eq:nottoforget1}, we have 

\begin{align*}
&\left|  t^3 \int_{\R^3}  \Psi(\omega, v)   \hat{Z}^\alpha f(t, x, v)  \dd v \right| 
= \left| t^3\int_{\R^3}  g (t, x-t \hat{v}, v)  \dd v| \right| 
\\
\lesssim & \left| \int_{\R^3} ( \frac {v_0^5} {c^5}  g ) (t, y, \frac {\check{x} }t  ) \dd y   \right| +   \frac {(1+t)^3} {(1+t+|x|)^3 }  \frac {\log^{\frac 92\pare{N_x+N}}(1+t)} {t} \left[ \sup_{(y, v) \in \R^3 \times \R^3}   \left ( |\frac {v_0} {c}|^{10}  \langle y \rangle^{10} |g |  \right)  (t, y, v) +  \left( |\frac {v_0} {c}|^{10}   \langle y \rangle^{10} | \nabla_v g |  \right)   (t, y, v) \right]
\\
\lesssim & \left| \int_{\R^3} ( \frac {v_0^5} {c^5}  \Psi(\omega, v) \hat{Z}^\alpha f ) (t, y+x, \frac {\check{x} }t  ) \dd y   \right|+\epsilon_1e^{DM_1} \frac {(1+t)^3} {(1+t+|x|)^3 } 
\\
\lesssim   &  \epsilon_1 M_1 e^{DM_1}  \langle \frac {\check{x} }t \rangle^{-10}  +\epsilon_1e^{DM_1}\frac {(1+t)^3} {(1+t+|x|)^3 }\lesssim   \epsilon_1 M_1 e^{DM_1} \langle \frac {|x| } t \rangle^{-10}  +\epsilon_1e^{DM_1}\frac {(1+t)^3} {(1+t+|x|)^3 } \lesssim \epsilon_1 M_1 e^{DM_1}\frac {(1+t)^3} {(1+t+|x|)^3 },
\end{align*}
so the lemma is thus proved. 
\end{proof}

Next we come to give a better estimate for $\frac 1 c\partial_t ,\partial_x $.
\begin{lemma}\label{better estimate for partial x i f dv}
For any $|\alpha| \le N-1$,  and $\Psi (\omega, v), \omega \in \mathbb{S}^2$ is sufficient regular and satisfies \eqref{eq:Phi}.  
Then for any $t\in R^+, v \in \R^3$, we have 
\[
\left | \int_{\R^3}  \Psi(\omega, v)  \frac 1  c \partial_{t } \hat{Z}^\alpha f(t, x, v)  \dd v \right| + \left | \int_{\R^3}  \Psi(\omega, v)  \partial_{x_i } \hat{Z}^\alpha f(t, x, v)  \dd v \right| \lesssim \epsilon_1e^{DM_1}  \frac 1  { (1+|t| +|x|   )^{4_{-}}}.
\]
\end{lemma}
\begin{proof}
This is obvious by \eqref{basic estimate for f dv improved version x le c t1} and \eqref{basic estimate for f dv improved version x ge c t2} if $ \av x \ge ct$ or $ |t- \frac {|x|} {c}| \le 1$, so we only need to prove the case $t- \frac {\av x} {c} \ge 1$. Using \eqref{basic representation for partial t partial x i} we have that 
\begin{equation}
\label{eq:integration1}
(t-\frac {|x|} {c} )\left[\left | \int_{\R^3}  \Psi(\omega, v)  \frac 1  c \partial_{t } \hat{Z}^\alpha f(t, x, v)  \dd v \right|  +  | \left | \int_{\R^3}  \Psi(\omega, v)  \partial_{x_i } \hat{Z}^\alpha f(t, x, v)  \dd v \right| \right] \lesssim \sum_{|\gamma| =1} \left | \int_{\R^3}  \Psi(\omega, v)  Z^\gamma \hat{Z}^\alpha f(t, x, v)  \dd v \right|,
\end{equation}
and we need to recall that $Z = \hat{Z}$ if $Z =\frac  1 c \partial_t, \partial_{x_i}, S$ and $Z = \hat{\Omega}_{i} - \frac {v_0}{c} \partial_{v_j}$, if $Z= \Omega_{i}$ and $Z = \hat{\Omega}_{i j}  - v_i\partial_{v_j} + v_j \partial_{v_i}$ if $Z = \Omega_{ij}$, so we can just finish the proof by integration by parts in $v$ and \eqref{basic estimate for f dv improved version x le c t1}. 
\end{proof}

\begin{cor}\label{basic estimate for Q f f dv}
For any $|\alpha| \le N-1$,  and $\Phi (\omega, v)$ is sufficient regular that satisfies  \eqref{eq:Phi}, we have
\begin{equation}
\label{eq:basic estimate for Q f f dv-1}
\left | \int_{\R^3}  \Psi(\omega, v)  \frac 1  c \partial_{t } \hat{Z}^\alpha   [Q_c(f, f)] (t, x, v)  \dd v \right| +\left| \int_{\R^3} \Psi(\omega, v)  \partial_{x_i}\hat{Z}^\alpha  [Q_c(f, f)]   (t, x , v) \dd v \right|    \lesssim \epsilon_1e^{DM_1} \frac {1 }  {(1+t + |x| )^{5 +}}.
\end{equation}
\end{cor}
\begin{proof}
Recall \eqref{accumulated chain rule for the Z b Q f f}. If $|x|\ge ct$, then we obtain \eqref{eq:basic estimate for Q f f dv-1} from \Cref{upper bound estimate for Q + -} and \eqref{basic estimate for f dv improved version x ge c t2},   If $\av{t-\frac{|x|}c}\le 1$, then we obtain \eqref{eq:basic estimate for Q f f dv-1} from \Cref{upper bound estimate for Q + -}, \eqref{main inequality from 1 t - r c to 1 + t + r} and    \eqref{basic estimate for f dv improved version x le c t1}.\\

For $t-\frac{|x|}c\ge 1$, recall \eqref{basic representation for partial t partial x i}. Similarly as \eqref{eq:integration1}, we have
\begin{equation}
\label{eq:integration2}
\begin{aligned}
&(t-\frac {|x|} {c} )\left[\left | \int_{\R^3}  \Psi(\omega, v)  \frac 1  c \partial_{t } \hat{Z}^\alpha   [Q_c(f, f)] (t, x, v)  \dd v \right|  +  | \left | \int_{\R^3}  \Psi(\omega, v)  \partial_{x_i } \hat{Z}^\alpha   [Q_c(f, f)] (t, x, v)  \dd v \right| \right]\\ \lesssim& \sum_{|\gamma| =1} \left | \int_{\R^3}  \Psi(\omega, v)  Z^\gamma \hat{Z}^\alpha   [Q_c(f, f)] (t, x, v)  \dd v \right|,
\end{aligned}
\end{equation}
and  using that  $Z = \hat{\Omega}_{i} - \frac {v_0}{c} \partial_{v_j}$, if $Z= \Omega_{i}$ and $Z = \hat{\Omega}_{i j}  - v_i\partial_{v_j} + v_j \partial_{v_i}$ if $Z = \Omega_{ij}$
and integration by parts in $v$ for $\partial_{v_i}$, we  obtain \eqref{eq:basic estimate for Q f f dv-1} from \eqref{main inequality from 1 t - r c to 1 + t + r} and  \Cref{CORbasic estimate for Q f f dv}.
\end{proof}

\section{Improvement of the bootstrap assumptions on the electromagnetic field}
\label{sec5}
In this section, we conclude the proof of \Cref{eq:theorem1.1-1} by improving the estimates for the electromagnetic fields $E$ and $B$. This is achieved through a detailed analysis of the Glassey-Strauss decomposition. In the following of the section we will denote $\omega = \frac {y-x} {|y-x|}$.

\subsection{Solutions of the wave equation}
For any smooth function $f(t, x)$, we define that 
\begin{equation}
\label{eq:solwave}
\Phi (f)(t, x) := \frac 1 {4\pi c^2 t^2} \int_{|y-x| = c t} f(0, y)  \dd y +\frac 1 {4 \pi c t}  \int_{|y-x| =c t}     \frac {y-x} {|y-x|}     \cdot \nabla_y f (0, y)     \dd y + \frac 1 {4 \pi c^2 t }  \int_{|y-x| = c t}  \partial_t f (0, y)  \dd y.  
\end{equation}
We can see that  $\Phi(f)$ only depends on the initial data of $f$. 
We first give a result  the wave equation.
%
\begin{lemma}\label{main representation of the modified wave equation}\begin{enumerate}
\item
For any smooth function $f$, the function $\Phi (f)(t, x)$ defined in \eqref{eq:solwave} solves the wave equation 
\[
\partial_t^2 f  - c^2 \Delta_x f = 0.
\]
\item
For any smooth functions $f, g$, the solution to
\[
\partial_t^2 f -c^2\Delta_x  f =  -  \int_{\R^3}  (  \hat{v}_i   \partial_t + c^2 \partial_{x_i}) g  (t, x, v) \dd v\quad f(0, x) = 0,\quad  \partial_{t} f(0, x) =0,\quad \forall x \in \R^3,\quad i=1,2,3
\]
satisfies $f(t, x)  = \Phi(f)(t,x)+ f_{i,1,g} (t, x)  +  f_{i,2,g} (t, x) +f_{i,3,g}(t, x)$, where 
\begin{align*}
&f_{i,1,g} (t, x)  = - \frac 1 {4 \pi   c } \int_{  |y-x | \le c t} \int_{\R^3} \frac { \omega_i   +c^{-1}  \hat{v}_i  }  {1+ c^{-1} \omega \cdot \hat{v} }      ( T_0    g) (t- \frac 1 c | y-x| , y, v )                   \frac 1 {|y-x|} \dd v  \dd y ,\\
&f_{i,2,g} (t, x)  =- \frac 1 {4 \pi  } \int_{  |y-x | \le c t} \int_{\R^3}  \frac {   (\omega_i +c^{-1}   \hat{v}_i)    } {|y-x|^2 (1   + c^{-1}  \omega \cdot \hat{v})^2}   \frac {c^2} {v_0^2 }  g (t-\frac 1 c | y-x| , y, v )            \dd v \dd y, \\
&f_{i,3,g} (t, x) =  -  \frac 1 {4 \pi c t }  \int_{|y-x| =  c t }  \int_{\R^3 } \pare{ \omega_i - \frac {c^{-1} ( \omega_i   + c^{-1}\hat{v}_i )  (\hat{v} \cdot \omega) }  {1+c^{-1} \omega \cdot \hat{v} }  }   g (0, y, v) \dd v \dd y.\\
\end{align*}
Moreover, for the $x$ derivative of $f_{i,2,g}(t, x)$ we have that 
\begin{align*}
\partial_{x_k}  f_{i,2,g} (t, x) = & - \frac 1 {4 \pi  } \int_{  |y-x | \le c t} \int_{\R^3}  \frac {   (\omega_i +c^{-1}   \hat{v}_i) c^{-1}  \omega_k      } {|y-x|^2 (1   + c^{-1}  \omega \cdot \hat{v})^3}  \frac {c^2 }{v_0^2}  ( T_0 g) (t-\frac 1 c | y-x| , y, v )            \dd v \dd y 
\\
& + \frac 1 {4 \pi  } \int_{  |y-x | \le c t} \int_{\R^3} a(\omega, v)   f (t-\frac 1 c | y-x| , y, v )      \frac 1 {|y-x|^3}       \dd v \dd y
\\
&-   \frac 1 { 4 \pi c^2 t^2  }  \int_{|y-x| =  c t }  \int_{\R^3 }  \frac {   (\omega_i +c^{-1}   \hat{v}_i)   \omega_k   } { (1   + c^{-1}  \omega \cdot \hat{v})^3}  \frac {c^2} {v_0^2}      g (0, y, v) \dd v \dd y: =\sum_{j=1}^3  \partial_{x_k} f_{i,2, j,g}  (t, x),
\end{align*}
with 
\[
a(\omega, v) = \frac {    3 (\omega_i   + c^{-1} \hat{v}_i )   (  -c^{-1} \hat{v}_k  ( 1+ c^{-1} \omega \cdot \hat{v} )   +(c^{-2} | \hat{v} |^2 -1 )\omega_{k})  + (1+ c^{-1}\omega \cdot \hat{v} ) ^2 \delta_{ik}         }{   (1+ c^{-1} \omega \cdot \hat{v} )^4   }  \frac{c^2} {v_0^2}.   
\]
\item
Similarly,  the solution to
\[
\partial_t^2 \tilde f -c^2\Delta_x  \tilde f =  \int_{\R^3}  c (\hat{v}_i \partial_{x_j} - \hat{v}_j \partial_{x_i}) g  (t, x, v) dv,\quad  \tilde f (0, x) = 0, \quad \partial_{t} \tilde f (0, x) =0,\quad \forall x \in \R^3,\quad{i=1,2,3}
\]
satisfies $\tilde f  (t, x)  =  \Phi(f)(t,x)+ \tilde f _{i, j, 1, g} (t, x)  +  \tilde f _{i, j, 2,g} (t, x) +\tilde f _{i, j, 3 ,g} (t, x)$, where 
\begin{align*}
&\tilde f _{i, j, 1,g} (t, x)  =  \frac 1 {4 \pi c^2} \int_{  |y-x | \le c t}  \int_{\R^3} \frac {\hat{v}_i \omega_j-   \hat{v}_j \omega_i}  {1 +   c^{-1}  \hat{v} \cdot  \omega}       ( T_0     g) (t- \frac 1 c | y-x| , y, v )                   \frac 1 {|y-x|} \dd v \dd y ,\\
&\tilde f _{i,j,2,g} (t, x)= \frac 1 {4 \pi c } \int_{  |y-x | \le c t} \int_{\R^3}   \frac {  (\hat{v}_i \omega_j  -  \hat{v}_j \omega_i)       }    {   (1+c^{-1} \omega \cdot \hat{v} )^2   }  \frac {c^2} {v_0^2}  g (t- \frac 1 c | y-x| , y, v )                   \frac 1 {|y-x|^2}\dd v  \dd y,\\
&\tilde f _{i, j , 3,g} (t, x) = \frac 1 {4 \pi c^2 t }    \int_{|y-x| = c t}  \int_{\R^3 }  \frac { \hat{v}_i  \omega_j - \hat{v}_j \omega_i   } { (1 + c^{-1}    \hat{v} \cdot  \omega) }   g (0, y, v) \dd v \dd y.\\
\end{align*}
Moreover, for the derivative of $\tilde f _{i, j, 2,g} (t, x)$ we have that 
\begin{align*}
\partial_{x_k}  \tilde f _{i, j, 2,g} (t, x)= &  \frac 1 {4 \pi c  } \int_{  |y-x | \le c t} \int_{\R^3}  \frac {   (\hat{v}_i \omega_j  -  \hat{v}_j \omega_i)    c^{-1}  \omega_k      } {|y-x|^2 (1   + c^{-1}  \omega \cdot \hat{v})^3}  \frac {c^2 }{v_0^2}  ( T_0 g) (t-\frac 1 c | y-x| , y, v )            \dd v \dd y 
\\
& - \frac 1 {4 \pi c } \int_{  |y-x | \le c t} \int_{\R^3} b(\omega, v)   f (t-\frac 1 c | y-x| , y, v )      \frac 1 {|y-x|^3}       \dd v \dd y
\\
&+  \frac 1 { 4 \pi c^3 t^2  }  \int_{|y-x| =  c t }  \int_{\R^3 }  \frac {   (\hat{v}_i \omega_j  -  \hat{v}_j \omega_i)    \omega_k   } { (1   + c^{-1}  \omega \cdot \hat{v})^3}  \frac {c^2} {v_0^2}      g (0, y, v) \dd v \dd y : =\sum_{n=1}^3  \partial_{x_k} \tilde f_{i, j, 2, n,g}(t, x),
\end{align*}
with $b (\omega, v)$ satisfies 
\[
b(\omega, v) = \frac {c^2} {v_0^2}   \frac { (   \delta_{ik}  \hat{v}_j   -\delta_{jk}  \hat{v}_i )   }  {  (1+ c^{-1} \omega \cdot \hat{v} )^2}     
  +  \frac {c^2} {v_0^2}  [  \frac {(\omega_{i}     \hat{v}_j   -\omega_{j} \hat{v}_i   ) [ ( -3   + 3c^{-2}   |\hat{v} |^2 )   \omega_k   -  3 c^{-1}   \hat{v}_k(1+ c^{-1 } \omega \cdot \hat{v} )  ] }  {  (1+ c^{-1} \omega \cdot \hat{v} )^4} .   
\]
\end{enumerate}
 Moreover, for any $v \in \R^3$ we have that $\int_{|\omega| =1} a(\omega, v) \dd \omega = \int_{|\omega| =1} b(\omega, v) \dd \omega =0$.

\end{lemma}

\begin{proof}  (1) is a classical result can be found in any classical textbook. For (2) and (3), the proof for the case $c=1$ can be found in  \cite{glassey1986singularity} Theorem 3, 4, 5, see also \cite{Glassey-1996-SIAM, schaeffer1986classical}. The proof for general $c\ge 1$ is similar (the main difference is to replace $\hat{v}$ in \cite{glassey1986singularity} by $c^{-1} \hat{v}$ in our proof), we omit the proof here. We refer to the second author's homepage for completeness. 
\end{proof}

With this equality, before  we actually give the proof, we first give the derivative formula.
\begin{lemma}
\label{lem:dev} For any smooth functions $f(\omega, v) ,  g(t, x), h(x)$ and $x, y \in \R^3, c \ge 1, t \ge 0$, we have 
\begin{equation}
\label{derivative for y - x = c t}
\partial_{x_i}   \int_{|y-x| = c t}   f(\omega, v) h (y)  \dd y   =   \int_{|y-x| =  c t}  f(\omega, v)   ( \partial_{x_i} h) (y)  \dd y ,
\end{equation}
\begin{equation}\label{derivative for y - x le c t}
\partial_{x_i}   \int_{|y-x| \le c t} \frac 1 {|y-x|^k }  f(\omega, v) g (t-\frac 1 c|y-x|, y)  \dd y=  \int_{|y-x| \le ct} \frac 1 {|y-x|^k }  f(\omega, v)   ( \partial_{x_i} g) (t- \frac 1 c |y-x|, y)  \dd y,
\end{equation}
and \eqref{derivative for y - x le c t} also applies for the integral on $|y-x| =  c t$.
\end{lemma}
\begin{proof}
 Using the change of variable $z=y-x$ we have that 
\[
\int_{|y-x| = c t}   f(\omega, v) h (y)  \dd y = \int_{|z| \le c t}  f(\frac z {|z|}, v) h ( z+x )  \dd z,
\]
and 
\[
\int_{|y-x| \le c t} \frac { f(\omega, v)} {|y-x|^k }  g (t-\frac 1 c |y-x|, y)  dy = \int_{|z| \le c t} \frac {f(\frac z {|z|}, v)} {|z|^k }   g (t-\frac 1 c z, z+x )  \dd z,
\]
which implies that 
\[
\textnormal{l.h.s. of}\,\,\eqref{derivative for y - x = c t}= \int_{|z| = c t}   f(\frac z {|z|}, v) \partial_{x_i}   h (z + x )  \dd z,
=\textnormal{r.h.s. of}\,\,\eqref{derivative for y - x = c t} ,  
\]
and 
\begin{align}\label{derivative for y - x le c t1}
\nonumber
\textnormal{l.h.s. of}\,\,\eqref{derivative for y - x le c t}= &\int_{|z| \le ct} \frac 1 {|z|^k }  f(\frac z {|z|}, v) \partial_{x_i}   g (t-\frac 1 c z, z + x )  \dd z
=\textnormal{r.h.s. of}\,\,\eqref{derivative for y - x le c t},
\end{align}
and \eqref{derivative for y - x le c t} applies similaly for the integral on $|y-x| =  c t$. so the proof is thus finished. 
\end{proof}

\subsection{Estimates of electromagnetic field}
In this subsection we improve Assumptions \ref{basic assumption 1}, \ref{basic assumption 2}. First by Lemma \ref{commutator of Z alpha E B} and Lemma \ref{main representation of the modified wave equation}, for $E, B$ we have the followng lemma.
\begin{lemma}\label{main representation of Z gamma E B}
For any $|\alpha |\le N-1$, for $l=1,2,..., 6 $     we have that 
\begin{equation}
\label{eq:EBEXTRA1}
\pare{Z^\alpha [E, B]}_l  =\sum_{\beta \le \alpha} \sum_{i, j, n=1}^3 C_{\alpha, \beta, 1} f_{i, n,\hat{Z}^\beta f} + C_{\alpha, \beta, 2} \tilde f_{i, j, l, \hat{Z}^\beta f  } + \Phi(\hat{Z}^\alpha[ E,B]_l),
\end{equation}
for some uniformly bounded constants $C_{\alpha, \beta, 1}, C_{\alpha, \beta, 1}$ depends also on $i, j, k, l$. And we have that 
\begin{align}
\label{eq:EBEXTRA2}
\pare{\partial_{x_k} Z^\alpha [E, B]}_l  = & \sum_{\beta \le \alpha} \sum_{i, j=1}^3 \sum_{n=1, 3} C_{\alpha, \beta, 1}  ( f_{i, n , \partial_{x_k}\hat{Z}^\beta f  } )+ C_{\alpha, \beta, 2}  (\tilde f_{i, j, n, \partial_{x_k}\hat{Z}^\beta f } )+\Phi(\partial_{x_k}\hat{Z}^\alpha[ E,B]_l)
\\ \nonumber
&+   \sum_{\beta \le \alpha} \sum_{i, j=1}^3 C_{\alpha, \beta, 1} \partial_{x_k} ( f_{i, 2 , \hat{Z}^\beta f  } )+ C_{\alpha, \beta, 2} \partial_{x_k} (\tilde f_{i, j, 2,\hat{Z}^\beta f } )   .
\end{align}
\end{lemma}
\begin{proof}
From classical theory for the  wave equation, the solution to
\[
\partial_t^2 f -c^2\Delta_x  f =  \sum_{i=1}^N g_i (t, x) ,\quad f(0, x) = h(x),\quad  \partial_{t} f(0, x) =g(x),\quad \forall x \in \R^3,
\]
satisfies $f=  f_0 + \sum_{i=1}^N f_i$, where  $f_0 = \Phi(f)$, and $f_i$ satisfies 
$\partial_t^2 f_i -c^2\Delta_x  f_i =   g_i (t, x), f_i(0, x) =  \partial_{t} f_i(0, x) =0.$
Using \Cref{eq:comm0}, together with \Cref{main representation of the modified wave equation} and \Cref{lem:dev}, we obtain \eqref{eq:EBEXTRA1} and \eqref{eq:EBEXTRA2}.
\end{proof}

Now we are ready to improve Assumption \ref{basic assumption 1}.

\begin{lemma}
\label{lem:improve2}
 Under the assumptions  \ref{basic assumption 1} to  \ref{basic assumption 3} and the initial data assumption \eqref{initial data requirement} and \eqref{extension1}, there exists a universal constant $C_2$, such  that 
\begin{equation}
\label{eq:improve2-1}|Z^\alpha [E, B] |(t, x)\le\frac { C_2(\epsilon_1M_1 e^{D M_1}+M)}  {(1 +   t + |x|  ) (1+   |t - \frac {|x|} c| )    }\quad \forall |\alpha|\le N-1.  \end{equation}
\end{lemma}

\begin{proof} 
First, the term $\left| \frac{1}{1 + c^{-1} v \cdot \omega} \right| \le \frac{1}{1 - \frac{|v|}{v_0}} = \frac{v_0}{v_0 - |v|} \le \frac{2v_0^2}{c^2} \le \langle v \rangle^2$ is uniformly bounded from above. According to Lemma \ref{main representation of Z gamma E B}, it suffices to estimate $f_{i, n , \hat{Z}^\beta f}$, $\tilde{f}_{i, j, n,\hat{Z}^\beta f}$, and $\Phi(\hat{Z}^\alpha [E,B]_l)$ for $|\beta| \le N-1$ and $i, n \in \{1,2,3\}, l=1, 2,..., 6$. Regarding $f_{i,1, \hat{Z}^\beta f}$ and $\tilde{f}_{i, j, 1, \hat{Z}^\beta f}$, recall from \eqref{Representation for T 0 Z alpha f} that $v_0 T_0(\hat{Z}^\alpha f)$ contains terms of the form $\hat{Z}^\beta(v_0 Q_c(f, f) )$ (the Boltzmann collision term) and $Z^{\beta_1}[E, B ]_k$ times  $\partial_v (\hat{Z}^{\beta_2} f)$, where $\beta_1 + \beta_2\le \beta$. By applying integration by parts with respect to $v$ and employing Lemma \ref{lem:inequality for y x c t 1}, Assumption \ref{basic assumption 1}, Lemma \ref{main improved estimate for f dv}, and Corollary \ref{basic estimate for Q f f dv}, we establish the existence of a function $\Psi(\omega, v)$ (maybe different in different lines) satisfying \eqref{eq:Phi} such that:
\begin{equation}
\label{eq:improvedBA11}
\begin{aligned}
&  |f_{i,1, \hat{Z}^\beta f} (t, x)  | +|\tilde f_{i,1, \hat{Z}^\beta f} (t, x) |   
\\
\lesssim  &\frac 1 {  c }  \left|\int_{  |y-x | \le c t} \int_{\R^3}    \Psi (\omega, v)   ( T_0    (\hat{Z}^\beta f) ) (t- \frac 1 c | y-x| , y, v )                   \frac 1 {|y-x|} \dd v  \dd y \right|
\\
\lesssim  & \sum_{|\beta_1| +|\beta_2| \le |\beta|} \left |   \frac 1 { c } \int_{  |y-x | \le c t}   \frac 1 {|y-x|}   |  Z^{\beta_1} [E, B]  |(t- \frac 1 c | y-x| , y )  \left| \int_{\R^3}  \nabla_v \Psi (\omega, v)  \hat{Z}^{\beta_2}   f  (t- \frac 1 c | y-x| , y, v )                   \dd v \right | \dd y  \right | 
\\
&+   \frac 1 {c } \left | \int_{  |y-x | \le  c t} \int_{\R^3}  \langle v \rangle^{30} \left |   \frac {1}{v_0} \hat{Z}^\beta [v_0 Q_c(f,f)  ]  \right | (t- \frac 1 c | y-x| , y, v )                   \frac 1 {|y-x|} \dd v  \dd y  \right|
\\
\lesssim& \epsilon_1M_1 e^{D M_1} \frac 1 {c } \int_{|y-x| \le c t} \frac 1 {(1+  t   -\frac 1 c |y-x| +|y|  )^{3_{+}} (1+  \left| t-\frac 1 c |y-x| -\frac 1 c| y|  \right| ) } \frac 1 {|x-y| } \dd y  
\lesssim \epsilon_1M_1 e^{D M_1}     \frac { \log (3+ | t-\frac {|x|} c | )}   {(1+t+|x|) (1+ | t-\frac {|x|} c |)^{1_{+}} }.
\end{aligned}
\end{equation}
For $f_{i, 2, \hat{Z}^\beta f} $ and $\tilde f_{i, j , 2, \hat{Z}^\beta f} $ term, recall the estimate of $\hat Z^\beta f$. From Assumption \eqref{basic assumption 3}, Lemma \ref{inequality for y x c t 2} and Lemma \ref{main improved estimate for f dv}, we have 
\begin{align}
\nonumber
|f_{i,2, \hat{Z}^\beta f} (t, x)  | +|\tilde f_{i, j , 2, \hat{Z}^\beta f} (t, x) |    \lesssim & \left|  \int_{  |y-x | \le c t} \int_{\R^3} \Psi (\omega, v)    (\hat{Z}^\beta f)  (t-\frac 1 c | y-x| , y, v )         \frac 1 {|y-x|^2  }    \dd v \dd y \right|
\\ \label{eq:improvedBA12}
\lesssim &   \epsilon_1M_1 e^{D M_1}  \int_{|y-x| \le c t} \frac 1 {(1+  t   -\frac 1 c |y-x| +|y|  )^3 } \frac 1 {|x-y|^2 } \dd y  \lesssim  \epsilon_1M_1 e^{D M_1}  \frac 1 {(1+t+|x|) (1+ | t-\frac {|x|} c |) }.
\end{align}
For $f_{i,3, \hat{Z}^\beta f} $, $\tilde f_{i,j,3, \hat{Z}^\beta f} $ and  $\Phi(\hat{Z}^\beta [E,B]_i)$  term,  recall the requirement \eqref{initial data requirement} of $(f_0,E_0, B_0)$. From  \Cref{Basic inequality for x - y = t}, we have 
\begin{equation}
\label{eq:improvedBA13}
\begin{aligned}
&|f_{i,3, \hat{Z}^\beta f} (t, x)  | +|\tilde f_{i, j, 3, \hat{Z}^\beta f} (t, x) |  +| \Phi(\hat{Z}^\beta [E,B]_l) (t, x) |\\
\lesssim & \pare{\sum_{|\gamma| \le N+1 } \Vert \langle x \rangle^{100} \partial_{x}^{\gamma}  [E_0, B_0] (x) \Vert_{L^\infty_x} +  \sum_{|\alpha| +|\beta| \le N+1 }  \Vert \langle x \rangle^{100} \langle v \rangle^{100}  \partial_{x}^{\alpha}  \partial_{v}^{\beta} f_0(x, v) \Vert_{L^\infty_{x, v}}  } 
\times \max \left\{ \frac 1 {c^2t^2  }, \frac 1 {ct}\right \}\times\int_{|x-y|  = c t} (1+|y|)^{-10} \dd S_y 
\\
\lesssim &( \epsilon_1M_1 e^{D M_1}+M) \max \left\{ \frac 1 {c^2t^2  }, \frac 1 {ct}\right \}  \min \{ct, c^2t^2 \}  (1+ ct+|x|)^{-1} (1+|c t -|x||)^{-10} \\
=&( \epsilon_1M_1 e^{D M_1} +M) (1+ ct+|x|)^{-1} (1+|c t -|x||)^{-10},
\end{aligned}
\end{equation}
and \Cref{eq:improve2-1} is thus proved by combining \eqref{eq:improvedBA11} to \eqref{eq:improvedBA13} together.
\end{proof}

We then improve Assumption  \ref{basic assumption 2}.

\begin{lemma}
\label{lem:improve3}
Let $\epsilon_0\lesssim \epsilon^2$, where $\epsilon_0$ is given in \eqref{initial data requirement}. Under the assumptions  \ref{basic assumption 1} to  \ref{basic assumption 3}, for any $|\alpha |\le N-1$, there exists a universal constant $C_3$, such  that 
\begin{equation}
\label{eq:improvedBA20}
|\frac 1 c \partial_t  Z^\alpha [E, B]  |(t, x)  + |\nabla_{x}   Z^\gamma [E, B] |(t, x) \le \frac {C_3( \epsilon_1M_1^2 e^{D M_1} +M)\log(3+ |t - \frac {|x|} c|  )} {(1+t+|x|   ) (1+ |t - \frac {|x|} c| )^2    },\quad \forall |\alpha| \le N-1.
\end{equation}
\end{lemma}

\begin{proof} We first prove \eqref{eq:improvedBA20} for $\nabla_{x}   Z^\gamma [E, B]  $ term. From Lemma \ref{main representation of Z gamma E B} we only need to estimate 
\[
 f_{i, n,   \partial_{x_k}\hat{Z}^\beta f  }, \quad\tilde f_{i, j, n, \partial_{x_k} \hat{Z}^\beta f  },    \quad \partial_{x_k} f_{i, 2,   \hat{Z}^\beta f  }, \quad\partial_{x_k} \tilde f_{i,j, 2, \hat{Z}^\beta f  },  \quad   (\Phi(\partial_{x_k} \hat Z^\beta [E,B] ))_l, \quad|\beta| \le N-1, \quad i, j=1,2,3, \quad n=1, 3.
\]
For $ f_{i, 1,   \partial_{x_k}\hat{Z}^\beta f  }$  and $\tilde f_{i, j, 1, \partial_{x_k} \hat{Z}^\beta f  }$, we follow the analysis similar as  \eqref{eq:improvedBA11}. Using  \eqref{Representation for T 0 Z alpha f}, together with integration by parts in $v$ and then taking $\partial_{x_k}$ derivative. Then since $|\beta| \le N-1$, using Assumption \ref{basic assumption 2},   Lemma \ref{better estimate for partial x i f dv} and Corollary \ref{basic estimate for Q f f dv}, and Lemma \ref{lem:inequality for y x c t 1} , there exists $\Psi(\omega, v)$ (maybe different in different lines) that satisfies \eqref{eq:Phi}, such that
\begin{equation}
\label{eq:improvedBA21}
\begin{aligned}
&  |f_{i, 1,   \partial_{x_k}\hat{Z}^\beta f  }   (t, x)  | +|\tilde f_{i, j, 1, \partial_{x_k} \hat{Z}^\beta f   } (t, x) |   
\\
\lesssim  &\frac 1 {  c }  \left|\int_{  |y-x | \le c t} \int_{\R^3}    \Psi (\omega, v)   (  \partial_{x_k} T_0    (\hat{Z}^\beta f) ) (t- \frac 1 c | y-x| , y, v )                   \frac 1 {|y-x|} \dd v  \dd y \right|
\\
\lesssim  & \sum_{|\beta_1| +|\beta_2| \le |\beta|} \left |   \frac 1 { c } \int_{  |y-x | \le c t}   \frac 1 {|y-x|}   |  \partial_{x_k} Z^{\beta_1} [E, B]  |(t- \frac 1 c | y-x| , y )  \left| \int_{\R^3} \nabla_v  \Psi (\omega, v)  \hat{Z}^{\beta_2}   f  (t- \frac 1 c | y-x| , y, v )                   \dd v \right | \dd y  \right | 
\\
&+ \sum_{|\beta_1| +|\beta_2| \le |\beta|} \left |   \frac 1 { c } \int_{  |y-x | \le c t}   \frac 1 {|y-x|}   |  Z^{\beta_1} [E, B]  |(t- \frac 1 c | y-x| , y )  \left| \int_{\R^3}  \Psi (\omega, v)  \partial_{x_k} ( \hat{Z}^{\beta_2} f)  (t- \frac 1 c | y-x| , y, v )                   \dd v \right | \dd y  \right | 
\\
&+   \frac 1 {c } \left | \int_{  |y-x | \le  c t} \int_{\R^3} \nabla_v  \Psi (\omega, v)   \frac {1}{v_0} \partial_{x_k} \hat{Z}^\beta [v_0 Q_c(f,f)  ]   (t- \frac 1 c | y-x| , y, v )                   \frac 1 {|y-x|} \dd v  \dd y  \right|
\\
\lesssim&\epsilon_1M_1 e^{D M_1}  \frac 1 {c } \int_{|y-x| \le c t} \frac 1 {(1+  t   -\frac 1 c |y-x| +|y|  )^{4} (1+  \left| t-\frac 1 c |y-x| -\frac 1 c| y|  \right| ) } \frac 1 {|x-y| } \dd y  
\lesssim \epsilon_1M_1 e^{D M_1}   \frac { \log (3+ | t-\frac {|x|} c | )}   {(1+t+|x|) (1+ | t-\frac {|x|} c |)^{2} }.
\end{aligned}
\end{equation}
	For the term $\partial_{x_k}f_{i,3}$ and $\partial_{x_k} \tilde f_{i,3}$, using \eqref{derivative for y - x = c t}, similarly as \eqref{eq:improvedBA13},  we have that 
\begin{equation}
\label{eq:improvedBA22}
\begin{aligned}
| f_{i, 3, \partial_{x_k} \hat{Z}^\beta f   }(t, x)  | +|\tilde f_{i, j, 3, \partial_{x_k} \hat{Z}^\beta f }  (t, x) |  
\lesssim \epsilon_1M_1 e^{D M_1}  (1+ ct+|x|)^{-1} (1+|c t -|x||)^{-10}.
\end{aligned}
\end{equation}
Finally we come to the term $\partial_{x_k} f_{i,2, \hat{Z}^\beta f}$ , $\partial_{x_k} \tilde f_{i, j, 2, \hat{Z}^\beta f} $ and and $\Phi(\partial_{x_k}\hat{Z}^\alpha[ E,B]_l) $. Recall Lemma \ref{main representation of the modified wave equation}, they are decomposed into 3 parts. For the last part $\partial_{x_k} f_{i,2, 3, \hat{Z}^\beta f} $ and $\partial_{x_k} \tilde f_{i, j, 2, 3, \hat{Z}^\beta f} $, the proof is the same as before, we have that 
\begin{equation}
\label{eq:improvedBA23}
\begin{aligned}
|\partial_{x_k} f_{i,2,3, \hat{Z}^\beta f} (t, x)  | +|\partial_{x_k} \tilde f_{i, j,2,3, \hat{Z}^\beta f} (t, x) |  + |\Phi(\partial_{x_k}\hat{Z}^\alpha[ E,B]_l)| \lesssim &(\epsilon_1M_1 e^{D M_1}+M)   (1+ ct+|x|)^{-1} (1+|c t -|x||)^{-10}.
\end{aligned}
\end{equation}
For the term $\partial_{x_k} f_{i,2, 1}$ and $\partial_{x_k} \tilde f_{i,2, 1}$, similar as \eqref{eq:improvedBA11}, using Lemma \ref{inequality for y x c t 2}, , Assumption \ref{basic assumption 1}  Lemma \ref{main improved estimate for f dv} and Corollary \ref{basic estimate for Q f f dv} we have that 
\begin{equation}
\label{eq:improvedBA24}
\begin{aligned}
&  |\partial_{x_k} f_{i, 2, 1, \hat{Z}^\beta f} (t, x)  | +|\partial_{x_k} \tilde f_{i, j, 2, 1, \hat{Z}^\beta f} (t, x) |
\\
\lesssim& \left|\int_{  |y-x | \le c t} \int_{\R^3}    \Psi (\omega, v)   ( T_0    (\hat{Z}^\beta f) ) (t- \frac 1 c | y-x| , y, v )                   \frac 1 {|y-x|^2 } \dd v  \dd y \right|
\\
\lesssim  & \sum_{|\beta_1| +|\beta_2| \le |\beta|} \left |   \int_{  |y-x | \le c t}   \frac 1 {|y-x|^2}   |  Z^{\beta_1} [E, B]  |(t- \frac 1 c | y-x| , y )  \left| \int_{\R^3}  \Psi (\omega, v)  \hat{Z}^{\beta_2}   f  (t- \frac 1 c | y-x| , y, v )                  \dd v \right | \dd y  \right | 
\\
&+    \left | \int_{  |y-x | \le  c t} \int_{\R^3}  \langle v \rangle^{30} \left |   \frac {1}{v_0} \hat{Z}^\beta [v_0 Q_c(f,f)  ]  \right | (t- \frac 1 c | y-x| , y, v )                   \frac 1 {|y-x|^2 } \dd v  \dd y  \right|
\\
\lesssim& \epsilon_1M_1^2 e^{D M_1} \int_{|y-x| \le c t} \frac 1 {(1+  t   -\frac 1 c |y-x| +|y|  )^4  } \frac 1 {|x-y|^2 } \dd y  
\lesssim \epsilon_1M_1^2 e^{D M_1}   \frac { 1 }   {(1+t+|x|) (1+ | t-\frac {|x|} c |)^2 }.
\end{aligned}
\end{equation}
For the term $\partial_{x_k} f_{i,2, 2, \hat{Z}^\beta f}$, we have  
\begin{align*}
\partial_{x_k} f_{i,2, 2,\hat{Z}^\beta f} (t, x)  =  &\frac 1 {4 \pi  } \int_{  |y-x | \le c t} \int_{\R^3} a(\omega, v)   (\hat{Z}^\beta f )(t-\frac 1 c | y-x| , y, v )      \frac 1 {|y-x|^3}       \dd v \dd y
\\
= &\frac 1 {4 \pi  }  \int_{  |y-x | \le 1} + \int_{  1\le |y-x | \le c t}  \int_{\R^3}  a(\omega, v)  ( \hat{Z}^\beta f )(t-\frac 1 c | y-x| , y, v )      \frac 1 {|y-x|^3}       \dd v \dd y  : = T_1+T_2.
\end{align*}
For the $T_2$ term, using Lemma \Cref{main improved estimate for f dv} and Lemma \ref{inequality for y x c t 3} we have that
\begin{equation}
\label{eq:improvedBA25}
|T_2| \lesssim  \epsilon_1M_1 e^{D M_1} \int_{1\le  |y-x| \le c t} \frac 1 {(1+  t   -\frac 1 c |y-x| +|y|  )^3 } \frac 1 {|x-y|^3 } \dd y  \lesssim \epsilon_1M_1 e^{D M_1} \frac {\log(3+t +|x| ) } {(1+t+|x|)^2  (1+ | t-\frac {|x|} c |) } .
\end{equation}
For the $T_1$ term, recall $\int_{|\omega| =1} a(\omega, v) \dd \omega =0$, we have
\[
\int_{|y-x| \le 1} \int_{\R^3}  a(\omega, v)    \frac 1 {|y-x|^3}      (\hat{Z}^\beta f) (t-\frac 1 c  |y-x| , x, v)  \dd v  \dd y  = 0.
\]
Thus 
\begin{align*}
T_1 = &\int_{|y-x| \le 1} \int_{\R^3}a(\omega, v)     \frac 1 {|y-x|^3}   [ (\hat{Z}^\beta f) (t-\frac 1 c  |y-x| , y, v)   - ( \hat{Z}^\beta f) (t-\frac 1 c  |y-x| , x, v) ]  \dd v  \dd y 
\\
= &\int_0^1 \int_{|y-x| \le 1} \int_{\R^3}  a(\omega, v)      \frac 1 {|y-x|^3}   [ (y-x) \cdot (\nabla_x \hat{Z}^\beta  f)  (t-\frac 1 c  |y-x| , y +s (x-y)  , v)  ]  \dd s \dd v  \dd y  .
\end{align*}
Since $|y-x| \le 1$ we have that $1+t  \lesssim 1+ t-\frac 1 c|y-x|   $, and for all $t' \in [0, 1]$, we have that  $1 + |y+t'(x-y) |\ge 1+ \min \{|y|, |x| \} \ge |x| $, which together with Lemma \ref{better estimate for partial x i f dv} implies that 
\begin{equation}
\label{eq:improvedBA26}
\begin{aligned}
T_1 \lesssim 
&\int_0^1 \int_{|y-x| \le 1} \frac 1 {|y-x|^2} \left | \int_{\R^3}   \Psi(\omega, v)    \nabla_x \hat{Z}^\beta f (t-\frac 1 c  |y-x| , y +s (x-y)  , v) \dd v \right |   \dd y  \dd s   
\\
\lesssim&\epsilon_1M_1 e^{D M_1} \int_0^1 \int_{|y-x| \le 1}     \frac 1 {|y-x|^2} \frac 1 {(1+ t-\frac 1 c  |y-x|  +|y+s(x-y) |)^3 }  \dd y \dd s   
\\
\lesssim&  \epsilon_1M_1 e^{D M_1}\frac {1} {(1+ t+|x|)^3 } \int_{|y-x| \le 1}     \frac 1 {|y-x|^2 }    \dd y   \lesssim \epsilon_1M_1 e^{D M_1} \frac 1 {(1+ t+|x|)^3 } ,
\end{aligned}
\end{equation}
so we have finish the proofs by gathering \eqref{eq:improvedBA21} to \eqref{eq:improvedBA26}, as the $\partial_{x_k} \tilde f_{i,j, 2,2, \hat{Z}^\beta f} $ term can be proved similarly. The proof for $\frac 1 c\partial_t$ follows from the fact that 
\[\partial_t E = c\nabla \times B -J,  \quad \partial_t B = -c \nabla \times E,\] together with Lemma \ref{main improved estimate for f dv} and Lemma \ref{basic commutator formula}. 
\end{proof}
Now we are ready to prove \Cref{thm:global}.
\begin{proof}[Proof of \Cref{thm:global}]
For any given $M>0$, let $D, C_1, C_2, C_3, C_4$ be the universal constants in \eqref{extension1}, \eqref{eq:improve2-1}, \eqref{eq:improvedBA20} and \eqref{improvement assumption 3} denote $C_5  = \max \{C_1, C_2, C_3, C_4\}$.  Now we choose  
\[
M_1 = 4C_5M,\quad \epsilon_0 = \min \left  \{  \frac {1} {16C_5^2e^{2DM_1}} , \frac {1} { 16C_5M_1^2 e^{DM_1}  } \right\},
\]
then we have for any $\epsilon_1\in (0,\epsilon_0]$,
\[
C_1\sqrt{\epsilon_1} e^{DM_1} \le \frac 1 2 ,\quad C_2(\epsilon_1 M_1 e^{DM_1}+M ) + C_3(\epsilon_1 M_1^2 e^{DM_1}+M ) + C_4 \epsilon_1 M_1 e^{DM_1} \le \frac 1 2 M_1,
\]
and we conclude that the bootstrap assumptions \eqref{basic assumption 1}--\eqref{basic assumption 3} are strictly improved. Through a standard continuity argument, this allows for the further extension of the solution’s lifespan, thereby implying that $T_{\textnormal{max}} = +\infty$.\\Furthermore, we recall that the terms $f_i$ and $\tilde{f}_i$ in the representation formula \eqref{main representation of the modified wave equation} correspond to the electric field $E(t,x)$ and the magnetic field $B(t,x)$, respectively. Based on the structural properties of $f_i$ and $\tilde{f}_i$, the estimates for $B(t,x)$ are found to be precisely $1/c$ times those obtained for $E(t,x)$. Consequently, the global bounds \eqref{eq:theorem1.1-1} and \eqref{eq:theorem1.1-2} follow directly from the bootstrap estimates \eqref{basic assumption 1} and \eqref{basic assumption 3}. This completes the proof of \Cref{thm:global}.
\end{proof}

\section{The relativistic Boltzmann operator}

\label{sec3}
This section focuses on the relativistic Boltzmann operator. We provide its core estimates and establish a rigorous proof of the associated chain rule. We recall that the relativistic Boltzmann collision operator $Q_c$ writes
\begin{equation}
\label{Bol1}
Q_c(h, f)  = \int_{\R^3} \int_{\mathbb{S}^2}   B_c(v, u) \left(   h(u') f(v' ) - h(u) f(v) \right)\dd\omega \dd u ,
\end{equation}
where 
\[
g = \sqrt{2 (v_0u_0 - v \cdot w - c^2)},\quad s =2 (v_0u_0 - v \cdot w + c^2), \quad s =g^2+4c^2,
\]
and
\[
B_c(v, w) = v_\phi \sigma(g,\vartheta),\quad v_\phi= \frac {c g \sqrt{s} }  {4 v_0u_0},\quad \sigma(g,\vartheta)\defeq |g|^{\gamma-1}\sigma_0(\vartheta),\quad 0\le \sigma_0(\vartheta)\lesssim 1,
\]
and the scattering angle $\vartheta$ is defined as
\begin{equation}
\label{eq:vartheta}
\cos\vartheta=\frac{-(v_0-u_0)(v_0'-u_0')+(v-u)(v'-u')}{g^2},
\end{equation}
 $v'$ and $u'$
 stand for the postcollisional momenta of particles which had the momenta $v$ and $u$ before the collision, respectively. In the center-of-momentum frame, under the scattering angle $\omega$, we have
 \begin{align*}
\nonumber
 v' = \frac {v+u} 2 +\frac g 2 \left(\omega + (\zeta-1) (v+u ) \frac {(v+u) \cdot \omega} {|v+u|^2} \right),\quad
u' = \frac {v+u} 2 -\frac g 2 \left(\omega + (\zeta-1) (v+u ) \frac {(v+u) \cdot \omega} {|v+u|^2} \right),
\end{align*}
where $\zeta = \frac {u_0+v_0} {\sqrt{s} }$. We compute that 
\[
\zeta -1  = \frac {v_0+u_0-\sqrt{s} }{\sqrt{s}}  = \frac {(v_0+u_0)^2 -s} {\sqrt{s}  (v_0+u_0 + \sqrt{s} )}  = \frac {v_0^2+u_0^2+2v_0u_0 -s} {\sqrt{s}  (v_0+u_0 + \sqrt{s} )}   = \frac {|v+u|^2 } {\sqrt{s}  (v_0+u_0 + \sqrt{s} )} ,
\]
thus we also have that 
\begin{equation}
\label{eq:p'q'}
\begin{aligned}
v' = \frac {v+u} 2 +\frac g 2 \left(\omega +   (v+u) \frac {(v+u) \cdot \omega} {\sqrt{s}  (v_0+u_0 + \sqrt{s} )}  \right),\quad
u' =\frac {v+u} 2 -\frac g 2 \left(\omega +   (v+u) \frac {(v+u) \cdot \omega} {\sqrt{s}  (v_0+u_0 + \sqrt{s} )}  \right),
\end{aligned}
\end{equation}
and we can also compute that 
\begin{equation}
\label{eq:p0'q0'}
v_0' = \frac {v_0+u_0} {2} + \frac {g} {2 \sqrt{s}} \omega \cdot (v+u), \quad u_0' = \frac {v_0+u_0} {2} - \frac {g} {2 \sqrt{s}} \omega \cdot (v+u).
\end{equation}

\subsection{Basic estimates}
Let us start by providing some basic estimates in the relativistic regime. 
\begin{lemma}[ \cite{cao2025classical}, Lemma 2.3]We have
\begin{equation}
\label{upper and lower bound g}
c\frac {|v-u|}  {\sqrt{v_0 u_0}} \le g \le |v-u|,  \quad   g \le 5 \min \left \{ \frac {\langle v \rangle \langle u \rangle \sqrt{u_0}} {\sqrt{v_0}} ,  \frac {\langle v \rangle \langle u \rangle \sqrt{v_0}} {\sqrt{u_0}} \right \},\quad 
 4c^2 \le s \le 4v_0u_0.
\end{equation}
\end{lemma}
\begin{lemma}
\label{upper and lower bound g}
For any $\gamma \in (-2, 0]$, we have that 
\begin{equation}
\label{upper bound for v phi}
|v_\phi \sigma(g,\vartheta)  |  \lesssim 1+ \frac 1 {|v-u|^{ - \gamma} }  .
\end{equation}
\end{lemma}
\begin{proof}
 Using that $\sqrt{s} \lesssim g+c, g \le \sqrt{v_0u_0} $ and \eqref{upper and lower bound g}, for $\gamma \in [-1, 0]$, we have that 
\[
v_\phi |\sigma(g,\vartheta)| \lesssim  v_\phi |g|^{\gamma-1}  = \frac {cg \sqrt{s}} {4v_0u_0}  |g|^{\gamma-1}       \lesssim   \frac {c^2  }     {4v_0u_0}  |g|^{\gamma}    +   \frac {c }     {4v_0u_0}  |g|^{\gamma+ 1}   \lesssim       \frac {c^2 } {v_0u_0} \left( \frac {   v_0u_0   }    {c^2} \right)^{\frac {-\gamma }  {2} } \frac 1 {|v-u|^{ - \gamma} } +     \frac {c } {\sqrt{v_0u_0} }  \lesssim 1+ \frac 1 {|v-u|^{ - \gamma} }   ,
\]
and if $\gamma +1 \in (-1, 0]$, using that $a^{\gamma+1 } \lesssim a^{\gamma} +1$ we have that 
\[
v_\phi |\sigma(g,\vartheta)| \lesssim v_\phi |g|^{\gamma-1}  = \frac {cg \sqrt{s}} {4v_0u_0}  |g|^{\gamma-1}  =   \frac {c   \sqrt{s}} {4v_0u_0}  |g|^{\gamma}      \lesssim   \frac {c^2  }     {4v_0u_0}  |g|^{\gamma}    +   \frac {c }     {4v_0u_0}    \lesssim       \frac {c^2 } {v_0u_0} \left( \frac {   v_0u_0    }    {c^2} \right)^{\frac {-\gamma }  {2} } \frac 1 {|v-u|^{ - \gamma} } + 1  \lesssim 1+ \frac 1 {|v-u|^{ - \gamma} } ,  
\]
so the proof is thus finished. 
\end{proof}

\begin{lemma}
Suppose that $v', u'$ are defined before, then for any $c\ge 1$ we have
\begin{equation}
\label{p 2 smaller than p' 2 q' 2}
1+|v'|^2 \le 1+3(|v|^2+|u|^2), \quad 1+|u'|^2 \le 1+3(|v|^2+|u|^2).
\end{equation}

\end{lemma}
\begin{proof}
Remind \eqref{eq:p'q'}. Together with $u_0+v_0\ge |v+u|$, \eqref{upper and lower bound g} and Cauchy-Schwartz inequality, we have
\[
|v'|^2\le 3\pare{ \frac {|v+u|^2}4+\frac {g^2}4+\frac{g^2|v+u|^2}{4s}}\le 3\pare{ \frac {|v+u|^2}4+ \frac {|v-u|^2}4+ \frac {|v+u|^2}4}\le 3(|v|^2+|u|^2),
\]
the proof of $u'$ is similar.
\end{proof}

\begin{lemma}
\label{lem:estimateweight}
For the $v, u, v'$ defined before, we have that
\begin{equation}
\label{estimate x t hat p p' q'0}
\min \left \{         \left |\frac {v'} {v_0'}  - \frac {v} {v_0}  \right  | , \left |\frac {v'} {v_0'}  - \frac {u} {u_0}  \right  | \right  \}\le C\min \left\{\frac {v_0^2}{c^2} , \frac {u_0^2 } {c^2} \right\}    \left |\frac {v} {v_0}  - \frac {u} {u_0}  \right  |    \le C\min \{\langle v \rangle^2 , \langle u \rangle^2   \}    \left |\frac {v}{ v_0}  - \frac {u} {u_0}  \right  | ,
\end{equation}
for some uniform constant $C>0$ independent of  $v, u, \omega, c$. By reverse the order of $v, u$ and $v', u'$ we have that 
\[
\min \left \{  \left | \frac {v} {v_0}  -  \frac {u'} {u_0'}   \right  |   ,  \left | \frac {v} {v_0}  -  \frac {v'} {v_0'}   \right  | \right  \}\le C\min \{\langle v' \rangle^2 , \langle u' \rangle^2   \}   \left |\frac {v'} {v_0'}  - \frac {u'} {u_0'}  \right  | .
\]
Moreover, there exists a constant $C>0$  independent of  $v, u, \omega, c, x, t$, such that for any $x \in \R^3, t  \ge0$,
\begin{equation}
\label{estimate x t hat p p' q'}
\langle x-t \hat{v} \rangle  \le C    \min\{   \langle v' \rangle^2    ,  \langle u' \rangle^2 \}   (\langle x-t\hat{v'}\rangle    +  \langle  x-t\hat{u'} \rangle ) .
\end{equation}
 \end{lemma}

\begin{proof}
We only prove the case $|v|\ge |u|$, i.e., $v_0\ge u_0$. The case $|v| \le |u|$ can be proved by switching $u$ and $v$. We have 
\begin{align*}
v_0u_0 g^2 - |v u_0 - u   v_0  |^2=& v_0u_0 (2v_0u_0-2v\cdot u -2c^2) - |v|^2u_0^2 +2 v \cdot u v_0u_0 -|u|^2 v_0^2
\\
= & 2v_0^2u_0^2 -  |v|^2u_0^2     -  |u|^2 v_0^2  - 2c^2v_0u_0  
= c^2u_0^2+c^2v_0^2 - 2c^2v_0u_0 = c^2(v_0-u_0)^2.
\end{align*}
Recall that $g \ge\frac {c|v-u|} {\sqrt{v_0u_0}}$, we have that $|v u_0 - u v_0  |^2  =   v_0u_0 g^2  - c^2(v_0-u_0)^2 \ge c^2 | v-u|^2 - c^2(v_0-u_0)^2 =c^2 g^2$. Thus
\begin{equation}
\label{eq:pqomegarevised1}
g\lesssim\frac{\av{uv_0-vu_0}}c\lesssim \frac{v_0u_0}{c}\av{\frac{v}{v_0}-\frac{u}{u_0}},\quad g\lesssim\frac{\av{uv_0-vu_0}+ c|u_0-v_0|}{\sqrt{u_0v_0}}=\sqrt{u_0v_0}\av{\frac{v}{v_0}-\frac{u}{u_0}}+\frac{ c|u_0-v_0|}{\sqrt{u_0v_0}}.
\end{equation}
Recall \eqref{eq:p'q'} and \eqref{eq:p0'q0'}, we compute
\begin{align*}
v'v_0-vv_0'  = &   \frac {(v+u)v_0} 2 +\frac {v_0g} 2 \left(\omega +   (v+u ) \frac {(v+u) \cdot \omega} {\sqrt{s}  (v_0+u_0 + \sqrt{s} )}  \right) -\frac {v(v_0+u_0)} {2} -  \frac {vg} {2 \sqrt{s}} \omega \cdot (v+u)
\\
= &\frac {uv_0 - vu_0} {2 } +\frac {v_0 g} {2} \omega + \frac {g (v+u)\cdot \omega } {2\sqrt{s} }   \left( \frac {(v+u) v_0  } {v_0 +u_0+\sqrt{s} }  - v \right)
\\
= &\frac {uv_0 - vu_0} {2 } +\frac {v_0 g} {2} \omega + \frac {g (v+u)\cdot \omega } {2\sqrt{s} }   \left( \frac {uv_0 - vu_0} {v_0 +u_0+\sqrt{s} }  -  \frac {v\sqrt{s}   } {v_0 +u_0+\sqrt{s} } \right)\defeq T_1+T_2+T_3+T_4.
\end{align*}
From $|T_1|\lesssim |uv_0 - vu_0|$ and $ |T_3| \lesssim  |uv_0 - vu_0|  \frac {(|v|+|u|)g} {\sqrt{s} (v_0 +u_0+\sqrt{s}) } \lesssim  |uv_0 - vu_0| $,
we obtain that
\begin{equation}
\label{eq:pqomegarevised2}
\frac{1}{v_0'v_0}\pare{\av{T_1}+\av{T_3}}\lesssim \frac 1{cv_0}|uv_0-vu_0|=\frac {u_0}c\av{\frac{v}{v_0}-\frac{u}{u_0}} \lesssim \min\left\{\frac{u^2_0}{c^2}, \frac{v^2_0}{c^2}\right\}\av{\frac{v}{v_0}-\frac{u}{u_0}}.
\end{equation}
Next, to estimate $T_2$ and $T_4$, we divide into two parts.
\begin{enumerate}
\item If $v_0\le 100 u_0$. Then $\av{T_2}\lesssim v_0g$ and $\av{T_4}\lesssim g|v|\frac{|u|+|v|}{u_0+v_0}\lesssim v_0g$.  Together with $v_0\sim u_0$, we have
\begin{equation}
\label{eq:pqomegarevised3}
\frac{1}{v_0'v_0}\pare{\av{T_2}+\av{T_4}}\lesssim\frac{v_0g}{cv_0}\lesssim \frac{u_0v_0}{c^2}\av{\frac{v}{v_0}-\frac{u}{u_0}} \lesssim \min\left\{\frac{u^2_0}{c^2}, \frac{v^2_0}{c^2}\right\}\av{\frac{v}{v_0}-\frac{u}{u_0}}.
\end{equation}
\item If $v_0\ge 100 u_0$. If $g\le \frac{\sqrt s}2$, we have
\[
v_0'=\frac{u_0+v_0}2+\frac {\sqrt{g } } {2\sqrt s}\omega\cdot(u+v)\ge \frac{u_0+v_0}2-\frac g{2\sqrt s}\av{u+v}\ge\frac {u_0+v_0}{4},
\]
and
\begin{equation}
\label{eq:pqomegarevised3+}
\frac{1}{v_0'v_0}\pare{\av{T_2}+\av{T_4}}\lesssim\frac{v_0g}{v^2_0}\lesssim \frac{u_0}{c}\frac{\av{uv_0-vu_0}}{u_0v_0} \lesssim \min\left\{\frac{u^2_0}{c^2}, \frac{v^2_0}{c^2}\right\}\av{\frac{v}{v_0}-\frac{u}{u_0}}.
\end{equation}
 Next, we compute
\begin{align*}
T_2+T_4&=\frac g2\pare{v_0\omega-\frac{(u+v)\cdot\omega}{v_0+u_0+\sqrt s}v}=\frac g2\pare{v_0\omega-\frac{(u+v)\cdot\omega}{v_0+u_0}v+\frac{(u+v) \cdot \omega\sqrt s}{(u_0+v_0+\sqrt s)(u_0+v_0)}v}
\\
&=\frac g2\pare{\underbrace{(u_0+v_0)\omega-\frac{(u+v)\cdot\omega}{u_0+v_0}(u+v)}_{T_5}\underbrace{-v_0\omega+\frac{(u+v)\cdot\omega}{u_0+v_0}u+\frac{(u+v)\cdot \omega\sqrt s}{(u_0+v_0+\sqrt s)(u_0+v_0)}v}_{T_6}}.
\end{align*}
Remind $v_0\ge 100 u_0$, we have $|v|\ge \max\{c, 100|u|\}$. Thus 
$(vu_0-u v_0) \cdot (uv_0+vu_0)=v^2u_0^2-u^2v_0^2=c^2(|v|^2-|u|^2)\ge  {c^2{|v|^2}/2},$ which implies that
\[
\av{uv_0-vu_0}\ge \frac {c^2|v|^2}{2\av{uv_0+vu_0}}\ge \frac {c^2|v|^2}{4u_0v_0}\ge \frac {c^2v_0}{16u_0}\quad\Longrightarrow \quad \frac {c^2}{u_0^2}\lesssim \frac {|uv_0-vu_0|}{u_0v_0}=\av{\frac{v}{v_0}-\frac{u}{u_0}}.
\]
Together with \eqref{eq:pqomegarevised1}, we have
\[
g\lesssim\sqrt{u_0v_0}\av{\frac{v}{v_0}-\frac{u}{u_0}}+\frac{ c\sqrt {v_0}}{\sqrt{u_0}}\lesssim \frac{\sqrt{u_0v_0}u_0}c\av{\frac{v}{v_0}-\frac{u}{u_0}}\quad\Longrightarrow \quad \frac g2\av{T_6}\lesssim gu_0+g\sqrt s\lesssim g\sqrt{u_0v_0}\lesssim \frac{v_0u_0^2}{c}\av{\frac{v}{v_0}-\frac{u}{u_0}},
\]
which means that
\begin{equation}
\label{eq:pqomegarevised4}
\frac 1{v_0'v_0}\frac g2\av{T_6}\lesssim  \frac{u_0^2}{c^2}\av{\frac{v}{v_0}-\frac{u}{u_0}}.
\end{equation}
Next, we compute
\begin{align*}
\left |T_5 \right|^2 =& \frac 1 {(v_0+u_0)^2} [ (v_0+u_0)^4 - 2 (v_0+u_0)^2 |\omega \cdot (v+u)|^2  + |\omega \cdot (v+u)|^2 |v+u|^2] 
\\
\lesssim & \frac 1 {(v_0+u_0)^2} [ (v_0+u_0)^4 -  (v_0+u_0)^2 |\omega \cdot (v+u)|^2 ] \lesssim [(v_0+u_0)^2  -|\omega \cdot (v+u)|^2  ].
\end{align*}
We also have that 
\[
v_0' u_0' =   \frac {(v_0+u_0)^2 } {4} + \frac {g^2 } {4 s} |\omega \cdot (v+u)|^2 \ge \frac {g^2 } {4s} [(v_0+u_0)^2  -|\omega \cdot (v+u)|^2  ] \ge \frac 1 {16 } [(v_0+u_0)^2  -|\omega \cdot (v+u)|^2  ] ,
\]
and we have that $u_0' \le  v_0+u_0 \le 2v_0$, thus we have that 
\begin{equation}
\label{eq:pqomegarevised5}
\left |T_5  \right|^2  \lesssim v_0' u_0'\lesssim  v_0 v_0'\quad\Longrightarrow \quad \frac 1 {v_0' v_0} g\left |T_5  \right| \lesssim g \frac 1 {\sqrt{v_0' v_0} } \lesssim   \frac {\sqrt{ u_0}u_0} {c^{\frac 3 2 } }\left |\frac {v} {v_0}  - \frac {u} {u_0}  \right  |,
\end{equation}
\end{enumerate}
and \eqref{estimate x t hat p p' q'0} is thus proved by combining \eqref{eq:pqomegarevised2}, \eqref{eq:pqomegarevised3}, \eqref{eq:pqomegarevised4}, \eqref{eq:pqomegarevised5} together.
 \eqref{estimate x t hat p p' q'} is direct from the estimates
\begin{equation}
\label{estimate x t hat p p' q'-1}
\min\{\langle x-t \hat{v} \rangle - \langle x-t \hat{v' } \rangle , \langle x-t \hat{v} \rangle - \langle x-t \hat{u' } \rangle \} \le\min\{t |\hat{v} -\hat{v'}|,  t |\hat{v} -\hat{u'}|\} \lesssim \min\{   \langle v' \rangle^{2}  , \langle u' \rangle^{2} \}  t   |\hat{v'} -\hat{u'}| ,
\end{equation}
and
\begin{equation}
\label{estimate x t hat p p' q'-1extra}
  t |\hat{v'} -\hat{u'} |  =  |(x-t\hat{v '}) -  (x-t\hat{u'}) | \le  \langle x-t\hat{v'}\rangle    +  \langle  x-t\hat{u'} \rangle .
\end{equation}
So the proof is thus finished. 
\end{proof}

\subsection{Estimates for the relativistic Boltzmann operator}
In this subsection, we provide the estimates of the relativistic Boltzmann operator. We start by listing some propositions.

\begin{lemma}[Pre-post collisional change of variable for relativistic Boltzmann operator, \cite{strain2011coordinates}, Corollary 5 and (23)] \label{pre-post change of variable}
For any smooth function $F$ we have that
\[
\int_{\R^3} \int_{\R^3} \int_{\mathbb{S}^2}   v_\phi \sigma(g,\vartheta) F(v, w, v', u')\dd\sigma \dd v \dd u = \int_{\R^3} \int_{\R^3}  v_\phi \sigma(g,\vartheta) F(v', u', v, w)\dd\sigma \dd v \dd u.
\]

\end{lemma}

\begin{lemma}[Relativistic Carleman Representation, \cite{jang2022asymptotic}, Lemma 2.12]
\label{lem:Carleman}For any smooth functions $f, h $, denote that 
\[
 \bar{g} := g (v, v')   = \sqrt{2 (v_0v_0' - v \cdot v'- c^2)},\quad \tilde{g} := g (u, v')   = \sqrt{2 (u_0v_0' - u \cdot v'- c^2)},
\]
then we have 
\begin{align*}
Q^+_c(h, f)  =& \int_{\R^3} \int_{\mathbb S^2}    v_\phi \sigma(g,\vartheta)  h (u') f(v' )   \dd\omega \dd u  
\\
= &\frac c {v_0} \int_{\R^3 }  \frac 1 {8\bar{g} } \frac 1 {v_0'} f(v') \dd v'   \int_{\R^3}   \frac {s} {u_0} \sigma(g,\vartheta)  \mathpzc{u}(v_0 +u_0-v_0')  h (v+u-v')    \delta \left (\frac {\bar{g} } 2   + \frac { - u_0 (v_0- v_0' )  + u \cdot (v-v')  } {\bar{g}} \right)  \dd u
\\
= &\frac c {v_0} \int_{\R^3 }  \frac 1 {8\bar{g} } \frac 1 {v_0'} f(v') \dd v'  \int_{\R^3}   \frac {s} {u_0}  \sigma(g,\vartheta)  \mathpzc u(v_0 +u_0-v_0')  h (v+u-v')      \delta    \left(  \frac {\bar{g}^2 + \tilde{g}^2  -g^2  } {2 \bar{g}} \right)     \dd u ,
\end{align*}
where

\[
\mathpzc{u}(x) = 1 ,\quad \hbox{if} \quad x \ge c, \quad \mathpzc{u}(x) = 0 ,\quad \hbox{if} \quad x <c. 
\]
\end{lemma}
\begin{proof} The first equality just follows from \cite{jang2022asymptotic}, Lemma 2.11, Lemma 2.12, see also \cite{jang2021propagation} Lemma 2.4 and (5.1) in \cite{jang2021propagation2}. We also note that 
\[
 2u_0 (v_0- v_0' )   -  2u \cdot (v-v')  = 2(v_0u_0 - v\cdot u- c^2) -   2(u_0v_0' - u \cdot v'- c^2) = g^2 - \tilde{g}^2 ,
\]
which equals to $\bar{g}^2 =  g^2 - \tilde{g}^2 $ (we also refer to \cite{jang2021propagation}, Section 2). Thus we have
\[
\delta \left (\frac {\bar{g} } 2   + \frac { - u_0 (v_0- v_0' )  + u \cdot (v-v')  } {\bar{g}} \right)  = \delta  \left (  \frac {\bar{g}^2 + \tilde{g}^2  -g^2  } {2 \bar{g}}   \right),
\]
so the proof is thus finished.
\end{proof}
\begin{lemma}[ Proposition 2.7 of \cite{jang2021propagation}]
\label{lem:vartheta}
Remind $\vartheta$ defined in \eqref{eq:vartheta}, we have $\sin\pare{\frac\vartheta 2}=\frac{\bar g}g.$
\end{lemma}

Before giving the estimates for the $Q^+_c, Q^-_c $ term, we first give an estimate for the intermediate term.

\begin{lemma}\label{spherical coordinates reduction for the relativistic Boltzmann operator}
For any $\beta \in \R$, we have that 
\begin{equation}
\label{eq:intermediate}
\begin{aligned}
&\int_{\R^3}   u_0^\beta     \mathpzc u(v_0 +u_0-v_0')  \langle  v+u-v' \rangle^{-k}       \delta \left (\frac {\bar{g} } 2   + \frac { - u_0 (v_0- v_0' )  + u \cdot (v-v')  } {\bar{g}} \right)  \dd u
\\ 
\lesssim &     \frac {\bar{g}} { |v'-v| }  \int_{0}^\infty  |u|   u_0^\beta         \mathpzc u(v_0 +u_0-v_0')  (1+   |v_0+u_0 - v_0'|^2 -c^2  )^{-k/2}   \dd |u| .
\end{aligned}
\end{equation}

\end{lemma}

\begin{proof}
We first take the solar coordinates  on $u$ into angular coordinates and choose the z-axis
parallel to $v-v'$, such that the angle between $u$ and $v-v'$ is equal to $\theta$, i.e. $v-v' = (0, 0, |v-v'|)$. Then we have 
$(v-v') \cdot u = |u|  |v-v' | \cos\theta$. Remind $g^2 = \bar{g}^2 + \tilde{g}^2$, where $\bar g, \tilde g$ defined in \Cref{lem:Carleman},  we have that 
\begin{align*}
|v_0+u_0 - v_0'|^2 - |v+u-v'|^2 -c^2 
=& v_0^2 +u_0^2 +v_0'^2 +2v_0u_0- 2u_0v_0' -2v_0v_0'    -v^2 -w^2 -u'^2 -2v \cdot w +2 w \cdot v' + 2 v \cdot u'  -c^2
\\
= &2v_0u_0- 2u_0v_0' -2v_0v_0'     -2v \cdot w +2 w \cdot v' + 2 v \cdot u'  +2c^2   = g^2 - \bar{g}^2 - \tilde{g}^2 = 0,
\end{align*}
or we can write that 
\begin{equation}
\label{g g g implies equality}
u_0+v_0-v_0' \ge c,\quad g^2 - \bar{g}^2 - \tilde{g}^2 = 0 \quad\Longrightarrow\quad  v_0+u_0 - v_0' = (v+u-v')_0 = \sqrt{c^2 + |v+u-v'|^2}. 
\end{equation}
Thus taking the polar coordinates of $u$,
using that
\begin{align*}
  \delta \left (\frac {\bar{g} } 2   + \frac { - u_0 (v_0- v_0' )  + u \cdot (v-v')  } {\bar{g}} \right) =& \delta(\frac 1 {2 \bar{g} }   ( \bar{g}^2 -   2 u_0 (v_0  -v_0')  + 2|u| |v'-v|\cos\theta   ) ) 
 =   \frac {\bar{g}} {|u| |v'-v| }  \delta(\frac  {\bar{g}^2 -   2 u_0 (v_0-v_0') }  {2 |u| |v'-v|}   + \cos\theta    ) ,
\end{align*}
under these spherical coordinates,  we have 
\begin{align*}
&\int_{\R^3}   u_0^\beta     \mathpzc u(v_0 +u_0-v_0')  \langle  v+u-v' \rangle^{-k}       \delta \left (\frac {\bar{g} } 2   + \frac { - u_0 (v_0- v_0' )  + u\cdot (v-v')  } {\bar{g}} \right)  \dd u
\\
=  & \int_{\R^3}   u_0^\beta        \mathpzc u(v_0 +u_0-v_0') (1+   |v_0+u_0 - v_0'|^2 -c^2  )^{-k/2}  \delta \left (\frac {\bar{g} } 2   + \frac { - u_0 (v_0- v_0' )  + u \cdot (v-v')  } {\bar{g}} \right)  \dd u
\\
\lesssim &  \int_{0}^\infty \int_{0}^{\pi}   \int_{0}^{2\pi} u_0^\beta            \mathpzc u(v_0 +u_0-v_0')  (1+   |v_0+u_0 - v_0'|^2 -c^2  )^{-k/2}   
\frac {\bar{g}} {|u| |v'-v| }  \delta\pare{\frac  {\bar{g}^2 -   2 u_0 (v_0-v_0') } {2 |u| |v'-v|}   + \cos\theta   }   |u|^2 \sin\theta    \dd|u| \dd \theta \dd \phi  
\\
\lesssim &   \frac {\bar{g}} { |v'-v| }  \int_{0}^\infty \int_{-1}^{1}   \int_{0}^{2\pi}   u_0^\beta        \mathpzc u(v_0 +u_0-v_0')  (1+   |v_0+u_0 - v_0'|^2 -c^2  )^{-k/2}   
   \delta\pare{\frac  {\bar{g}^2 -   2 u_0 (v_0-v_0') }  {2 |u| |v'-v|}   + \cos\theta    }     |u|    \dd|u| \dd (-\cos \theta)  \dd \phi  
\\
\lesssim &     \frac {\bar{g}} { |v'-v| }  \int_{0}^\infty  |u|   u_0^\beta         \mathpzc u(v_0 +u_0-v_0')  (1+   |v_0+u_0 - v_0'|^2 -c^2  )^{-k/2}   \dd |u|,
\end{align*}
the proof is thus finished. 
\end{proof}

\begin{lemma}\label{estimates for mid term of Carleman q}
Define
\begin{equation}
\label{eq:Cvbetak}
\mathpzc C(v,v',\beta,k,c)\defeq \int_{0}^\infty   |u|  |u_0|^\beta   \mathpzc u(v_0 +u_0-v_0')  (1+   |v_0+u_0 - v_0'|^2 -c^2  )^{-k/2}  \dd|u| ,
\end{equation}
then for any $v , v' \in \R^3$ satisfies $|v| \ge |v'|$, and for any $\beta \ge -1, k \ge\max\{9, \beta+10\}, c \ge 1$,  we have
\begin{equation}
\label{eq:Carlemanq1}
\begin{aligned}
\mathpzc C(v,v',\beta,k,c)
\lesssim c^{\beta}  (1+  |v_0-v_0'|^2   +2c(v_0-v_0'))  )^{-(k-1)/2} + (1+  |v_0-v_0'|^2  +2c(v_0-v_0')  )  +c^2  )^{-(k-\beta-2)/2} \lesssim c^\beta.
\end{aligned}
\end{equation}
Moreover if we assume further that $|v| \ge 2|v'| $, for any $\beta \ge -1, k \ge 8$ we have that 
\begin{equation}
\label{estimate for q q 0 u -k p > 2p'}
\mathpzc C(v,v',\beta,k,c) \lesssim c^{\beta}  \langle v \rangle^{ -(k-1) }+ \langle v \rangle^ {-(k-\beta-2) },
\end{equation}
for any $v , v' \in \R^3$ satisfies $|v| \le |v'|$, for any $ k \ge 8$ we have that
\begin{equation}
\label{eq:Carlemanq2}
\mathpzc C(v,v',\beta,k,c)   \lesssim \frac {(v_0')^{\beta+1}}  {c }  .
\end{equation}
Gathering \eqref{eq:Carlemanq1} and \eqref{eq:Carlemanq2}, for $|v| \le 2 |v'| $, for any $ k \ge 8$ we have that
\begin{equation}
\label{estimate for q q 0 u -k p < 2p'}
\mathpzc C(v,v',\beta,k,c)  \lesssim \frac {(v_0')^{\beta+1}}  {c }   +  c^\beta , 
\end{equation}
where all the constants are independent of $v, v', c, f$.
\end{lemma}
\begin{proof}
\begin{enumerate}
\item If  $|v| \ge |v'|$ then we have  $v_0 \ge v_0'$, let  $\mathpzc v=u_0, \mathpzc w=\mathpzc v-c$,  we have
\begin{align*}
\mathpzc C(v,v',\beta,k,c)
\lesssim &\int_{0}^\infty    |u|   |u_0|^\beta   (1+   |v_0+u_0 - v_0'|^2 -c^2  )^{-k/2}  \dd|u| 
\lesssim \int_{c}^\infty \mathpzc v^{\beta+1}   (1+   |v_0+\mathpzc v  - v_0'|^2 -c^2  )^{-k/2}  \dd \mathpzc v
\\
\lesssim &\int_{0}^\infty (\mathpzc w+c )^{\beta+1}   (1+   |v_0+\mathpzc w+c - v_0'|^2 -c^2  )^{-k/2}  \dd \mathpzc w\\
\lesssim &\int_{0}^\infty  (\mathpzc w+c)^{\beta+1} (1+  |v_0+\mathpzc w-v_0'|^2+  2 \mathpzc w c   +2c(v_0-v_0')) )^{-k/2}  \dd \mathpzc w \defeq
\mathpzc C_1(v,v',\beta,k,c).
\end{align*}
We estimate two cases $\mathpzc w \ge c$ and $\mathpzc w \le c$ seperately. For the case $\mathpzc w \le c$,  taking $y= \mathpzc w c$, we have that 
\begin{equation}
\label{eq:Carlemanq3}
\begin{aligned}
\mathpzc C _1(v,v',\beta,k,c)
\lesssim &(1+  |v_0-v_0'|^2 +  +2c(v_0-v_0')   )    )^{-(k-1)/2} \int_{0}^c  c^{\beta+1}  (1+   2 \mathpzc w  c    )^{-1/2}  \dd \mathpzc w
\\
\lesssim &c^{\beta} (1+  |v_0-v_0'|^2  +2c(v_0-v_0'))  )^{-(k-1)/2} \int_{0}^1    (1+   2 y    )^{-1/2}  \dd y\\
 \lesssim & c^{\beta} (1+  |v_0-v_0'|^2  +2c(v_0-v_0'))  )^{-(k-1)/2}.
\end{aligned}
\end{equation}
For the case $\mathpzc w\ge c$, using the change of variable $\mathpzc w=\sqrt{1+|v_0-v'_0|^2 +2c(v_0-v_0') +c^2    }y$, we have
\begin{equation}
\label{eq:Carlemanq4}
\begin{aligned}
\mathpzc C_1(v,v',\beta,k,c)
\lesssim &\int_{0}^\infty  \mathpzc w^{\beta+1}        (1+  |v_0+\mathpzc w-v_0'|^2  +2c(v_0-v_0')) +c^2     )^{-k/2}  \dd \mathpzc w
\\
\le &\int_{0}^\infty     \mathpzc w^{\beta+1}     (\mathpzc w^2 + 1+  |v_0-v_0'|^2   +2c(v_0-v_0')  +c^2     )^{-k/2}  \dd \mathpzc w
\\
\lesssim & (1+  |v_0-v_0'|^2  +2c(v_0-v_0')  +c^2     )^{-(k-\beta-2)/2}  \int_{0}^\infty   y^{\beta+1}   (1+y^2    )^{-k/2}  \dd y 
\\
\lesssim&  (1+  |v_0-v_0'|^2  +2c(v_0-v_0')  +c^2     )^{-(k-\beta-2)/2},
\end{aligned}
\end{equation}
and \eqref{eq:Carlemanq1} is thus proved after \eqref{eq:Carlemanq3} and \eqref{eq:Carlemanq4}.
\item
For $2|v'| \le |v| $ we have $v_0- v_0' = \frac {|v|^2-|v'|^2} {v_0'+v_0} \ge \frac {|v|^2-|v'|^2} {2v_0} \ge \frac {|v|^2} {8v_0}$.
Moreover, we have $2c (v_0-v_0') \ge \frac {c|v|^2} {4v_0} \ge \frac {|v|^2} {8} $ for $|v| \le c$ and $|v_0-v_0'|^2 \ge \frac {|v|^4} {64v_0^2 } \ge \frac {|v|^2} {128 }$ for $|v| \ge c$. Gathering the cases above, we have
\[
(1+  |v_0-v_0'|^2  +2c(v_0-v_0')) )^{-1} \lesssim \langle v \rangle^{-2},\quad \forall |v| \ge 2|v'|, \quad v, v' \in \R^3,
\]
and \eqref{estimate for q q 0 u -k p > 2p'} is thus proved. \item If  $|v| \le |v'|$, then we have  $v_0 \le v_0'$, let  $\mathpzc v=u_0, \mathpzc w=\mathpzc v-( c+ v_0'- v_0)$, we have
\begin{align*}
\mathpzc C(v,v',\beta,k,c)\lesssim &\int_{v_0+u_0 -v_0' \ge c} |u| |u_0|^\beta   (1+   |v_0+u_0 - v_0'|^2 -c^2  )^{-k/2}  \dd |u| 
\\
\lesssim &\int_{v_0'+c-v_0 }^\infty \mathpzc v^{\beta+1}    (1+   |v_0+\mathpzc v - v_0'|^2 -c^2  )^{-k/2}  \dd \mathpzc v
\\
\lesssim &\int_{0}^\infty ( \mathpzc w+c+v_0'-v_0)^{\beta+1}    (1+   \mathpzc w^2+  2 \mathpzc w c    )^{-k/2}  \dd \mathpzc w\defeq \mathpzc C_2(v,v',\beta,k,c).
\end{align*}
If  $\mathpzc w \le c + v_0-v_0' $, we have
\begin{equation}
\label{eq:Carlemanq5}
\begin{aligned}
\mathpzc C_2(v,v',\beta,k,c) \lesssim &  \int_{0}^\infty ( c+v_0'-v_0)^{\beta+1}     (1+   2 \mathpzc w c    )^{-k/2}  \dd \mathpzc w  
\lesssim   \frac{(v_0')^{\beta+1}}c \int_{0}^\infty   (1+   2 y    )^{-k/2}  \dd y \lesssim \frac {(v_0')^{\beta+1}}  {c } . 
\end{aligned}
\end{equation}
iIf $\mathpzc w \ge  c + v_0-v_0' $,  we have
\begin{equation}
\label{eq:Carlemanq6}
\begin{aligned}
\mathpzc C_2(v,v',\beta,k,c) 
&\lesssim \int_{0}^\infty      \mathpzc w^{\beta+1}  (1+ \mathpzc w^2    )^{-k/2}  \dd \mathpzc w\lesssim 1,
\end{aligned}
\end{equation}
and  \eqref{eq:Carlemanq2} is thus proved after combining \eqref{eq:Carlemanq5} and \eqref{eq:Carlemanq6}. 
\end{enumerate}
\end{proof}

\begin{lemma}\label{upper bound estimate for Q + -}
Suppose $\gamma \in (-2, 0]$ , then for any $p >\frac {3} {3+\gamma}, k\ge 10$ and any smooth functions $h, f$ we have that 
\begin{align*}
\int_{\R^3} \langle v \rangle^k |Q_c(h, f)  | \dd v   \lesssim   (\Vert    h   \Vert_{L^1_v}   +\Vert    h   \Vert_{L^p_v} )    \Vert \langle v \rangle^{k}  f \Vert_{L^1_v}    + (\Vert   f  \Vert_{L^1_v}    +   \Vert   f  \Vert_{L^p_v} ) \Vert \langle v \rangle^{k} h \Vert_{L^1_v}  .
\end{align*}
\end{lemma}

\begin{proof}
For $Q^-_c $ , we use \eqref{upper bound for v phi} and Lemma \ref{basic lemma for L p estimate Hardy} to obtain that 
\begin{align*}
|  \int_{\R^3} \langle v \rangle^k |Q^-_c(h, f)  | \dd v     | \lesssim & \int_{\R^3} \int_{\R^3 }  v_\phi\sigma(g,\vartheta)   |h(u)  |  \langle v \rangle^k  |f(v)  |  \dd u    \dd v 
\\
\lesssim  &    \int_{\R^3 }  \int_{\R^3 }   \pare{  1+ \frac 1 {|v-u|^{ - \gamma} } }       |h(u)  |   |f(v)  |  \langle v \rangle^k \dd u \dd v  
\lesssim (\Vert     h \Vert_{L^1_v}    + \Vert h  \Vert_{L^p_v} ) \Vert    \langle v \rangle^k f \Vert_{L^1_v} . 
\end{align*}
For $Q^+_c$ term, using the pre-post collisional change of variable Lemma \ref{pre-post change of variable} and Lemma \ref{basic lemma for L p estimate Hardy} we have that 
\begin{align*}
\int_{\R^3 }  \langle v \rangle^k |Q^+_c(h, f)  | \dd v  \lesssim  &  \int_{\R^3} \int_{\R^3} \int_{\mathbb{S}^2}  v_\phi\sigma(g,\vartheta)  \langle v \rangle^k |f(v')| |h(u')|  \dd\sigma \dd v \dd u 
\lesssim   \int_{\R^3} \int_{\R^3} \int_{\mathbb{S}^2}   v_\phi\sigma(g,\vartheta)  \langle v' \rangle^k |f(v)| |h(u)|  \dd\sigma \dd v \dd u 
\\
\lesssim  & \int_{\R^3} \int_{\R^3} \int_{\mathbb{S}^2}  \pare{  1+ \frac 1 {|v-u|^{ - \gamma} } }( \langle v \rangle^k + \langle w  \rangle^k)|f(v)| |h(u)|  \dd\sigma \dd v \dd u 
\\
  \lesssim &   (\Vert    h   \Vert_{L^1_v}   +\Vert    h   \Vert_{L^p_v} )    \Vert \langle v \rangle^{k}  f \Vert_{L^1_v}    + (\Vert   f  \Vert_{L^1_v}    +   \Vert   f  \Vert_{L^p_v} ) \Vert \langle v \rangle^{k} h \Vert_{L^1_v},
\end{align*}
so the proof is thus finished. 
\end{proof}
As a direct result, we obtain that
\begin{cor}
\label{cor:upper bound estimate for Q + -}
Suppose $\gamma \in (-2, 0]$ , then for any $p >\frac {3} {3+\gamma}, k\ge 10$ and any smooth function $ f$, we have that 
\begin{align*}
\int_{\R^3} \langle v \rangle^k |Q_c(f, f)  | \dd v   \lesssim  (1+t)^{-3-\frac 3p}\Vert \bkk v^ {k+5}\bkk{x-t\hat v}^4f\Vert_{L^\infty_v}\Vert \bkk v^ {k}\bkk{x-t\hat v}^4f\Vert_{L^\infty_v}.
\end{align*}
\end{cor}
\begin{proof}
From \Cref{L 1 L infty estimate on x t v}, we have
\begin{align*}
\int_{\R^3} \langle v \rangle^k |Q_c(f, f)  | \dd v  & \lesssim   (\Vert    f  \Vert_{L^1_v}   +\Vert    f  \Vert_{L^p_v} )    \Vert \langle v \rangle^{k}  f \Vert_{L^1_v} \\& \lesssim (1+t)^{-3}\Vert \bkk v^ {k+5}\bkk{x-t\hat v}^4f\Vert_{L^\infty_v}\cdot (1+t)^{-\frac 3p}\Vert \bkk v^ {k}\bkk{x-t\hat v}^4f\Vert_{L^\infty_v}\\
&= (1+t)^{-3-\frac 3p}\Vert \bkk v^ {k+5}\bkk{x-t\hat v}^4f\Vert_{L^\infty_v}\Vert \bkk v^ {k}\bkk{x-t\hat v}^4f\Vert_{L^\infty_v},
\end{align*}
and the proof of the corollary is thus finished.
\end{proof}
\begin{lemma}\label{upper bound estimate for Q + -}
Suppose $\gamma \in (-2, 0]$, then for any smooth function $h, f$ and $ p > \frac {9} {5+\gamma} $, for any $k \ge 20$, we have that 
\begin{align*}
|\langle v \rangle^k Q^+_c(h, f)  | \lesssim& (\Vert \langle v \rangle^k f   \Vert_{L^1_v}    + \Vert \langle v \rangle^k f   \Vert_{L^{p}_v}  )   \Vert \langle v \rangle^{k} h\Vert_{L^\infty}   ,
\end{align*}
also by reverse $v'$, $u' $ (by the change of variable  $\omega =-\omega$) we have that
\begin{align*}
|\langle v \rangle^k Q^+_c(h, f)  | \lesssim& (\Vert \langle v \rangle^k h   \Vert_{L^1_v}    + \Vert \langle v \rangle^k h  \Vert_{L^{ p}_v}  )   \Vert   \langle v \rangle^{k}    f   \Vert_{L^\infty}   .
\end{align*}
For the $Q^-_c$ term we have that
\begin{align*}
|\langle v \rangle^k Q^-_c(h, f)  |  \lesssim (\Vert \langle v \rangle   h   \Vert_{L^1_v}   + \Vert \langle v \rangle   h   \Vert_{L^{ p}_v}  ) \Vert \langle v \rangle^{k} f \Vert_{L^\infty_v}  .
\end{align*}

\end{lemma}

\begin{proof}
The case $|v| \le 1$ is obvious, thus we can suppose $|v| \ge 1$. 
Using \eqref{upper bound for v phi} and Lemma \ref{basic lemma for L p estimate Hardy}, for any $p > \frac {3} {3+\gamma}$, we have
\begin{align*}
|\langle v \rangle^k Q^-_c(h, f)  | \lesssim& \langle v \rangle^k    \int_{\R^3 }  v_\phi \sigma(g,\vartheta)|h(u)  |   |f(v)  |  \dd u   \lesssim \int_{\R^3} \left(1+ \frac 1 {|p-q|^{ - \gamma} } \right)      |h(u)  |   \dd u \Vert \langle v \rangle^{k}\ f \Vert_{L^\infty}    \lesssim        ( \Vert    h   \Vert_{L^1_v}  +    \Vert    h   \Vert_{L^{p}_v} ) \Vert \langle v \rangle^{k}\ f \Vert_{L^\infty}   .
\end{align*}
Next we come to compute $Q^+_c$. For simplicity, define
\[
\mathpzc G_k\p{v,u,v',c}\defeq   \mathpzc u(v_0 +u_0-v_0') \langle v+u-v' \rangle^{-k}     \delta    \left(  \frac {\bar{g}^2 + \tilde{g}^2  -g^2  } {2 \bar{g}} \right). 
\]
Using that $s =g^2+4c^2$, taking the $L^\infty$ norm of $h$ and using Lemma \ref{lem:Carleman} we have that 
\begin{small}
\begin{align*}
|\langle v \rangle^k Q^+_c(h, f) |  \lesssim & \int_{\R^3 } \frac c {v_0} \langle v \rangle^{k} \frac 1 {\bar{g} } \frac 1 {v_0'} |f(v') |\dd v'   \int_{\R^3}   \frac {s} {u_0} |g|^{\gamma-1}  \mathpzc u(v_0 +u_0-v_0') | h (v+u-v')  |    \delta    \left(  \frac {\bar{g}^2 + \tilde{g}^2  -g^2  } {2 \bar{g}} \right)     \dd q 
\\
\lesssim 
&  \int_{\R^3 } \frac c {v_0} \langle v \rangle^{k} \frac 1 {\bar{g} } \frac 1 {v_0'} |f(v') |\dd v' \int_{\R^3}   \frac {1} {u_0}  |g|^{\gamma+1}  \mathpzc G_k\p{v,u,v',c}\dd u \cdot  \Vert \langle v \rangle^{k}     h \Vert_{L^\infty}\\
+ & \int_{\R^3 } \frac c {v_0} \langle v \rangle^{k} \frac 1 {\bar{g} } \frac 1 {v_0'} |f(v') |\dd v'  \int_{\R^3}   \frac {1} {u_0}   c^2 |g|^{\gamma-1 }    \mathpzc G_k\p{v,u,v',c}    \dd u \cdot \Vert \langle v \rangle^{k}     h \Vert_{L^\infty}
:=  T_1+T_2.
\end{align*}
\end{small}
We will split it into two cases $g \ge \frac 1 {100}$ and $g \le \frac 1 {100}$.
\begin{enumerate}
\item For the case $g \ge \frac 1 {100}$, we first compute $T_2$,  by Lemma \ref{spherical coordinates reduction for the relativistic Boltzmann operator}  and using \eqref{estimate for q q 0 u -k p > 2p'},\eqref{estimate for q q 0 u -k p < 2p'} for $|v| \ge 2|v'|$ and $|v |\le 2|v'|$ respectively, and using \eqref{upper and lower bound g} to $\bar{g}$ (by changing $u$ to $v'$),   for any $p >\frac 3 2$ we have that 
\begin{align*}
T_2 \lesssim  &  \int_{\R^3 } \frac c {v_0} \langle v \rangle^{k} \frac 1 {\bar{g} } \frac 1 {v_0'} |f(v') |\dd v'  \int_{\R^3}   \frac {c^2 } {u_0}  |g|^{-1}   \mathpzc G_k\p{v,u,v',c}     \dd u  \Vert \langle v \rangle^{k}     h \Vert_{L^\infty}
\\
\lesssim & \int_{\R^3 }  \frac c {v_0}  \langle v \rangle^{k} \frac 1 {|v'-v|  } \frac c {v_0'} f(v') 1_{|v| \le 2|v ' | } \dd v'   
 \int_{0}^\infty  |u|   \frac c {u_0}      \mathpzc u(v_0 +u_0-v_0')  (1+   |v_0+u_0 - v_0'|^2 -c^2  )^{-k/2}   \dd|u|    \cdot \Vert \langle v \rangle^{k}     h \Vert_{L^\infty}  
\\
+ &\int_{\R^3 }  \frac c {v_0}  \langle v \rangle^{k} \frac 1 {|v'-v|  } \frac c {v_0'} f(v')  1_{|v| \ge 2|v'| } \frac {1} {\bar{g} } \dd v'   
\int_{0}^\infty  |u|        \mathpzc u(v_0 +u_0-v_0')  (1+   |v_0+u_0 - v_0'|^2 -c^2  )^{-k/2}   \dd|u|    \cdot \Vert \langle v \rangle^{k}     h \Vert_{L^\infty}  
\\
\lesssim & \int_{\R^3 }  \frac c {v_0}  \langle v \rangle^{k} \frac 1 {|v'-v|  } \frac c {v_0'} |f(v') | (1_{|v| \le 2|v ' | } + 1_{|v| \ge 2|v'| }  \langle v \rangle^{-k+2 }  \bar{g}^{-1})  \dd v'  \cdot \Vert \langle v \rangle^{k} h\Vert_{L^\infty}  
\\
\lesssim & \int_{\R^3 }  \frac c {v_0}  \langle p \rangle^{k} \frac 1 {|v'-v|  } \frac c {v_0'} |f(v') | \pare{1_{|v| \le 2|v' | } + 1_{|v| \ge 2|v'| }  \langle v \rangle^{-k+2 }  \frac {\sqrt{v_0v_0'} } {c |v-v'| }}  \dd v'  \cdot \Vert \langle v \rangle^{k} h\Vert_{L^\infty}  
\\
\lesssim & \int_{\R^3 } \left(1+   \frac 1 {|v'-v|  }\right) \langle v' \rangle^{k} |f(v') | \dd v' \cdot \Vert \langle v \rangle^{k} h\Vert_{L^\infty}    \lesssim    (\Vert \langle v \rangle^k f   \Vert_{L^1_v}    + \Vert \langle  v \rangle^k f   \Vert_{L^{ p}_v}  )   \Vert \langle v \rangle^{k} h\Vert_{L^\infty}  ,
\end{align*}
where we note that $\langle v \rangle \lesssim\langle v' \rangle$ if $|v| \le 2 |v'|$ and $\langle v \rangle \lesssim |v-v'| $ if $|v| \ge 2|v'|$.

We next estimate $T_1$.  If  $\gamma +1 \ge 0$, then using $|g|^{\gamma +1} \lesssim |g| \lesssim v_0+u_0 $, for any $p >\frac 3 2 $ we have
\begin{align*}
 T_1  \lesssim &\left( \int_{\R^3 } \langle v \rangle^{k} \frac 1 {\bar{g} } \frac c {v_0'}  |f(v')  | \dd v'   \int_{\R^3}   \frac {1} {u_0}\mathpzc G_k\p{v, u, v', c}     \dd u  
+  \int_{\R^3 }\frac c {v_0} \langle u \rangle^{k} \frac 1 {\bar{g} } \frac 1 {v_0'}   |f(v')  |   \dd v'   \int_{\R^3}  \mathpzc G_k\p{v,u,v',c}    \dd u\right) \Vert \langle v \rangle^{k}     h \Vert_{L^\infty}
\\
\lesssim & \int_{\R^3 }  \langle v \rangle^{k} \frac 1 {{|v'-v|} } \frac c {v_0'} |f(v') |\dd v'     \int_{0}^\infty  |u|   \frac 1 {u_0}       \mathpzc u  (v_0 +u_0-u_0')  (1+   |v_0+u_0 - v_0'|^2 -c^2  )^{-k/2}   \dd|u|     \Vert \langle v \rangle^{k}     h \Vert_{L^\infty}  
\\
  +&  \int_{\R^3 } \frac c {v_0} \langle v \rangle^{k} \frac 1 {|v'-v|  } \frac 1 {v_0} f(v') \dd v'    \int_{0}^\infty  |u|        \mathpzc u  (v_0 +u_0-u_0')  (1+   |v_0+u_0 - v_0'|^2 -c^2  )^{-k/2}   \dd|u|     \Vert \langle v \rangle^{k}     h \Vert_{L^\infty}  \\
\lesssim & \int_{\R^3 } \langle v \rangle^{k} \frac 1 {|v'-v|  } \frac c {v_0'} f(v')  (1_{|v| \le 2|v ' | } + 1_{|v| \ge 2|v'| }  \langle v \rangle^{-k+1 } )   \dd v'  \Vert \langle v \rangle^{k}     h \Vert_{L^\infty}  \\
+ & \int_{\R^3 } \frac c {v_0} \langle v \rangle^{k} \frac 1 {|v'-v|  } \frac 1 {v_0'} \pare{\frac {v_0'} {c} 1_{|v| \le 2|v ' | } + 1_{|v| \ge 2|v'| }  \langle v \rangle^{-k+2 } }   f(v') \dd v'    \Vert \langle v \rangle^{k}     h \Vert_{L^\infty}  
 \\
\lesssim & \int_{\R^3 }  \frac 1 {|v'-v|  }  \langle  v' \rangle^{k}   f(v')  1_{|v| \le 2|v ' | }  \dd v'  \Vert \langle v \rangle^{k}     h \Vert_{L^\infty}  +  \int_{\R^3 }      f(v')    1_{|v| \ge 2|v'| }   \dd v'      \Vert \langle v \rangle^{k}    h \Vert_{L^\infty}  
\\
\lesssim & (\Vert \langle v \rangle^k f   \Vert_{L^1_v}    + \Vert \langle v \rangle^k f   \Vert_{L^{ p}_v}  )   \Vert \langle v \rangle^{k} h\Vert_{L^\infty}.  
\end{align*}
If $\gamma+1\le 0$, then we have
\begin{align*}
 T_1  \lesssim & \int_{\R^3 } \langle v \rangle^{k}\frac c {v_0} \frac 1 {\bar{g} } \frac 1 {v_0'}  |f(v')  | \dd v'   \int_{\R^3}   \frac {1} {u_0} \mathpzc G_k\p{v,u,v',c}     \dd u   \Vert \langle v \rangle^{k}     h \Vert_{L^\infty} 
\\
\lesssim & \int_{\R^3 } \langle v \rangle^{k} \frac c {v_0} \frac 1 {|v'-v|  } \frac 1 {v_0'} f(v') \dd v'    \int_{0}^\infty  |u|   \frac 1 {u_0}      \mathpzc u(v_0 +u_0-v_0')  (1+   |v_0+u_0 - v_0'|^2 -c^2  )^{-k/2}   \dd|u|     \Vert \langle v \rangle^{k}     h \Vert_{L^\infty}  
\\
\lesssim & \int_{\R^3 } \langle v \rangle^{k} \frac c {v_0} \frac 1 {|v'-v|  } \frac 1 {v_0'} f(v')  (1_{|v| \le 2|v ' | } + 1_{|v| \ge 2|v'| }  \langle v \rangle^{-k+1 } )   \dd v'  \Vert \langle v \rangle^{k}     h \Vert_{L^\infty}  
 \\
\lesssim & (\Vert \langle v \rangle^k f   \Vert_{L^1_v}    + \Vert \langle v \rangle^k f   \Vert_{L^{p}_v}  )   \Vert \langle v \rangle^{k} h\Vert_{L^\infty} .
\end{align*}
So the case for $|g| \ge \frac 1 {100}$ is proved.
\item For the case $g \le \frac  1 {100}$, we have that 
$\frac {|v-u|}  {\sqrt{\langle v \rangle \langle u \rangle }}     \le c\frac {|v-u|}  {\sqrt{v_0 u_0}}     \le g \le \frac 1 {100}$.
Since $|v|>1$, we have $|v-u| \le \frac 1 {100}  \max \{   \langle v \rangle, \langle u \rangle  \} \le \frac {1} {50}  \max \{   |v|, |u|  \} ,$
which implies that 
\begin{equation}
\label{eq:gsmall1}
 \min \{   |v|, |u|  \}  \ge \frac {19} {20}  \max \{   |v|, |u|  \} 
\end{equation}
remind $g^2 = \bar{g}^2 + \tilde{g}^2$, we have $\bar{g} \le \frac 1 {100}$, thus we also have 
\begin{equation}
\label{eq:gsmall2}
 \min \{   |v'|, |u|  \}  \ge \frac {19} {20}  \max \{   |v'|, |u|  \},
\end{equation}
\eqref{eq:gsmall1} and \eqref{eq:gsmall2}  imply that $\min \{  |u|,  |v|, |v'|  \}  \ge \frac {9} {10}  \max \{  |u|, |v|, |v'|  \} $ as well as $\min \{  |u_0|,  |v_0|, |v'_0|  \}  \ge \frac {9} {10}  \max \{  |u_0|, |v_0|, |v'_0|  \} $, which first implies that  $v_0 + u_0 - v_0'  \ge \frac 1 5 v_0$. Next,
\[
v_0 + u_0 - v_0'  -c   = \frac {|v|^2} {v_0 + c} -  \frac {|v'|^2 - |u|^2  } {v_0' + u_0 } \ge \frac 4 5 \frac {|v|^2} {v_0' + u_0'} - \frac 1 5 \frac {|v'|^2  } {v_0' + u_0 } \ge  \frac 1 2 \frac {|v|^2} {v_0' + u_0'} \ge \frac 1 {5} \frac {|v|^2} {v_0},
\]
so we have 
\[
 \langle u \rangle^2 + \langle v' \rangle^2 + \langle v \rangle^2 \lesssim 1+ (v_0 + u_0 - v_0'  -c )(v_0 +u_0-v_0'  + c)  =    1+ (v_0 + u_0 - v_0' )^2 -c^2 .
\]
Now we give estimates for the case $|g| \le \frac 1 {100}$. Choosing  the axis of $u$ such that $v-v' = (0, 0, |v-v'|)$, denote $u$ as  $u = (\mathfrak{U}, u_3), \mathfrak U \in \R^2,$ then we have  $(v -v')\cdot u = |v-v'| u_3$, and $u^2 =\mathfrak U^2 +  w_3^2$, $\langle u \rangle^{-k} \le \langle {\mathfrak U} \rangle^{-k}$ for $k \ge 0$. Denote 
\[
v_\parallel = \frac {v-v'} {|v-v'|}   (  v \cdot \frac {v-v'} {|v-v'|}),\quad  v_\perp = v - v_\parallel ,\quad v_\perp \perp   v_\parallel ,\quad  v_\perp  = (\bar{v}_\perp, 0),\quad  v_\parallel = (0, 0, \bar{v}_\parallel) ,
\]
where $\bar{v}_\perp \in \R^2$ is the first two dimensional component of $v_\perp$, then we have that 
\[
|v - u| =| v_\parallel  + v_\perp -u| =  \sqrt{|\bar{v}_\parallel -u_3|^2 + |\bar{v}_\perp -\mathfrak U|^2} \ge |\bar{v}_\perp - \mathfrak U|.
\]
Thus under the condition $\bar g\le g \le \frac 1 {100}$, we have that 
$\frac { |\bar{v}_\perp -\mathfrak U|   }  {\sqrt{\langle v \rangle \langle u \rangle }} \lesssim \frac {|v-u|}  {\sqrt{\langle v \rangle \langle u \rangle }}    \lesssim g $, and
\[
|g|^{\gamma - 1} =|g|^{-(1-\gamma)} \lesssim  | \bar{g}|^{-\frac {1- \gamma} {3 }}  |   \frac { |\bar{v}_\perp -  \mathfrak U|   }  {\sqrt{\langle v \rangle \langle u \rangle }}  |^{-\frac {2(1- \gamma)} {3 }}     \lesssim \frac 1 {|v'-v|^{\frac {1- \gamma} {3 }}} \frac 1 { |\bar{v}_\perp - \mathfrak U|^{\frac {2(1- \gamma)} {3 }}    }  \langle q \rangle^{10}.
\]
Using that
\begin{align*}
  \delta \left (\frac {\bar{g} } 2   -\frac {u_0 (v_0- v_0' )   - u \cdot (v-v')  } {\bar{g}} \right)  =& \delta(\frac 1 {2 \bar{g} }   ( \bar{g}^2 -   2 u_0 (v_0  -v_0')  + 2u_3 |v-v'| ) ) 
=  \frac {\bar{g}} {|v'-v| }  \delta(\frac  {\bar{g}^2 -   2 u_0 (v_0-v_0') }  {2  |v'-v|}   + u_3    ) ,
\end{align*}
by Lemma \ref{basic lemma for L p estimate Hardy} and using the same spherical coordinates as  Lemma \ref{spherical coordinates reduction for the relativistic Boltzmann operator} we compute that for any $p > \frac {9} {5+\gamma}$, 
\begin{align*}
|\langle v \rangle^k Q^+_c(h, f) | \lesssim &\Vert \langle v \rangle^{k}     h \Vert_{L^\infty} \cdot\int_{\R^3 } \langle v \rangle^{k} \frac 1 {\bar{g} } \frac c {v_0'}  |f(v')  | \dd v'    \int_{\R^3}   \frac {|g|^{\gamma-1}} {u_0} \mathpzc u(v_0 +u_0-v_0')     1_{|g| \le \frac 1 {100}}  \langle v+u-v' \rangle^{-k}       \delta    \left(  \frac {\bar{g}^2 + \tilde{g}^2  -g^2  } {2 \bar{g}} \right)     \dd u   
\\
\lesssim & \int_{\R^3 } \frac 1 {\bar{g} } \langle v' \rangle^{k}   |f(v')  | \dd v'   \int_{\R^3}   |g|^{\gamma-1}   1_{|g| \le \frac 1 {100}}   \langle u \rangle^{-k}     \delta    \left(  \frac {\bar{g}^2 + \tilde{g}^2  -g^2  } {2 \bar{g}} \right)     \dd u   \Vert \langle v \rangle^{k}     h \Vert_{L^\infty} 
\\
\lesssim & \int_{\R^3 }    \frac {\langle v' \rangle^{k}  |f(v')|} {|v'-v|^{1+\frac {1- \gamma} {3 }  } }      \dd v' \int_{\R^2}  \frac 1 { |\bar{v}_\perp - \mathfrak U|^{\frac {2(1- \gamma)} {3 }}    }      \langle \mathfrak U  \rangle^{k-10} \dd \mathfrak U        \int_{\R} \delta(\frac  {\bar{g}^2 -   2 u_0 (v_0-u_0') }  {2 |v'-v|}   + u_3   )        \dd u_3      \Vert \langle v \rangle^{k} h\Vert_{L^\infty}  
\\
\lesssim &   (\Vert \langle v \rangle^k f   \Vert_{L^1_v}    + \Vert \langle  v \rangle^k f   \Vert_{L^p_v}  )   \Vert \langle v \rangle^{k} h\Vert_{L^\infty},  
\end{align*}
we then end the proof by noting that $\frac  {9} {5+\gamma} >\frac 3 2$ and $\frac  {9} {5+\gamma} >\frac {3} {3+\gamma}$ for $\gamma \in (-2, 0]$.
\end{enumerate}
\end{proof}
\begin{theorem}\label{L infty estimates for Q f f z version}
Suppose $\gamma \in (-2, 0]$ , then for any smooth functions $h, f$ and any $t\ge 0, x\in\mathbb R^3, 0< \alpha < \frac {5+\gamma} {3}$, $k \ge 10$, we have 
\begin{equation}
\label{eq:thmLinfty1}
\begin{aligned}
&|\langle x-t \hat{v} \rangle^{k}   \langle v \rangle^{4k+50}   Q_c  (h, f)|  + |\langle x-t \hat{v} \rangle^{k+10}   \langle v \rangle^{2k+20}   Q_c  (h, f)|  
\\
\lesssim &   (1+t)^{-\alpha} \left( \Vert \langle x-t \hat{v} \rangle^{k}   \langle v \rangle^{4k+50}    h  \Vert_{L^\infty_v}  + \Vert \langle x-t \hat{v} \rangle^{k+10}   \langle v\rangle^{2k+20}    h  \Vert_{L^\infty_v}  \right)\left( \Vert \langle x-t \hat{v} \rangle^{k}   \langle v \rangle^{4k+50}    f  \Vert_{L^\infty_v}  + \Vert \langle x-t \hat{v} \rangle^{k+10}   \langle v \rangle^{2k+20}   f  \Vert_{L^\infty_v}  \right).
\end{aligned}
\end{equation}
\end{theorem}
\begin{remark}
For $\gamma \in (-2, 0]$ we have that $\frac {5+\gamma} {3}>1$, which means that the convergence rate is larger than $1$. 
\end{remark}

\begin{proof} Denote $p=\frac {3} {\alpha}$, then  $p >\frac {9 } {5+\gamma} >\frac 3 {3+\gamma}$. For $Q^-_c$ term, by Lemma \ref{L 1 L infty estimate on x t v} and \eqref{upper bound for v phi},  we have
\begin{align*}
&|\langle x-t \hat{v} \rangle^{k_1}   \langle v \rangle^{k_2}  Q^-_c(h, f)  | \lesssim  \int_{\R^3 }  v_\phi \sigma(g,\vartheta)    |h(u)  |   |f(v)  |  \dd u   
\lesssim       \int_{\R^3 }  \left(  1+ \frac 1 {|v-u|^{ - \gamma} }    \right)    |h(u)  |   \dd u   \Vert    \langle x-t \hat{v} \rangle^{k_1}   \langle v \rangle^{k_2 }  f \Vert_{L^\infty_v}    
\\
\lesssim& (\Vert     h \Vert_{L^1_v}    + \Vert h  \Vert_{L^p_v} ) \Vert    \langle x-t \hat{v} \rangle^{k_1}   \langle v \rangle^{k_2 }  f \Vert_{L^\infty_v}  
\lesssim (1+  t)^{-\alpha }  \Vert    \langle x-t \hat{v} \rangle^{10}   \langle v \rangle^{10 } h \Vert_{L^\infty_v}       \Vert    \langle x-t \hat{v} \rangle^{k_1}   \langle v \rangle^{k_2 }  f \Vert_{L^\infty_v}  .
\end{align*}
Then we come to prove the $Q^+_c $ term, by \eqref{p 2 smaller than p' 2 q' 2} and \eqref{estimate x t hat p p' q'} we have that
\begin{align*}
\langle x-t \hat{v} \rangle^{k_1}   \langle v \rangle^{k_2}\lesssim \min   \{   \langle v' \rangle^{2k_1}    ,  \langle u' \rangle^{2k_1}  \}     (\langle x-t\hat{v'}\rangle^{k_1}   +  \langle  x-t\hat{u'} \rangle^{k_1} ) (\langle v' \rangle^{k_2}   +  \langle u' \rangle^{k_2} ) ,
\end{align*}
thus we have that 
\begin{align*}
|\langle x-t \hat{v} \rangle^{k_1}   \langle v \rangle^{k_2}  Q_c^+  (h, f)|  & \lesssim   \int_{\R^3} \int_{\mathbb{S}^2 }   v_\phi \sigma(g,\vartheta)          \langle x-t\hat{u'}\rangle^{k_1}     \langle u' \rangle^{k_2}  | h(u') |    \langle v' \rangle^{2k_1}        | f(v' )| \dd \omega \dd u 
\\
& +  \int_{\R^3} \int_{\mathbb{S}^2 }  v_\phi \sigma(g,\vartheta)     \min   \{   \langle v' \rangle^{2k_1}    ,  \langle u' \rangle^{2k_1}  \}         \langle u' \rangle^{k_2}    | h(u') |    \langle x-t\hat{v'}\rangle^{k_1}  | f(v' )| \dd \omega \dd u
\\
&  +  \int_{\R^3} \int_{\mathbb{S}^2 }  v_\phi \sigma(g,\vartheta)      \min   \{   \langle v' \rangle^{2k_1}    ,  \langle u' \rangle^{2k_1}  \}   \langle x-t\hat{u'}\rangle^{k_1}         | h(u') |   \langle v' \rangle^{k_2}      | f(v' )| \dd \omega \dd u
\\
&+  \int_{\R^3} \int_{\mathbb{S}^2 }  v_\phi \sigma(g,\vartheta)   \langle u' \rangle^{2k_1}    | h(u') |      \langle x-t\hat{v'}\rangle^{k_1}     \langle v' \rangle^{k_2}    | f(v' )|\dd \omega \dd u
\\
&\lesssim  Q^+_c(  \langle x-t \hat{v} \rangle^{k_1}   \langle v \rangle^{k_2}   |h|  , \langle v \rangle^{2k_1} |f| )  + Q_c^+(\langle v \rangle^{2k_1} |h| ,   \langle x-t \hat{v} \rangle ^{k_1}   \langle v \rangle^{k_2}|h| ) 
\\
&+ \min \{  Q^+_c(     \langle v \rangle^{k_2} \langle v \rangle^{2k_1}  |h|  ,  \langle x-t \hat{v}\rangle^{k_1}|f| )  ,  Q_c^+(  \langle v \rangle^{k_2}    |h|  , \langle v \rangle^{2k_1}   \langle x-t \hat{v} \rangle^{k_1}  |f| ) \}
\\
&+ \min \{  Q^+_c(  \langle x-t \hat{v} \rangle^{k_1}     |h|  , \langle v \rangle^{2k_1}\langle v \rangle^{k_2}  |f| )  ,  Q_c^+(  \langle x-t \hat{v} \rangle^{k_1}  \langle v \rangle^{2k_1}   |h|  , \langle v \rangle^{k_2}   |f| ) \}.
\end{align*}
By taking $k=10$ in Lemma \ref{upper bound estimate for Q + -} we have that 
\begin{equation}
\label{basic Q + estimate}
|\langle v \rangle^{10} Q^+_c(h, f)  | \lesssim \min \{  (\Vert \langle v \rangle^{10} h   \Vert_{L^1_v}    + \Vert \langle v \rangle^{10} h   \Vert_{L^p_v}  )    \Vert \langle v \rangle^{10} f\Vert_{L^\infty_v}   ,  (\Vert \langle v \rangle^{10} f   \Vert_{L^1_v}    + \Vert \langle v \rangle^{10} f   \Vert_{L^p_v}  )    \Vert \langle v \rangle^{10}       h   \Vert_{L^\infty_v}    \}  .
\end{equation}
Now we are ready to give the estimates for the left side of \eqref{eq:thmLinfty1}.
\begin{enumerate}
\item
We have
\begin{align*}
|\langle x-t \hat{v} \rangle^{k}   \langle v \rangle^{4k+50}   Q_c^+  (h, f)|  & \lesssim \langle v \rangle^{10}   |\langle x-t \hat{v} \rangle^{k}   \langle v \rangle^{4k+40}   Q_c^+  (h, f)| 
\\
&\lesssim  \langle v \rangle^{10} Q^+_c(  \langle x-t \hat{v} \rangle^{k}   \langle v \rangle^{4k+40}   |h|  , \langle v \rangle^{2k} |f| )  +  \langle v \rangle^{10} Q^+_c(\langle v \rangle^{2k} |h| ,   \langle x-t \hat{v} \rangle ^{k}   \langle v \rangle^{4k+40}    |f| ) 
\\
&+  \langle v \rangle^{10}     Q^+_c(  \langle v \rangle^{4k+40}    |h|  , \langle v \rangle^{2k}   \langle x-t \hat{v} \rangle^{k}  |f| ) + \langle v \rangle^{10}       Q^+_c(  \langle x-t \hat{v} \rangle^{k}  \langle v \rangle^{2k}   |h|  , \langle v \rangle^{4k+40}   |f| )\\
& \defeq T_1+T_2+T_3+T_4.
\end{align*}
From \eqref{basic Q + estimate} and Lemma \ref{L 1 L infty estimate on x t v}, we have 
\begin{equation}
\label{eq:thmLinfty1-1}
\begin{aligned}
T_1
&\lesssim  \Vert \langle x-t \hat{v} \rangle^{k}   \langle v \rangle^{4k+50}     h  \Vert_{L^\infty_v} ( \Vert \langle v \rangle^{2k+10} f   \Vert_{L^1_v}+  \Vert \langle v \rangle^{2k+10} f \Vert_{L^p_v})\\
&\lesssim   (1+t)^{-\alpha}    \Vert \langle x-t \hat{v} \rangle^{k}   \langle v \rangle^{4k+50}    h  \Vert_{L^\infty_v}\Vert \langle x-t \hat{v} \rangle^{k}   \langle v \rangle^{4k+50}   f  \Vert_{L^\infty_v}.
\end{aligned}
\end{equation}
Similarly, 
\begin{equation}
\label{eq:thmLinfty1-2}
T_2\lesssim (1+t)^{-\alpha}    \Vert \langle x-t \hat{v} \rangle^{k}   \langle v \rangle^{4k+50}  f  \Vert_{L^\infty_v}\Vert \langle x-t \hat{v} \rangle^{k}   \langle v \rangle^{4k+50}   h  \Vert_{L^\infty_v}.
\end{equation}Next, we have

\begin{equation}
\label{eq:thmLinfty1-3}
\begin{aligned}
T_3\lesssim & \Vert    \langle v \rangle^{4k+50}      h  \Vert_{L^\infty_v} ( \Vert \langle v \rangle^{2k+10}  \langle x-t \hat{v} \rangle ^{k}    f   \Vert_{L^1_v}+  \Vert \langle x-t \hat{v} \rangle ^{k}     \langle v \rangle^{2k+10} f \Vert_{L^p_v})
\\\lesssim  & (1+t)^{-\alpha}       \Vert \langle x-t \hat{v} \rangle^{k}   \langle v \rangle^{4k+50}    h  \Vert_{L^\infty_v}\Vert \langle x-t \hat{v} \rangle^{k+10}   \langle v \rangle^{2k+20}   f  \Vert_{L^\infty_v}.
\end{aligned}
\end{equation}
Similarly, 
\begin{equation}
\label{eq:thmLinfty1-4}
T_4\lesssim  (1+t)^{-\alpha}       \Vert \langle x-t \hat{v} \rangle^{k}   \langle v \rangle^{4k+50}    h  \Vert_{L^\infty_v}\Vert \langle x-t \hat{v} \rangle^{  k + 10  }   \langle v \rangle^{2k+20}   f  \Vert_{L^\infty_v}.
\end{equation}

\item
Next, we calculate
\begin{align*}
&|\langle x-t \hat{v} \rangle^{k+10}   \langle v \rangle^{2k+20}   Q_c^+  (h, f)|  \lesssim     \langle v \rangle^{10}   |\langle x-t \hat{v} \rangle^{k+10}   \langle v \rangle^{2k+10}   Q_c^+  (h, f)|   
\\
\lesssim& \langle v \rangle^{10}  Q^+_c(  \langle x-t \hat{v} \rangle^{k+10}   \langle v \rangle^{2k+10}   |h|  , \langle v \rangle^{2k+20} |f| ) + \langle v \rangle^{10} Q_c^+(\langle v \rangle^{2k+20} |h| ,   \langle x-t \hat{v} \rangle ^{k+10}   \langle v \rangle^{ 2 k+10  }  |h| ) 
\\
+& \langle v \rangle^{10}  Q^+_c(     \langle v \rangle^{2k+10} \langle v \rangle^{2k+20}  |h|  ,  \langle x-t \hat{v}\rangle^{k+10}|f| )   +\langle v \rangle^{10}   Q^+_c(  \langle x-t \hat{v} \rangle^{k+10}     |h|  , \langle v \rangle^{2k+10}\langle v \rangle^{2k+20}  |f| ) \\
\defeq & T_5+T_6+T_7+T_8.
\end{align*}
We have
\begin{equation}
\label{eq:thmLinfty1-5}
\begin{aligned}
T_5\lesssim & \Vert \langle x-t \hat{v} \rangle^{k+10}   \langle v \rangle^{2k+20} h  \Vert_{L^\infty_v} ( \Vert \langle v \rangle^{2k+30} f   \Vert_{L^1_v}+  \Vert \langle v \rangle^{2k+30} f \Vert_{L^p_v})\\
\lesssim & (1+t)^{-\alpha}    \Vert \langle x-t \hat{v} \rangle^{k+10}   \langle v \rangle^{2k+20}    h  \Vert_{L^\infty_v}\Vert \langle x-t \hat{v} \rangle^{k}   \langle v \rangle^{4k+50}   f  \Vert_{L^\infty_v}.
\end{aligned}
\end{equation}
Similarly,
\begin{equation}
\label{eq:thmLinfty1-6}
T_6\lesssim (1+t)^{-\alpha}    \Vert \langle x-t \hat{v} \rangle^{k+10}   \langle v \rangle^{2k+20}    h  \Vert_{L^\infty_v}\Vert \langle x-t \hat{v} \rangle^{k}   \langle v \rangle^{4k+50}   f  \Vert_{L^\infty_v}.
\end{equation}
Next,
\begin{equation}
\label{eq:thmLinfty1-7}
\begin{aligned}
T_7&\lesssim  \Vert \langle x-t \hat{v} \rangle^{k+10}   \langle v \rangle^{10} f  \Vert_{L^\infty_v} ( \Vert \langle v \rangle^{4k+40} h   \Vert_{L^1_v}+  \Vert \langle v \rangle^{4k+40} h \Vert_{L^p_v})\\
&\lesssim    (1+t)^{-\alpha}   \Vert \langle x-t \hat{v} \rangle^{k+10}   \langle v \rangle^{2k+20}    f  \Vert_{L^\infty_v}\Vert \langle x-t \hat{v} \rangle^{k}   \langle v \rangle^{4k+50}   h  \Vert_{L^\infty_v}.
\end{aligned}
\end{equation}
Similarly,
\begin{equation}
\label{eq:thmLinfty1-8}
T_8\lesssim     (1+t)^{-\alpha}   \Vert \langle x-t \hat{v} \rangle^{k+10}   \langle v \rangle^{2k+20}    f  \Vert_{L^\infty_v}\Vert \langle x-t \hat{v} \rangle^{k}   \langle v \rangle^{4k+50}   h  \Vert_{L^\infty_v}.
\end{equation}
\end{enumerate}
Gathering all the calculations \eqref{eq:thmLinfty1-1} to \eqref{eq:thmLinfty1-8} we have \eqref{eq:thmLinfty1}.
The theorem is thus proved. 
\end{proof}

\subsection{Chain rule for the relativistic Boltzmann operator}
\label{sec:proofchainrule}
This subsection is devoted to prove chain rule for the relativistic Boltzmann operator. We prove the following theorem that contains \Cref{thm:chainrule}.
\begin{theorem}
\label{lem:chainrule}
For any smooth functions $f,h$, for any $i, j=1, 2, 3, i \neq j$ we have
\begin{equation}
\label{eq:commutatorQ2extra}
v_0\partial_{v_j } Q_c (h, f)  = Q_c(h, (v_0\partial_{v_j} f))  + Q_c((v_0\partial_{v_j} h) , f)  - \frac {v_j} {v_0}Q_c(h, f)  ,
\end{equation}
and
\begin{equation}
\label{eq:commutatorQ2}
(v_j \partial_{v_i}  - v_i \partial_{v_j})    Q_c(h, f) = Q_c (h, (v_j \partial_{v_i}  - v_i \partial_{v_j})  f  ) +Q_c ((v_j \partial_{v_i}  - v_i \partial_{v_j})   h,  f  ).
\end{equation}
\end{theorem}
\begin{proof}
Since $g$ is the function of $v_0u_0- v\cdot u$, we have
\begin{equation}
\label{equality partial p i g partial q i g}
\partial_{v_i}  (v_0u_0- v\cdot u) =  \frac {v_iu_0 } {v_0} - u_i  = u_0 \left(\frac {v_i} {v_0}  - \frac {u_i} {u_0}  \right) =  - \frac{u_0} {v_0}    \partial_{u_i}  (v_0u_0- v\cdot u)\Longrightarrow \partial_{v_i} g= - \frac{u_0} {v_0}    \partial_{u_i}  g ,\,\,[v_0\partial_{v_i}  + u_0  \partial_{u_i}] g =0.
\end{equation}
Moreover,
\begin{equation}
\label{equality partial p i g partial q i g-2}
(v_j \partial_{v_i}  - v_i \partial_{v_j})  v_0 = \frac {v_iv_j} {v_0} - \frac {v_iv_j} {v_0} = 0 ,\quad(v_j \partial_{v_i}  - v_i \partial_{v_j})   (v\cdot u) =  u_i v_j - u_j v_i    =-   (u_j \partial_{u_i}  - u_i \partial_{u_j})     \partial_{u_i}  (v\cdot u).
\end{equation}
Thus we obtain from  \eqref{equality partial p i g partial q i g}, \eqref{equality partial p i g partial q i g-2} that
$(v_j \partial_{v_i}  - v_i \partial_{v_j}) g = -   (u_j \partial_{u_i}  - u_i \partial_{u_j})   g $.\\

Now we are ready to prove \eqref{eq:commutatorQ2extra}. We treat $Q^+_c$ and $Q^-_c$ seperately.
\begin{enumerate}
\item
 For $Q^+_c$, for simplicity, we define function
\begin{equation}
\label{eq:mathpzchsim1}
\mathpzc H(v,u,v',c):=c\sigma(g,\vartheta) \mathpzc u(v_0 +u_0-v_0')  h (v+u-v') \delta \left (  \frac {\bar{g}^2 + \tilde{g}^2  -g^2  } {2 \bar{g}}  \right).
\end{equation}Recall \Cref{lem:Carleman}, we obtain that
\begin{equation}
\label{eq:CarlemanH0}
Q^+_c(h, f)  =\frac 1 {v_0} \int_{\R^3 }  \frac 1 {8\bar{g} } \frac 1 {v_0'} f(v') \dd v'   \int_{\R^3}   \frac {s} {u_0}\mathpzc H(v,u,v',c)     \dd u .
\end{equation}
Therefore using integration by parts in $v'$ and $u$ we have
\begin{equation}
\label{eq:commutatorQ1-1}
\begin{aligned}
v_0\partial_{v_i} Q^+_c(h, f) 
=&   \frac 18 \int_{\R^3 }\left( v_0\partial_{v_i} (\frac 1 {v_0} )    \frac 1 {\bar{g} }+\partial_{v_i} (\frac 1 {\bar{g} } )\right) \frac 1 {v_0'} f(v') \dd v'   \int_{\R^3}   \frac {s} {u_0} \mathpzc H(v,u,v',c)     \dd u 
\\
+&  \frac 18 \frac 1 {v_0} \int_{\R^3 } \frac 1 {\bar{g} } \frac 1 {v_0'} f(v') \dd v'   \int_{\R^3}   \frac{v_0}{u_0}\pare{\partial_{v_i} (s   )\mathpzc H(v,u,v',c)+ s    \partial_{v_i}\mathpzc H(v,u,v',c)  }   \dd u 
\\
=&    - \frac 18\int_{\R^3 }\left( \frac {v_i}{v^2_0}    \frac 1 {\bar{g} }+\frac{v_0'}{v_0}\partial_{v'_i} (\frac 1 {\bar{g} } )\right) \frac 1 {v_0'} f(v') \dd v'   \int_{\R^3}   \frac {s} {u_0}   \mathpzc H(v,u,v',c)     \dd u \\
+&   \frac 18\frac 1 {v_0} \int_{\R^3 } \frac 1 {\bar{g} } \frac 1 {v_0'} f(v') \dd v'   \int_{\R^3}  \pare{ -\partial_{u_i} (s    )  \mathpzc H(v,u,v',c) +s     \frac {1} {u_0}    v_0\partial_{v_i} \left(\mathpzc H(v,u,v',c)         \right)}     \dd u
\\
=&  -\frac {v_i} {v_0} Q^+_c(h, f) + \frac 18 \frac 1 {v_0} \int_{\R^3 }   \frac 1 {\bar{g} } \frac 1 {v_0'}    v_0'   \partial_{v_i'} \left[  f(v') \dd v'   \int_{\R^3}   \frac {s} {u_0} 
\mathpzc H(v,u,v',c)         \right]  \dd u \\
+& \frac 18 \frac 1 {v_0} \int_{\R^3 } \frac 1 {\bar{g} } \frac 1 {v_0'} f(v') \dd v'   \int_{\R^3}     \frac s {u_0}\pare{   u_0   \partial_{u_i} \left (\mathpzc H(v,u,v',c)        \right)+v_0   \partial_{v_i} \left (\mathpzc H(v,u,v',c)            \right) } \dd u 
\\
=&  -\frac {v_i} {v_0} Q^+_c(h, f) + \frac 18 \frac 1 {v_0} \int_{\R^3 }   \frac 1 {\bar{g} } \frac 1 {v_0'}    v_0'   \partial_{v_i'} ( f(v') )   \dd v'   \int_{\R^3}   \frac {s} {u_0} \mathpzc H(v,u,v',c)     \dd u  
\\
+&  \frac 18 \frac 1 {v_0} \int_{\R^3 } \frac 1 {\bar{g} } \frac 1 {v_0'} f(v') \dd v'   \int_{\R^3}    \frac s {u_0} (  u_0   \partial_{u_i} + v_0\partial_{v_i} +v_0'\partial_{v_i'} )\left [\mathpzc H(v,u,v',c)         \right] \dd u 
\\
=&  -\frac {v_i} {v_0} Q_c^+(h, f)  +Q_c^+(h, v_0\partial_{v_i} f  )+ \frac 18 \frac 1 {v_0} \int_{\R^3 } \frac {f(v')} {\bar{g}v'_0 } \dd v'   \int_{\R^3}    \frac {s  } {u_0} (  u_0   \partial_{u_i} + v_0\partial_{v_i} +v_0'\partial_{v_i'} )\left [ \mathpzc H(v,u,v',c)        \right] \dd u.
\end{aligned}
\end{equation}
We define opeartor $\mathpzc L_1\p{v,u,v'}\defeq u_0   \partial_{u_i} + v_0\partial_{v_i} +v_0'\partial_{v_i'} $ for simplicity. From \eqref{equality partial p i g partial q i g} and direct calculations, we have
\begin{equation}
\label{eq:diractheta}
 \mathpzc L_1\p{v,u,v'}\begin{bmatrix} g\\\bar g\\\tilde g\end{bmatrix} = \begin{bmatrix}0\\0\\0\end{bmatrix},\quad\mathpzc L_1\p{v,u,v'}  (v_0+u_0 - v_0' ) = v_i + u_i - v_i',\quad \mathpzc L_1\p{v,u,v'}  (v_j+u_j-v'_j  )  = \delta_{ij} (v_0+u_0-v_0' ) ,
\end{equation}
and we obtain from \Cref{lem:vartheta} and \eqref{eq:diractheta} that $\mathpzc L_1\p{v,u,v'}\sigma\pare{g,\vartheta}=0$.
 
 Next, since $\mathpzc u'(x) = \delta(x-c)$ , we have $\mathpzc L_1\p{v,u,v'}\delta \left (  \frac {\bar{g}^2 + \tilde{g}^2  -g^2  } {2 \bar{g}}  \right)    = 0$
and $\mathpzc L_1\p{v,u,v'} \mathpzc u(v_0+u_0-v_0' ) = \delta(v_0+u_0 -v_0' -c) (v_i + u_i - v_i') $. Remind \eqref{g g g implies equality}, we have
\[
v_0+u_0 -v_0' -c=0,\quad  \frac {\bar{g}^2 + \tilde{g}^2  -g^2  } {2 \bar{g}} =0\quad\Longrightarrow v_i + u_i - v_i'=0, \quad i=1,2,3.
\]
Thus we obtain that
\begin{equation}
\label{eq:key}
\delta(v_0+u_0 -v_0' -c) (v_i + u_i - v_i')  \delta \left (  \frac {\bar{g}^2 + \tilde{g}^2  -g^2  } {2 \bar{g}}  \right) =0. 
\end{equation}
Finally we compute that 
\[
\mathpzc L_1\p{v,u,v'}      h (v+u-v')   =  \sum_{j} \partial_{v_j} h(v+u-v')  \mathpzc L_1\p{v,u,v'}  (v_j+u_j-v'_j  ) = (v_0+u_0-v_0' )  \partial_{v_i} h (v+u-v').
\]
Using again \eqref{g g g implies equality}, we have $\mathpzc L_1\p{v,u,v'}       h (v+u-v')    =   (v+u-v')_0  \partial_{v_i} h (v+u-v')$. Thus
\begin{equation}
\label{eq:commutatorQ1-2}
\begin{aligned}
& \frac 18\frac 1 {v_0} \int_{\R^3 } \frac {f(v')} {\bar{g}v_0' } \dd v'   \int_{\R^3}      \frac {s } {u_0} (  u_0   \partial_{u_i} + v_0\partial_{v_i} +v_0'\partial_{v_i'} )\left [  \mathpzc H(v,u,v',c)     \right] \dd u 
\\
= & \frac 18\frac 1 {v_0} \int_{\R^3 } \frac {f(v')} {\bar{g}v_0' } \dd v'  \int_{\R^3}   \frac {s } {u_0} \mathpzc u(v_0 +u_0-v_0')  [(v+u-v')_0 \partial_i h (v+u-v')  ] \delta \left (  \frac {\bar{g}^2 + \tilde{g}^2  -g^2  } {2 \bar{g}}  \right)    \dd u 
\\
= &  Q^+_c(v_0 \partial_{v_j} h,  f ).
\end{aligned}
\end{equation}
\item For $Q^-_c$ term, recall \begin{equation}
\label{eq:defQ-}
Q^-_c(h, f)  = \int_{\R^3} \int_{\mathbb S^2}   \frac {c g \sqrt{s} }  {4 v_0u_0}      \sigma(g,\vartheta)  h(u) f(v )   \dd\omega \dd u, \end{equation}
 we compute that 
 \begin{equation}
\label{eq:commutatorQ1-3}
\begin{aligned}
&v_0\partial_{v_j } Q_c^-(h, f) \\= & \int_{\R^3} \int_{\mathbb S^2}    \frac {1} {4v_0u_0   }     cg \sqrt{s}  \sigma(g,\vartheta)  h(u) v_0\partial_{v_j }  f(v )   \dd\omega \dd u  
\\
+ & \int_{\R^3} \int_{\mathbb S^2}   v_0\partial_{v_j } ( \frac {1} {v_0   })   \frac 1 {4u_0}  c s  \sigma(g,\vartheta)   h(u)  f(v )   \dd\omega \dd u  
+  \int_{\R^3} \int_{\mathbb S^2}  \frac {1} {v_0   }   \frac 1 {4u_0}    v_0\partial_{v_j } (   c s  \sigma(g,\vartheta)   )    h(u) f(v )   \dd\omega \dd u
\\
= &Q^-_c(h, v_0\partial_{v_j} f) 
- \frac {v_i} {v_0} \int_{\R^3} \int_{\mathbb S^2}   \frac {c s  \sigma(g,\vartheta) } {4v_0u_0}   h(u)   f(v )   \dd\omega \dd u 
-  \int_{\R^3} \int_{\mathbb S^2}  \frac 1 {4v_0u_0}    u_0\partial_{u_j } (   c s \sigma(g,\vartheta)  )    h(u)  f(v )   \dd\omega \dd u
\\  
= &Q^-_c(h, v_0\partial_{v_j} f) - \frac {v_i} {v_0}Q^-_c(h, f)    +  \int_{\R^3} \int_{\mathbb S^2}   \frac 1 {4v_0u_0}      c s  \sigma(g,\vartheta)  u_0\partial_{u_j }    h(u)   f(v )   \dd\omega \dd u 
\\
= &Q^-_c(h, v_0\partial_{v_j} f) - \frac {v_i} {v_0}Q^-_c(h, f)   + Q^-_c(v_0\partial_{v_j} h , f).
\end{aligned}
\end{equation}
\end{enumerate}
Combining  \eqref{eq:commutatorQ1-1}, \eqref{eq:commutatorQ1-2} and \eqref{eq:commutatorQ1-3}, we prove \eqref{eq:commutatorQ2extra}.
\footnote{This completes the proof of \Cref{thm:chainrule}.}\\

To show \eqref{eq:commutatorQ2}, we still seperate  $Q^+_c$ term and $Q^-_c$ term. 
\begin{enumerate}
\item For $Q^+_c$ term, similarly as before, we have
\begin{align*}
&(v_j \partial_{v_i}  - v_i \partial_{v_j})    Q^+_c(h, f) 
\\
=& \frac 18 (v_j \partial_{v_i}  - v_i \partial_{v_j})   (\frac 1 {p_0} )   \int_{\R^3 }  \frac {f(v')} {\bar{g}v_0' } \dd v'   \int_{\R^3}   \frac {s} {u_0}   \mathpzc H(v,u,v',c) \dd u + \frac 18 \frac 1 {p_0} \int_{\R^3 }  (p_j \partial_{p_i}  - p_i \partial_{p_j})   (\frac 1 {\bar{g} } )\frac {f(v')} {p_0'}  \dd v'   \int_{\R^3}   \frac {s } {q_0}   \mathpzc H(v,u,v',c) \dd u 
\\
+& \frac 18 \frac 1 {v_0} \int_{\R^3 } \frac {f(v')} {\bar{g}v_0' } \dd v'     \int_{\R^3}   \frac {v_0(v_j \partial_{v_i}  - v_i \partial_{v_j})    (s    ) } {u_0}   \mathpzc H(v,u,v',c) \dd u 
+  \frac 18\frac 1 {v_0} \int_{\R^3 } \frac {f(v')} {\bar{g}v_0' } \dd v'   \int_{\R^3}     \frac {s  } {u_0}    (v_j \partial_{v_i}  - v_i \partial_{v_j})  \left(  \mathpzc H(v,u,v',c)   \right)    \dd u 
\\
=   &  \frac 18 \frac 1 {v_0} \int_{\R^3 }  -(v_j' \partial_{v_i'}  - v_i' \partial_{v_j'})  (\frac 1 {\bar{g} } )\frac 1 {v_0'} f(v') \dd v'   \int_{\R^3}   \frac {s} {u_0}\mathpzc H(v,u,v',c)    \dd u 
\\
+& \frac 18 \frac 1 {v_0} \int_{\R^3 } \frac {f(v')} {\bar{g}v_0' } \dd v'   \int_{\R^3} \frac { -   (u_j \partial_{u_i}  - u_i \partial_{u_j})  (s    ) } {u_0}    \mathpzc H(v,u,v',c)  \dd u +  \frac 18\frac 1 {v_0} \int_{\R^3 }  \frac {f(v')} {\bar{g}v_0' } \dd v'  \int_{\R^3}      \frac {s} {u_0}    (v_j \partial_{v_i}  - v_i \partial_{v_j})   \left(   \mathpzc H(v,u,v',c)  \right) \dd u
\\
=
& \frac 18  \frac 1 {v_0} \int_{\R^3 }   \frac 1 {\bar{g} } \frac 1 {v_0'}    (v_j' \partial_{v_i'}  - v_i' \partial_{v_j'})  ( f(v') )   \dd v'   \int_{\R^3}   \frac {s} {u_0}   \mathpzc H(v,u,v',c) \dd u 
\\
+& \frac 18 \frac 1 {v_0} \int_{\R^3 } \frac 1 {\bar{g} } \frac 1 {v_0'} f(v') \dd v'   \int_{\R^3}      \frac s {u_0} (  (v_j \partial_{v_i}  - v_i \partial_{v_j})   + (v_j' \partial_{v_i'}  - v_i' \partial_{v_j'})  + (u_j \partial_{u_i}  - u_i \partial_{u_j}) )
\left [  \mathpzc H(v,u,v',c)   \right] \dd u 
\\
=&  Q^+_c (h, (v_j \partial_{v_i}  - v_i \partial_{v_j})  f  )
+  \frac 18\frac 1 {v_0} \int_{\R^3 } \frac {f(v')} {\bar{g}v_0' } \dd v'      \int_{\R^3}       \frac {s} {u_0} (  (v_j \partial_{v_i}  - v_i \partial_{v_j})   + (v_j' \partial_{v_i'}  - v_i' \partial_{v_j'})  + (u_j \partial_{u_i}  - u_i \partial_{u_j}) )
\left [   \mathpzc H(v,u,v',c) \right] \dd u .
\end{align*}
We define operator $\mathpzc L_2\p{v,u,v'}\defeq v_j \partial_{v_i}  - v_i \partial_{v_j}  + v_j' \partial_{v_i'}  - v_i' \partial_{v_j'} + u_j \partial_{u_i}  - u_i \partial_{u_j}$.  From direct calculations,
\begin{equation}
\label{eq:chainrule1}
\begin{aligned}
& \mathpzc L_1\p{v,u,v'}\begin{bmatrix} g\\\bar g\\\tilde g\end{bmatrix} = \begin{bmatrix}0\\0\\0\end{bmatrix},\quad \mathpzc L_2\p{v,u,v'} \vect{  v_0 +u_0-v_0'}{ \delta \left (  \frac {\bar{g}^2 + \tilde{g}^2  -g^2  } {2 \bar{g}}  \right)  } = \vect{0}{0},
\\
&\mathpzc L_2\p{v,u,v'}  (v_k+u_k-v_k' )  =( v_j +u_j - v_j') \delta_{ik} -  (v_i +u_i - v_i' )\delta_{jk} .\\
\end{aligned}
\end{equation}
Thus we obtain from \eqref{eq:chainrule1}  that 
\[
\mathpzc L_2\p{v,u,v'}   (f(v+u-v')  ) =( v_j +u_j - v_j') (\partial_{i} f )(v+u-v')-  (v_i +u_i - v_i' ) (\partial_j f )(v+u-v'),
\]
which implies that 
\begin{equation}
\label{eq:chainrule2-1}
(v_j \partial_{v_i}  - v_i \partial_{v_j})    Q_c^+(h, f)  = Q_c^+(h, (v_j \partial_{v_i}  - v_i \partial_{v_j})  f  ) +Q_c^+ ((v_j \partial_{v_i}  - v_i \partial_{v_j})   h,  f  ).
\end{equation}
\item
Recall that $
(v_j \partial_{v_i}  - v_i \partial_{v_j}) g = -   (u_j \partial_{u_i}  - u_i \partial_{u_j})   g $
and  \eqref{eq:defQ-}, we compute that 
\begin{equation}
\label{eq:chainrule2-2}
\begin{aligned}
&(v_j \partial_{v_i}  - v_i \partial_{v_j})  Q_c^-(h, f)\\ = & \int_{\R^3} \int_{\mathbb S^2}    \frac {1} {4v_0u_0   }     cg \sqrt{s} \sigma(g,\vartheta)  h(u)(v_j \partial_{v_i}  - v_i \partial_{v_j}) f(v) \dd\omega \dd u 
\\
+ & \int_{\R^3} \int_{\mathbb S^2}  (v_j \partial_{v_i}  - v_i \partial_{v_j})  ( \frac {1} {v_0   })   \frac 1 {4u_0}  c s \sigma(g,\vartheta)   h(u)  f(v )   \dd\omega \dd u   +  \int_{\R^3} \int_{\mathbb S^2}    \frac 1 {4v_0u_0}   (v_j \partial_{v_i}  - v_i \partial_{v_j})    (   c s  \sigma(g,\vartheta)  )    h(u)   f(v )   \dd\omega \dd u  
\\
= &Q^-_c(h, (v_j \partial_{v_i}  - v_i \partial_{v_j})  f) -  \int_{\R^3} \int_{\mathbb S^2}    \frac 1 {4v_0u_0}    (u_j \partial_{u_i}  - u_i \partial_{u_j})   (   c s \sigma(g,\vartheta)  )    h(u)  f( v )   \dd\omega \dd u\\  
= & Q^-_c(h, (v_j \partial_{v_i}  - v_i \partial_{v_j})  f)  \\
 + & \int_{\R^3} \int_{\mathbb S^2}  \frac {1} {v_0   }  (u_j \partial_{u_i}  - u_i \partial_{u_j})  ( \frac 1 {4u_0} )     cs \sigma(g,\vartheta)      h(u)   f(v)   \dd\omega \dd u
+  \int_{\R^3} \int_{\mathbb S^2}    \frac 1 {4v_0u_0}      c s  \sigma(g,\vartheta)  (u_j \partial_{u_i}  - u_i \partial_{u_j})    h(u)   f(v )   \dd\omega \dd u 
\\
= &Q^-_c(h, (v_j \partial_{v_i}  - v_i \partial_{v_j})   f)    + Q^-_c((v_j \partial_{v_i}  - v_i \partial_{v_j})     h , f),
\end{aligned}
\end{equation}
\end{enumerate}
Combine \eqref{eq:chainrule2-1} and  \eqref{eq:chainrule2-2} together, we obtain \eqref{eq:commutatorQ2}. This completes the proof of \Cref{lem:chainrule}.
\end{proof}

From \Cref{lem:chainrule}, we directly obtain that
\begin{cor}
\label{cor:chainrule1}
We have
\[
\hat{Z} Q_c(f, g) = Q_c(\hat{Z}f, g ) + Q_c(f, \hat{Z} g ), \quad \hat{Z} =\partial_t, \partial_{x_i},\hat{\Omega}_{ij}, S,
\]
and
 \[
 \hat{\Omega}_{i} Q_c(f, g) = Q_c( \hat{\Omega}_{i} f, g ) + Q_c(f, \hat{\Omega}_{i} g ) - \frac{v_i}{v_0} Q_c(f, g).
 \]
 Therefore, the relativistic Boltzmann operator satisfies a commutator structure compatible with the vector field algebra.
\end{cor}
\begin{cor}
\label{cor:chainrule2}
We have
\begin{equation}
\label{accumulated chain rule for the Z b Q f f}
|   \frac 1 {v_0}  (\hat{Z}^\alpha v_0 Q_c(f, f) )|  \lesssim \sum_{\alpha_1+\alpha_2\le \alpha} |Q_c(\hat{Z}^{\alpha_1} f,  \hat{Z}^{\alpha_2} f)|.
\end{equation}
\end{cor}

\begin{proof}
Notice that $|\frac 1 {v_0} \hat{Z}^\beta  (v_0)| \lesssim 1, | \hat{Z}^\beta (\frac {v_i} {v_0})| \lesssim 1$ for any $\beta$,
so \eqref{accumulated chain rule for the Z b Q f f} follows by  \Cref{lem:chainrule} and \Cref{cor:chainrule1}.
\end{proof}

\appendix

\section{Appendix}
\label{appendix}
\subsection{Proof of estimates about weight function}
\label{app:proof}
In this subsection, we prove \Cref{L 1 L infty estimate on x t v} to \Cref{lem:xgect}.
\begin{proof}[Proof of \Cref{L 1 L infty estimate on x t v}]
 For $t\le 1$ this is just follows by 
$\Vert f \Vert_{L^1_v} \lesssim \Vert \langle v \rangle^4  f \Vert_{L^\infty_v} $. For $t \ge 1$, from direct calculations,
\[
\frac {\partial  \hat{v}_i }{\partial v_j} = \frac {c \delta_{ij}} { v_0   }   - \frac {c v_i v_j} { v_0^3  } = \frac {c} {v_0} \left[\delta_{ij} - \frac {v_i v_j} {v_0^2 }  \right]\quad\Longrightarrow\quad\det \av{\frac {\dd \hat{v}   } {\dd v}} =  \frac {c^3} {v_0^3}  \det \left| I_{3\times 3} - \frac {v} {v_0}\otimes  \frac {v} {v_0}  \right| =\frac {c^3} {v_0^3}  \pare{1-| \frac {v} {v_0}|^2} = \frac {c^5} {v_0^5}.
\]
Thus for any $(t, x) \in \R_+ \times \R^3$, the Jacobian determinant of the transformation $v \mapsto x-\widehat{v}t$ is equal to $-\frac{c^5  t^3}{|v^0|^5}$. 
Thus for any smooth function $h$, remind the operator $y\to \check y$ defined in \eqref{eq:opeatorcheck},  we have that 
\begin{equation}
\label{change of variable x - t hat v}
\int_{\R^3} t^3 h(t, x -t \hat{v}   , v) \dd v = \int_{|y-x| < c t}  \frac {v_0^5} {c^5}  h\pare{   t, y, \widecheck{\frac {x-y} {t}} }\dd y.
\end{equation}
Next for any $k >3$ we have that 
\begin{align*}
\int_{\R^3}   \langle  x -t \hat{v} \rangle^{-k } t ^3 \langle v \rangle^{-5} \dd v    \lesssim  \int_{\R^3}   \langle  x -t \hat{v} \rangle^{-k } t ^3 \frac {c^5} {v_0^5} \dd v  \lesssim \int_{\R^3} \langle u \rangle^{-k} \dd u \lesssim 1,\quad \forall t\ge 0,\quad  x \in \R^3,
\end{align*}
thus for any $p\in [1,\infty)$, we have
\begin{align*}
\Vert f \Vert_{L^p_v}  =\left(\int_{\R^3 }   |f|^p (v) \dd v \right)^{\frac 1p} \lesssim \Vert \langle v \rangle^{5}  \langle  x -t \hat{v} \rangle^{4 }  f\Vert_{L^\infty_v}  \left( \int_{\R^3}   \langle  x -t \hat{v} \rangle^{-4p }  \langle v \rangle^{-5p} \dd v    \right)^{\frac 1p}  \lesssim t^{-\frac 3p}\Vert \langle v \rangle^{5}  \langle  x -t \hat{v} \rangle^{4 }  f\Vert_{L^\infty_v} .     
\end{align*}
The proof is thus finished. 

\end{proof}
\begin{proof}
[Proof of \Cref{basic lemma for L p estimate Hardy}]
Denote ${ p}' =\frac {p} {p-1}  $, then we have  $p' \gamma>-3$, thus using H\"older inequality we easily compute that 
\begin{equation*}
 \int_{|u-v| \ge 1 } |u-v|^\gamma f(u) \dd u \lesssim   \Vert f \Vert_{L^1_v} ,  \quad \int_{|u-v| \le 1 } |u-v|^\gamma f(v) \dd v 
\lesssim   \int_{|u-v| \le 1 } |u-v|^{ p'  \gamma}   \dd u     \Vert f \Vert_{L^{ p}_v}\lesssim   \Vert f \Vert_{L^{ p}_v},
\end{equation*}
the proof is thus finished by using Lemma \ref{L 1 L infty estimate on x t v}. 
\end{proof}
\begin{proof}[Proof of \Cref{lem:extraesti}]
We compute that 
\begin{equation}
\label{inequality for v L}
2\kappa(v)  v_0^2  \ge \kappa(v)   v_0 (v_0 + \frac {v \cdot x} {r}) = (v_0 -  \frac {v \cdot x} {r}) (v_0 + \frac {v \cdot x} {r}) =   v_0^2 - |v \cdot   \frac {x} {r} |^2 = c^2 + |v \times \frac {x} {r} |^2,
\end{equation}
which implies that $\frac { |v \times \frac {x} {r} | } {v_0 }, \frac {c } {v_0 } \lesssim \sqrt{\kappa(v) }$. Together with that $v \cdot e_1' = v \cdot \frac {x} {r}$ proves \eqref{inequality for hat v L}. To prove \eqref{main inequality from 1 t - r c to 1 + t + r}, we first compute
\[
t\kappa(v) v_0 = t v_0 - \frac {tv\cdot x} {r}  = t  v_0 - \frac {r}{c}v_0  + \frac {|x|^2} {rc}  v_0  - \frac {tv\cdot x} {r}  = (t-\frac {r}{c} ) v_0 +\frac {v_0} {cr} [|x|^2 - \frac {c t v \cdot x} {v_0}]       = (t-\frac {r}{c} ) v_0  +  \frac {v_0} {cr}  x \cdot (x-t \hat{v}),
\]
thus together with \eqref{inequality for v L}, we have that
\begin{equation}
\label{main inequality from 1 t - r c to 1 + t + r-1}
c^2 (1 + t )   \lesssim v_0^2   \kappa(v)    (1+ t)   \lesssim  (1+ |t-\frac {r} {c} | ) v_0^2 + \frac {v_0^2} {c}|x-t \hat{v}|\Longrightarrow(1+ t) \lesssim (1 +  |t-\frac {r} {c} | ) \frac {v_0^2} {c^2} \langle x-t \hat{v} \rangle \lesssim (1 +  |t-\frac {r} {c} | ) \langle v \rangle^2 \langle x-t \hat{v} \rangle.
\end{equation}
We split the next step it into three  cases. 
\begin{enumerate}
    \item {$|\hat{v}| \le 1$}. We have $t +|x -\hat{v} t | \ge |x|$, thus 
$(1+t +|x|)  \lesssim (1+t) (1+|x -  t \hat{v}  |)\lesssim (1+t) \langle  x - t \hat{v}  \rangle$.
\item{$|\hat{v}| \ge 1, |x-t\hat{v}| \ge 3 t| \hat{v}| $ }. We have  $|x| \ge 2t   |\hat{v} |$ and thus  $|x-t\hat{v}| \ge \frac 1 2|x| $. So
$(1+t + |x| ) \lesssim (1 + t |\hat{v}| + |x| )  \lesssim \langle  x - t \hat{v}\rangle  $.
\item{$|\hat{v}| \ge 1,  |x-t\hat{v}| \le 3 t| \hat{v}|$}. Then  $|x | \le 4 t | \hat{v}|$, which implies that 
$(1+t +|x|)  \lesssim (1+t + t|\hat{v} |)    \lesssim (1+t)  \langle  v \rangle^2$.

\end{enumerate}
thus gathering all the cases, we have that
\begin{equation}
\label{main inequality from 1 t - r c to 1 + t + r-2}
1+t +|x|  \lesssim (1+t)  \langle  v \rangle^2  \langle x-t \hat{v} \rangle,
\end{equation}
and 
\eqref{main inequality from 1 t - r c to 1 + t + r} is thus proved by gathering 
\eqref{main inequality from 1 t - r c to 1 + t + r-1} and \eqref{main inequality from 1 t - r c to 1 + t + r-2}.
\end{proof}
\begin{proof}[Proof of \Cref{lem:xgect}]
\eqref{eq:xgect} is direct if $|x| \le 1$ since $1+t + |x|  \lesssim 1$.
If $|x| \ge\max\{ c t,1\}$, then 
\[
6\langle v \rangle^2  \langle x-t\hat v  \rangle  \ge 6\langle v \rangle^2  (|x| - t|\hat{v}|) \ge  6 \langle v \rangle^2  |x|(1 - |\frac {v} {v_0} |) \ge 2\langle v \rangle^2 (1 - |\frac {v } {v_0} |) (1+ c t+ |x|) \ge  1 + t + |x| ,
\]
and the proof of \eqref{eq:xgect} is thus finished .
\end{proof}
\subsection{Some technical lemmas}

In this subsection, we establish the algebraic foundation and analytical estimates required for the proofs. We begin by recalling several fundamental vector identities.
\begin{lemma}
 For any 3D vectors ${\textbf a}, {\textbf b}, {\textbf c}, {\textbf d}$, we have
\begin{subequations}
\begin{align}
&({\textbf a} \times {\textbf b}) \cdot ({\textbf c} \times {\textbf d}) = ({\textbf a} \cdot {\textbf c}) ({\textbf b} \cdot {\textbf d}) - ({\textbf a} \cdot{\textbf d}) ({\textbf b} \cdot{\textbf c}) ,\quad ({\textbf a} \times{\textbf b}) \times{\textbf c} = {\textbf b}({\textbf a} \cdot{\textbf c}) -{\textbf a} ({\textbf b} \cdot {\textbf c}),\quad{\textbf a} \cdot ({\textbf b} \times{\textbf c}) ={\textbf b} \cdot ({\textbf c}\times{\textbf a}) ={\textbf c} \cdot ({\textbf a} \times{\textbf b}),\label{basic identities for cross product}\\
&\quad\quad\quad\qquad\quad\quad\quad\qquad\quad\quad\quad\qquad({\textbf a} \times{\textbf b}) \times{\textbf c} +  ({\textbf b} \times{\textbf c}) \times{\textbf a} + ({\textbf c} \times{\textbf a}) \times{\textbf b} =0,\label{basic identities for cross product-1}\\
&\quad\quad\quad\qquad\quad\quad\quad\qquad\quad\quad\quad\qquad{\textbf a}{\textbf b}+{\textbf c}{\textbf d} = \frac 1 2[({\textbf a}+{\textbf c}) ({\textbf b}+{\textbf d}) + ({\textbf a}-{\textbf c}) ({\textbf b}-{\textbf d})],\label{basic identities for cross product-3}
\\
&{\textbf a} \times e_2' = ({\textbf a} \cdot e_1') e_3' - ({\textbf a} \cdot  e_3') e_1' ,\quad {\textbf a} \times e_3' = ({\textbf a} \cdot e_2') e_1' - ({\textbf a} \cdot e_1') e_2',\quad  {\textbf a} \times e_1' = ({\textbf a} \cdot e_3') e_2' - ({\textbf a} \cdot e_2') e_3 \label{basic identities for cross product-2}.
\end{align}
\end{subequations}
\end{lemma}
\begin{lemma}
\label{lem:yxct1}
For any $x,y\in\mathbb R^3, c\ge 1, t\ge 0$ that satisfy $\av x\le \frac 14 t, |y-x|\le ct$, we have
\begin{equation}
\label{eq:yxct1}
|y-x|+1+t\lesssim 1+t-\frac {|y-x|}c+|y|.
\end{equation}
\end{lemma}
\begin{proof}
If $|y-x|\le \frac c2t$, then $t-\frac{|y-x|}c\ge \frac t2$, and $|y-x|+1+t\lesssim |y|+|x|+1+t\lesssim 1+t-\frac {|y-x|}c+|y|$. If $|y-x|\in [\frac c2 t,ct]$, we  deduce \eqref{eq:yxct1} from $|y|\ge |y-x|-|x|\ge \frac c4t\ge\frac 14|y-x|$ and $c\ge 1$.
\end{proof}
\begin{lemma}
We have the following integral inequalities.
\begin{itemize}
\item For any $\mathfrak a, \mathfrak b, \mathfrak c\in\mathbb R, \mathfrak a\le \mathfrak b$, we have
\begin{equation}
\label{lem:yxct0}
\int_{\mathfrak a}^{\mathfrak b} \frac 1 {1 + |z-\mathfrak c|} \dd z \lesssim \log (1+|\mathfrak a-\mathfrak c|) + \log (1+|\mathfrak b-\mathfrak c| ).\end{equation}
\item Given $N>1$. For any
 $\mathfrak a, \mathfrak b, \mathfrak c>0$,\begin{equation}
\label{lem:yxct2}
\int_{0}^{\mathfrak b}\frac 1 {{\mathfrak a}^N +\mathfrak c^Nz^N} \dd z  \lesssim_N  \frac {1} {\mathfrak a^{N-1}\mathfrak c}.
\end{equation}
\item For any $\mathfrak a>0$, we have
\begin{equation}
\label{eq:yxct3}
\int_1^\infty \frac{1}{z\pare{z+\mathfrak a}^3}\dd\mathpzc z<\frac{\log\pare{1+\mathfrak a}}{\mathfrak a^3}.
\end{equation}
\end{itemize}
\end{lemma}

\begin{proof}
\Cref{lem:yxct0} is obvious. \Cref{lem:yxct2} can be obtained from
\begin{equation}
\label{eq:mathfrakABC2}
\int_{0}^{\mathfrak b}\frac 1 {{\mathfrak a}^N +\mathfrak c^Nz^N} \dd z = \frac 1{\mathfrak c} \int_{0}^{\mathfrak c \mathfrak b} \frac 1 {\mathfrak a^N +w^N } \dd u = \frac {1} {\mathfrak a^{N-1}\mathfrak c}  \int_{0}^{\frac {\mathfrak b \mathfrak c }{\mathfrak a }} \frac 1  {1 +y^{N-1} } \dd y \lesssim_N  \frac {1} {\mathfrak a^{N-1}\mathfrak c}.
\end{equation}
and \Cref{eq:yxct3} is the direct result from
\[
    \int_1^\infty \frac{1}{z\pare{z+\mathfrak a}^3}\dd z=\frac{\log\pare{1+\mathfrak a}}{\mathfrak a^3}-\frac{1}{\mathfrak a^2\pare{1+\mathfrak a}}-\frac{1}{2\mathfrak a\pare{1+\mathfrak a}^2}<\frac{\log\pare{1+\mathfrak a}}{\mathfrak a^3}.
    \]
\end{proof}
\begin{lemma}\label{Basic inequality for x - y = t} Suppose $f$ is a smooth function in $\R^3$ such that
$|f(x)| \lesssim K (1+|x|)^{-k}$ for some constant $K >0, k \ge 3$, then for any $t \ge 0, x, y \in \R^3$ we have 
\begin{equation}
\label{eq:inequalitysphere1}
\int_{|x-y|  = t} f(y) \dd S_y \lesssim      t^i (1+t +|x| )^{-1} (1+|t-|x||)^{-(k-2) }, \quad i=1,2,
\end{equation}
where the constant is independent of $x, t, y$. 
\end{lemma}
\begin{proof}
The case $i=1$ follows from \cite{wei20213d}, Lemma 4.1. For the case $i=2$, when $ t \ge \frac 1 2$, it is direct from $i=1$.  When $t \le \frac 1 2$, we have $ |y| \ge |x| -\frac 1 2 $, which implies that $\langle x \rangle \le 2 \langle y \rangle$ and $3\langle x \rangle \ge 1+t+|x|$, thus we compute that
\[
\int_{|x-y|  = t} f(y) \dd S_y \lesssim   t^2 \langle x \rangle^{-k} \lesssim  t^2 ( 1+t +|x| )   ^{-k} ,
\]
so the lemma is thus proved. 
\end{proof}
\subsection{Integral inequalities}
In this subsection, we prove three integral bounds,  which are the general versions of Lemma 5.10 to 5.12 of  \cite{bigorgne2025global} for any $c\ge 1$, and they will be used to estimate the Maxwell field $(E,B)$. We recall a change of variable on the sphere $|x-y| = \sigma$. 
\begin{lemma}[\cite{Glassey-1996-SIAM}, Lemma 6.5.2 ]\label{basic change of variable on the sphere x y sigma}
For any $\sigma \ge 0, x, y \in \R^3$, for any smooth function $f$ we have 
\[
\int_{|x-y| = \sigma}   f(|y|) \dd S_y = \frac {2\pi \sigma } {|x|} \int_{| |x| -\sigma|}^{|x| +\sigma} y  f( y) \dd  y. 
\]
\end{lemma}
This could implies that 
\begin{lemma}\label{basic change of variable on the sphere x y sigma improved version}
For any smooth function $\mathpzc g$ and any function $\mathpzc h$,  for any $t, \sigma \ge 0,c \ge 1, x, y \in \R^3$, we have that 
\begin{equation}
\label{eq:changevariable1}
\begin{aligned}
\int_{|y-x| \le c t} \mathpzc g(t - \frac 1 c |y-x|, |y|) \mathpzc h(|x-y| )\dd y
= &c\frac {2 \pi} {|x|}  \int_{0}^{t} \int_{| |x| - c(t - z) |}^{|x| +c(t-z) } y\mathpzc g(z  , y   ) \dd y c(t- z) \mathpzc h(c (t- z) ) \dd z
\\
= &c\frac {2 \pi} {|x|}  \int_{0}^{t} \int_{| |x| - c  z' |}^{|x| +cz' }  y\mathpzc g( t-z'   , y   ) \dd  y c z'  \mathpzc h(  c z' ) \dd z'.
\end{aligned}
\end{equation}
\end{lemma}
\begin{proof}
By Lemma \ref{basic change of variable on the sphere x y sigma} we have that 
\begin{align*}
\textnormal{l.h.s. of }\eqref{eq:changevariable1} =&  \int_{0}^{c t} \int_{|x-y| =\tau}\mathpzc g(t-\frac 1 c \tau, |y|   ) \dd S_y \mathpzc h(\tau) \dd\tau
=c\int_{0}^{t} \int_{|x-y| =c\tau_1}\mathpzc g(t-\tau_1 , |y|   ) \dd S_y\mathpzc h(c\tau_1 ) \dd\tau_1\\
=& c\int_{0}^{t} \int_{|x-y| =c (t- z) }\mathpzc g( z  , |y|   ) \dd S_y \mathpzc h(c (t-z) ) \dd z
=c\frac {2 \pi} {|x|}  \int_{0}^{t} \int_{| |x| - c(t - z) |}^{|x| +c(t-z) } y\mathpzc g(z  , y   ) \dd y c(t- z) \mathpzc h(c (t- z) ) \dd z\\
= &\textnormal{r.h.s. of }\eqref{eq:changevariable1},
\end{align*}
so the proof is thus finished. 
\end{proof}
We are now ready to prove the integral bounds.
\begin{lemma}\label{lem:inequality for y x c t 1}
For any $t \ge 0, x, y \in \R^3, c \ge 1, a>3$, we have that 
\begin{equation}
\label{eq:inequality for y x c t1}
\mathcal I_1  = \int_{|y-x| \le c t} \frac 1 {(1+  t   -\frac 1 c |y-x| +|y|  )^a (1+  \left| t-\frac 1 c |y-x| -\frac 1 c| y|  \right| ) } \frac 1 {|x-y| } \dd y  \lesssim \frac {c  \log (3+ | t-\frac {|x|} c | )}   {(1+t+|x|) (1+ | t-\frac {|x|} c |)^{a-2} }.
\end{equation}
where the constant is independent of $t, y, x, c$. 
\end{lemma}
\begin{proof}
First, by choosing $\mathpzc g\pare{x,y}=\frac{1}{\pare{1+x+y}^a\pare{1+\av{x-\frac yc}}}, \mathpzc h\pare x=\frac 1x$ in Lemma \ref{basic change of variable on the sphere x y sigma improved version}, we compute that 
\begin{align*}
\mathcal I_1 = &c\frac {2 \pi} {|x|}  \int_{0}^{t} \int_{| |x| - c(t - z) |}^{|x| +c(t-z) } y\mathpzc g(z  , y   ) \dd y c(t- z) \mathpzc h(c (t- z) ) \dd z
\le c\frac {2 \pi} {|x|}  \int_{0}^{t} \int_{| |x| - c(t - z) |}^{|x| +c(t-z) }   (1+ z+ y)^{1-a}   \pare{1 + \av{  z - \frac { y} {c} }  }^{-1}  \dd y   \dd z.
\end{align*}
We split it into two cases.
\begin{enumerate}
    \item {{ $|x | \le \frac 1 4 t$ }}. Remind $ y\ge | |x| - c(t -z) |,\,\,  z\ge 0$, we have
$1+ z+ y \ge1+  t-\frac {|x| } c \ge1+ \frac 1 2 t  \ge \frac 1 4 (1+t+|x|)$. So
\begin{align*}
(1+t+|x|)^{a-1} \mathcal I_1 \lesssim  & \frac {c} {|x|}  \int_{0}^{t} \int_{| |x| - c(t - z) |}^{|x| +c(t- z) }   \pare{1 + \av{ z - \frac { y} {c} }  }^{-1}  \dd  y   \dd z\\
= &\frac {c^2} {|x|}  \int_{0}^{t} \int_{| \frac {|x|} {c} - (t- z) |}^{\frac {|x|} {c} +t- z}    (1 + |  z - z' |  )^{-1}  \dd z' \dd z
=  \frac {c^2} {|x|}  \int_{0}^{t} \int_{| \frac {|x|} {c} -y  |}^{\frac {|x|} {c} + y}    (1 + |  t- y - z' |  )^{-1}  \dd z' \dd y
\\
= &  \frac {c^2} {|x|}  \int_{0}^{ \frac {|x|} c } \int_{ \frac {|x|} {c} -  y   }^{\frac {|x|} {c} + y}    (1 + |  t- y - z' |  )^{-1}  \dd z' \dd  y +  \frac {c^2} {|x|}  \int_{ \frac {|x|} c }^t  \int_{y-  \frac {|x|} {c}    }^{\frac {|x|} {c} +y}    (1 + |  t-y - z' |  )^{-1}  \dd z' \dd y\defeq T_1+T_2.
\end{align*}
Recall \Cref{lem:yxct0}, we have 
\begin{equation}
    \label{eq:inequality for y x c t1-1}
   T_1 \lesssim c \log (3+t + \frac {|x| } {c} )  \lesssim    c \log (3+ | t-\frac {|x| } {c} |  ),
\end{equation}
and
\begin{equation}
    \label{eq:inequality for y x c t1-2}
\begin{aligned}
T_2=&\frac {c^2} {|x|}  \int_{ \frac {|x|} c }^t  \int_{ 0    }^{\frac {2 |x|} {c} }    (1 + |  t +\frac {|x|} c- 2 y - z |  )^{-1}  \dd z \dd  y
\lesssim c   \int_{ \frac {|x|} c }^t  \int_{ 0    }^{1 }    \pare{1 + \av{  t +\frac {|x|} c- 2 y - 2\frac {|x|} {c} z }  }^{-1}  \dd z\dd y\\
=& c  \int_{ 0    }^{1 } \int_{ \frac {|x|} c }^t      \p{1 + \av{  t +\frac {|x|} c- 2y - 2\frac {|x|} {c} z}  }^{-1} \dd y \dd z 
\lesssim  c \log (3+t )  \lesssim    c \log (3+ | t-\frac {|x| } {c} |  ).
\end{aligned}
\end{equation}
\item {{ $|x | > \frac 1 4 t$ }}. Let $\mathfrak u =z+ y, \mathfrak v = z-y $, then  $
z-\frac {y} c = \frac { \mathfrak u+ \mathfrak v}{2}   - \frac { \mathfrak u- \mathfrak v} {2c }  = \frac{c-1}{2c} \mathfrak u +  \frac{c+1}{2c} \mathfrak v .
$
From direct calculations, we have
\[
\min_{z \in [0, t] }  ||x|-c(t-z)|+ z    =\begin{cases} t -\frac {|x| } c , \quad \hbox {if} \quad |x| \le c t,\\|x - c t| , \quad \hbox {if} \quad |x| \ge ct\end{cases}.
\]
Since $z, y \ge 0$, we have $ - \mathfrak u \le \mathfrak v \le  \mathfrak u $, and 
${t-\frac {|x|} c}\le  z+ |x-c (t-z ) | \le \mathfrak u \le \max_{z \in [0, t]} \{z + |x|+c (t - z)      \} \le |x| +ct$. Together with \Cref{lem:yxct0} and  $a>3$, we have 
\begin{equation}
    \label{eq:inequality for y x c t1-3}
\begin{aligned}
\mathcal I_1\lesssim  & \frac {c} {|x| } \int_{ \left| t-\frac {|x| } c \right|  }^{x + c t} \int_{- \mathfrak u}^ \mathfrak u  (1+ \mathfrak u)^{1-a} \pare{1+\left|  \frac{c-1}{2c} \mathfrak u +  \frac{c+1}{2c} \mathfrak v \right|}^{-1}  
\dd  \mathfrak v \dd \mathfrak u 
\lesssim   \frac {c} {|x| } \int_{ \left| t-\frac {|x| } c \right|  }^{x + c t} \int_{-2 \mathfrak u}^{2 \mathfrak u}  (1+ \mathfrak u)^{1-a} \pare{1+\av{\frac{c-1} {c+1} \mathfrak u + \mathfrak v }}^{-1}  \dd \mathfrak v \dd \mathfrak u 
\\
\lesssim  & \frac {c} {|x| } \int_{ \left| t-\frac {|x| } c \right|  }^{x + c t}  (1+ \mathfrak u)^{1-a}\log(3+ \mathfrak u) \dd \mathfrak u   
\lesssim  \frac {c} {|x| }  \frac {\log ( 3+ | t-\frac {|x| } c |  )}  {( 1+ | t-\frac {|x| } c | )^{a-3} }  \int_{ \left| t-\frac {|x| } c \right|  }^{x + c t}  
(1+ \mathfrak u)^{-2} \dd \mathfrak u
\\
\lesssim & \frac {c} {|x| }  \frac {\log ( 3+ | t-\frac {|x| } c |  )}  {(1+ | t-\frac {|x| } c |)^{a-2} } \frac {x + c t} {1 + x + c t}\lesssim c   \frac {\log ( 3+ | t-\frac {|x| } c |  )}  {(1+t+|x|)(1+ | t-\frac {|x| } c |)^{a-2} } ,
\end{aligned}
\end{equation}
\end{enumerate}
so \eqref{eq:inequality for y x c t1} is thus proved by combining \eqref{eq:inequality for y x c t1-1}, \eqref{eq:inequality for y x c t1-2}, \eqref{eq:inequality for y x c t1-3} together.
\end{proof}
\begin{lemma}\label{inequality for y x c t 2}
For any $t \ge 0, x, y \in \R^3, c \ge 1, a \ge 3$ we have that 
\begin{equation}
\label{eq:inequality for y x c t2}
\mathcal I_ 2 = \int_{|y-x| \le c t} \frac 1 {(1+  t   -\frac 1 c |y-x| +|y|  )^a } \frac 1 {|x-y|^2 } \dd y  \lesssim \frac 1 {(1+t+|x|) (1+ | t-\frac {|x|} c |)^{a-2} },
\end{equation}
where the constant is independent of $t, y, x, c$. 
\end{lemma}
\begin{proof}
First, for $\av{y-x}\le ct$, we have that 
\[
 t   -\frac 1 c |y-x| +|y| \ge t- \frac 1 c |y| -  \frac 1 c |x| +|y| \ge  t-  \frac 1 c |x|,\quad  t   -\frac 1 c |y-x| +|y|  \ge  |y| \ge |x| - |y-x| \ge |x| -ct .
\]
The two inequalities imply that
$1+ t   -\frac 1 c |y-x| +|y| \ge 1 + | t-\frac {|x|} c | $. So we only need to prove \eqref{eq:inequality for y x c t2} for $a =3$. We split the proof into several cases. 
\begin{enumerate}
\item 
{If $ct \le \frac 1 2 $}. For this case we have $|x|  + 1   \lesssim |y|+1$, then we have that 
\begin{equation}
\label{eq:inequality for y x c t2-1}
\mathcal I_2
\lesssim \frac 1 {(1+|x| )^3  } \int_{|y-x| \le 1}  \frac 1 {|x-y|^2 } \dd y  \lesssim \frac 1 {(1+t+|x|  )^3 }.
\end{equation}

    \item {If $|x| \le \frac 1 2$}. Then we have $|y-x| +1 \lesssim |y| +1 $, thus 
    \begin{equation}
\label{eq:inequality for y x c t2-2}
\begin{aligned}
\mathcal I_2
\lesssim &\int_{|y-x| \le c t} \frac 1 {(1+t +|y-x| )^3  } \frac 1 {|x-y|^2 } \dd y  
= \int_{|z| \le c t} \frac 1 {(1+t  +|z| )^3  } \frac 1 {|z|^2 } \dd z\\
\lesssim & \int_0^{ct}  \frac 1 {(1+t  +|z| )^3}  \dd|z| \lesssim \frac 1 {(1+t)^2 } \lesssim \frac 1 {(1+t+|x|  )^2 }.
\end{aligned}
\end{equation}
\item{If $|x | \le \frac 1 4 t$}. From Lemma \ref{lem:yxct1}, we have
\begin{equation}
\label{eq:inequality for y x c t2-3}
\begin{aligned}
\mathcal I_2
\lesssim &\int_{|y-x| \le c t} \frac 1 {(1+t +|y-x| )^3  } \frac 1 {|x-y|^2 } \dd y   \lesssim \frac 1 {(1+t+|x|  )^2 }.
\end{aligned}
\end{equation}
\item{If $|x| \ge ct$}. We choose $\mathpzc g(x,y)=\frac 1{(1+x+y)^3}, \mathpzc h(x)=\frac 1{x^2}$ in \eqref{eq:changevariable1}. Remind that $z'=t-
\mathpzc z$, we have $|x| \ge c z' $. So
\begin{equation}
\label{eq:inequality for y x c t2-4}
\begin{aligned}
\mathcal I_2= &c\frac {2 \pi} {|x|}  \int_{0}^{t} \int_{| |x| - c  z' |}^{|x| +c z'}  y (1+t-z'+ y)^{-3} \dd y \frac 1 { c z'} \dd z'\le c\frac {2 \pi} {|x|}  \int_{0}^{t} \int_{ |x| - c  z' }^{|x| +c z' } (1+t-z'+ y)^{-2}  \dd y \frac 1 { c z'}    \dd z'
\\
= &c\frac {4 \pi} {|x|}  \int_{0}^{t} \frac1 {(  1+t-z' +|x| -c z' )(1+t-z' +|x| + c z' )}      \dd z'
\\
 \lesssim& \frac1 { (1+|x| -c t)(1+t+ |x|  )  }\frac {c} {|x|}      \int_{0}^t   \dd z'    \lesssim  \frac 1 {(1+t+|x|) (1+ | t-\frac {|x|} c |)}.
\end{aligned}
\end{equation}

\item{If $\frac 1 {8} (1+t) \le |x| \le c t$}. Remind that $\frac 18\pare{1+t}\le ct$ if $ct>\frac 12$. 
choose $N=2$ in \Cref{lem:yxct2}, we have
\begin{equation}
\label{eq:inequality for y x c t2-5}
\begin{aligned}
\mathcal I_2 \le &c\frac {2 \pi} {|x|}  \int_{0}^{\frac {|x|} c } \int_{ |x| - c  z' }^{|x| +c z' } (1+t-z'+y)^{-2}  \dd  y \frac 1 { c z'}    \dd z' + c\frac {2 \pi} {|x|}  \int_{\frac {|x|} c }^t \int_{ c z' -|x| }^{|x| +c z' } (1+t-z'+ y)^{-2}  \dd y \frac 1 { c z'}    \dd z'
\\
= &4\pi\int_0^t\frac1 {1+t-z' +|x| + c z' }\pare{\frac {c\mathbbm{1}_{z'\in [0,\frac{|x|}c]}}{|x|\pare{1+t-z' +|x| - c z' }}+\frac{\mathbbm{1}_{z'\in [\frac{|x|}c, t]}} {z'\pare{1+t-z' +c z'-|x| }}}\dd z'\\
\lesssim  &\frac {c} {|x|}  \int_{0}^{\frac {|x|} c }    \frac1 {(  1+t -\frac {|x|}{c} )(1+t + |x|  ) }      \dd z' + \frac 1 { 1+t + |x| } \int_{\frac {|x|} c }^t          \frac1 {(  1+t-z'  + c z' -|x|  )  z'} \dd z'  .
\end{aligned}
\end{equation}
Thus for $c\ge 2$, we have
\begin{equation}
\label{eq:inequality for y x c t2-5extra}
\begin{aligned}
\mathcal I_2 
\lesssim  &  \frac1 {(  1+t -\frac {|x|}{c} )(1+t + |x|  ) }     + \frac 1 { 1+t + |x| } \int_{0 }^{  t  - \frac {|x|} c}    \frac  1 {(  1+t - \frac {|x|} {c} +\frac c 2 z  )  (z+\frac {|x|} {c} )} \dd z 
\\
\lesssim  &  \frac1 {(  1+t -\frac {|x|}{c} )(1+t + |x|  ) }     + \frac 1 { 1+t + |x| } \int_{0 }^{  t  - \frac {|x|} c}    \frac  1 {(  1+t - \frac {|x|} {c}  ) \frac {|x|} {c}   +c z^2  } \dd z  
\\
\lesssim &   \frac1 {(  1+t -\frac {|x|}{c} )(1+t + |x|  ) } +  \frac 1 { 1+t + |x| }  \frac 1 {  \sqrt{ (  1+t - \frac {|x|} {c}  )  |x|   }}\lesssim   \frac1 {(  1+t -\frac {|x|}{c} )(1+t + |x|  ) } .
\end{aligned}
\end{equation}
For $c\in [1,2]$, from $\int_{0 }^{  t  - \frac {|x|} c}    \frac  1 {(  1+t - \frac {|x|} {c}  )  (z+\frac {|x|} {c} )} \dd z =\frac 1{1+t-\frac{|x|}c}\log\pare{\frac{ct}{\av x}}\lesssim\frac 1{1+t-\frac{|x|}c}$, we still have
\begin{equation}
\label{eq:inequality for y x c t2-6}
\mathcal I_2 \lesssim   \frac1 {(  1+t -\frac {|x|}{c} )(1+t + |x|  ) } .
\end{equation}
\end{enumerate}
The proof is thus finished by gathering \eqref{eq:inequality for y x c t2-1} to  \eqref{eq:inequality for y x c t2-6} together. 
\end{proof}

%

\begin{lemma}\label{inequality for y x c t 3}
For any $t \ge 0,  c \ge 1, c t\ge 1, x, y \in \R^3$, we have that 
\begin{equation}
\label{eq:inequality for y x c t3}
\mathcal I_3 = \int_{1\le  |y-x| \le c t} \frac 1 {(1+  t   -\frac 1 c |y-x| +|y|  )^3 } \frac 1 {|x-y|^3 } \dd y  \lesssim \frac {\log(3+t +|x| ) } {(1+t+|x|)^2  (1+ | t-\frac {|x|} c |) }.
\end{equation}
\end{lemma}
\begin{proof}
We split into different cases of $x$.
\begin{enumerate}
    \item {If $|x| \le \frac 1 2$}. Then we have that $|y-x| +1 \lesssim |y| +1 $. We obtain from \Cref{eq:yxct3} that
    \begin{equation}
    \label{eq:inequality for y x c t3-1}
\begin{aligned}
\mathcal I_3\lesssim &\int_{1 \le |y-x| \le c t} \frac 1 {(1+t +|y-x| )^3  } \frac 1 {|x-y|^3 } \dd y  
= \int_{1\le |z| \le c t} \frac 1 {(1+t  +|z| )^3  }\frac 1  { |z|^3} \dd |z|
\\\lesssim & \int_1^{ct}  \frac {1}  {|z|(1+t  +|z| )^3}  \dd |z| 
\lesssim  \frac {\log(3+t) } {(1+t)^3 } \lesssim  \frac {\log(3+t) } {(1+t+|x| )^3 }.
\end{aligned}
\end{equation}
\item{If $|x | \le \frac 1 4 t$}.  From Lemma \ref{lem:yxct1}, we have
\begin{equation}
\label{eq:inequality for y x c t3-2}
\mathcal I_3 \lesssim \int_{|y-x| \le c t} \frac 1 {(1+t +|y-x| )^3  } \frac 1 {|x-y|^2 } \dd y   \lesssim  \frac {\log(3+t) } {(1+t+|x| )^3 }.
\end{equation}
\item{If $|x| \ge ct$}. Choose $\mathpzc g\pare{x,y}=\frac{1}{\pare{1+x+y}^3}, \mathpzc h\pare x=\frac 1{x^3}$ in Lemma \ref{basic change of variable on the sphere x y sigma improved version}, and remind that $|x| \ge c z' $, we have
\begin{equation}
\label{eq:inequality for y x c t3-3}
\begin{aligned}
\mathcal I_3 = &c\frac {2 \pi} {|x|}  \int_{\frac 1 c  }^{t} \int_{| |x| - c  z' |}^{|x| +c z' }  y(1+t-z'+ y)^{-3}\frac 1 { c^2 z'^2 }   \dd y \dd z'
\le c\frac {2 \pi} {|x|}  \int_{\frac 1 c }^{t} \int_{ |x| - c  z' }^{|x| +c z' } (1+t-z'+ y)^{-2}  \dd  y \frac 1 { c^2 z'^2 }    \dd z'
\\
= &c\frac {4 \pi} {|x|}  \int_{\frac 1 c }^{t} \frac1 {(  1+t-z' +|x| -c z' )(1+t-z' +|x| + c z' )}  \frac 1 { c z'}     \dd z'
\\
\lesssim & \frac {c} {|x|}  \frac1 {(  1+t+|x| )( 1 + |x| -ct)   }  \int_{\frac 1 c }^{t}   \frac 1 {cz'}     \dd z'
= \frac1 {(  1+t+|x| )( 1 + |x| -ct)   }   \frac  {\log (3+ct ) } {|x| }  \\  \lesssim &\frac {\log(3+t+|x| )} {(1+t+|x|)^2(| 1+t -\frac {|x|} c |    )  } .
\end{aligned}
\end{equation}

\item{If $\frac 1 {8} (1+t) \le |x| \le c t$}.  Similarly as \eqref{eq:inequality for y x c t2-5},
we have 
\begin{equation}
\label{eq:inequality for y x c t3-4extra}
\mathcal I_3
\lesssim  \frac {1} {|x|}  \int_{\frac 1 c }^{\frac {|x|} c }    \frac1 {(  1+t -\frac {|x|}{c} )(1+t + |x|  ) }   \frac 1 {z'}  \dd z' + \frac 1 { 1+t + |x| } \int_{\frac {|x|} c }^t          \frac1 {(  1+t-z'  + c z' -|x|  )  cz'^2 } \dd z'  .
\end{equation}If $c\ge 2$, choose $N=3$ in \Cref{lem:yxct2}, we have
\begin{equation}
\label{eq:inequality for y x c t3-4}
\begin{aligned}
\mathcal I_3
\lesssim  &\frac {1} {|x|}  \int_{\frac 1 c }^{\frac {|x|} c }    \frac1 {(  1+t -\frac {|x|}{c} )(1+t + |x|  ) }   \frac 1 z'   \dd z' + \frac 1 { 1+t + |x| } \int_{\frac {|x|} c }^t          \frac1 {(  1+t-w  + c z' -|x|  )  cz'^2 } \dd z'  
\\
\lesssim  &  \frac{ \log(3+|x| )}  {(  1+t -\frac {|x|}{c} )(1+t + |x|  )^2 }     + \frac 1 { 1+t + |x| } \int_{0 }^{  t  - \frac {|x|} c}    \frac  1 {(  1+t - \frac {|x|} {c}  ) \frac {|x|^2 } {c}   +c^2 z^3  } \dd z  
\\
\lesssim &   \frac{\log(3+|x| )} {(  1+t -\frac {|x|}{c} )(1+t + |x|  )^2 } +  \frac 1 { 1+t + |x| }  \frac 1 { ( (  1+t - \frac {|x|} {c}  )^2    |x|^4  )^{\frac 1 3}  }\lesssim   \frac{ \log(3+|x| ) }  {(  1+t -\frac {|x|}{c} )(1+t + |x|  )^2  } .
\end{aligned}
\end{equation}
If $c\in [1,2]$, from $\int_{\frac{|x|}c}^t\frac 1{cz'^2}\dd z'=\frac{ct-|x|}{ct|x|}\lesssim \frac{\log(3+|x|)}{1+t+|x|}$, we still have
\begin{equation}
\label{eq:inequality for y x c t3-5}
\mathcal I_3\lesssim   \frac{ \log(3+|x| ) }  {(  1+t -\frac {|x|}{c} )(1+t + |x|  )^2  } .
\end{equation}
\end{enumerate}
The proof is thus finished by combining \eqref{eq:inequality for y x c t3-1} to \eqref{eq:inequality for y x c t3-5} together. 
\end{proof}


	\begin{footnotesize}
		\bibliography{refs}
		\bibliographystyle{plain}
	\end{footnotesize}

\end{document}